\documentclass[reqno]{amsart}
\usepackage[a4paper, margin=1.3in]{geometry}
\usepackage[foot]{amsaddr}

\usepackage[hypertexnames=false]{hyperref}

\usepackage{amsmath}
\usepackage{mathrsfs}
\usepackage{amssymb}
\usepackage{amsthm}

\usepackage{appendix}

\usepackage{graphicx} 
\usepackage[rgb]{xcolor} 
\usepackage{mathtools}
\definecolor{Red}{cmyk}{0,1,1,0}

\definecolor{Blue}{cmyk}{1,1,0,0}

\usepackage{fourier}
\usepackage[T1]{fontenc}

\theoremstyle{plain}
\newtheorem{theorem}{Theorem}[section]
\newtheorem{corollary}[theorem]{Corollary}
\newtheorem{proposition}[theorem]{Proposition}
\newtheorem{lemma}[theorem]{Lemma}

\theoremstyle{definition}
\newtheorem{definition}[theorem]{Definition}
\newtheorem{remark}[theorem]{Remark}
\newtheorem{example}[theorem]{Example}

\newcommand{\N}{\mathbb{N}}
\newcommand{\R}{\mathbb{R}}

\def\.{\cdot}

\def\R{{\mathbb R}}

\def\N{{\mathbb N}}

\def\S^1{\mathbb{S}^1}

\def\S{{\cal S}}
\def\.{\cdot}
\def\diam{\text{diam}}
\def\({\left(}
\def\){\right)}

\newcommand{\supp}{\mbox{supp}}

\begin{document}

\title[Spectral triples and Dixmier Trace Representations]
{Spectral Triples and Dixmier Trace Representations of Gibbs Measures: theory and examples}

\author[L. Cioletti]{Leandro Cioletti}
 \address{\rm L. Cioletti, Departamento de Matemática, Universidade de Bras\'ilia, Bras\'ilia, Brazil, 
\href{https://orcid.org/0000-0002-8131-2043}{Orcid 0000-0002-8131-2043}, \href{mailto:cioletti@mat.unb.br}{cioletti@mat.unb.br}}

\author[L. Y. Hataishi]{Lucas Y. Hataishi}
\address{\rm L. Y. Hataishi, Department of Mathematics, University of Oslo, Oslo, Norway, \href{mailto:lucasyh@math.uio.no}{lucasyh@math.uio.no}}

\author[A. O. Lopes]{Artur O. Lopes}
\address{\rm A. O. Lopes, IME, Universidade Federal do Rio Grande do Sul, Porto Alegre, Brazil, \href{https://orcid.org/0000-0003-2040-4603}{Orcid 0000-0003-2040-4603}, \href{mailto:artur.lopes@gmail.com}{artur.lopes@gmail.com}}

\author[M. Stadlbauer]{Manuel Stadlbauer}
\address{\rm M. Stadlbauer, Instituto de Matemática, Universidade Federal do Rio de Janeiro, Rio de Janeiro, Brazil, \href{https://orcid.org/0000-0003-2537-9128}{Orcid 0000-0003-2537-9128},\href{mailto:manuel@im.ufrj.br}{manuel@im.ufrj.br}}



\date{}

\begin{abstract}
In this paper we study spectral triples and non-commutative expectations 
associated to expanding and weakly expanding maps. 
In order to do so, we generalize the Perron-Frobenius-Ruelle 
theorem and obtain a polynomial decay of the operator, 
which allows to prove differentiability of a dynamically 
defined $\zeta$-function at its critical parameter. 
We then generalize Sharp's construction of spectral triples 
to this setting and provide criteria when the associated spectral 
metric is non-degenerate and when the non-commutative expectation 
of the spectral triple is colinear to the integration with respect 
to the associated equilibrium state from thermodynamic formalism.  
Due to our general setting, we are able to simultaneously analyse 
expanding maps on manifolds or connected fractals, subshifts of 
finite type as well as the Dyson model from statistical physics,
which underlines the unifying character of noncommutative geometry. 
Furthermore, we derive an explicit representation of 
the $\zeta$-function associated to a particular class of 
pathological continuous potentials, giving rise to examples 
where the representation as a non-commutative expectation 
via the associated zeta function holds, and others where it does not hold. 
\end{abstract}

\maketitle

\section{Introduction}
This paper aims to provide a further contribution to the relation between thermodynamic formalism and ideas from noncommutative geometry. More specifically, we are interested in a noncommutative representation of Gibbs measures and the underlying mechanisms from thermodynamic formalism.

The starting point of our investigations are the works
by Richard Sharp \cite{MR2990125,MR3485414}.
In there, Sharp shows, among other things,
that Gibbs Measures, appearing in the context
of Thermodynamic Formalism and conformal graph directed Markov systems,
can be recovered from suitable spectral triples and Dixmier traces.
These are important mathematical objects in noncommutative geometry, where certain
geo\-me\-tric spaces are analyzed by using operator and $C^{*}$-algebras (see \cite{MR823176,MR1303779}.
\mbox{Although} this theory has a geometric origin,
there has been considerable interest in finding
examples where the $C^{*}$-algebra is the space of continuous functions on
suitable subspaces of infinite cartesian products, sometimes called
Cantor sets. The first examples were
given by Connes in \cite{MR823176,MR1007407,MR1303779}
and in the last decades, this subject has attracted some attention, see
for example, 
\cite{MR2261610,MR3381484,MR3438329,MR3115301,MR3084488,
MR2511637,MR2990125,MR3485414,MR3054307}
and references therein.
These works concern not just Cantor sets,
but also hyperbolic dynamics, directed Markov systems,
IFSs and symbolic dynamics.

Connes showed that spectral triples can be used to define a
pseudo-metric on the state space of the associated $C^{*}$-algebra.
The way this pseudo-metric is defined is analogous to the Wasserstein distance (or Monge-Kantorovitch metric)
and some general conditions ensuring that this pseudo-metric is indeed a metric
were obtained independently by Pavlovi\'c and Rieffel (see
\cite{MR1663810,MR1647515}).
Kesseb\"ohmer and Samuel studied metric aspects of this
theory in \cite{MR3084488}, following \cite{MR2261610}, but in the context of Gibbs measures.
They proved that Connes' pseudo-metric is actually a metric that induces
the weak-${*}$-topology on the space of Borel probability measures
on certain subshifts of finite type.
Several other geometric aspects of the spectral triples associated
to Gibbs measure of H\"older potentials are studied in \cite{MR3084488}.
For example, they compute the dimension of their spectral triples,
and show that the noncommutative volume constant is given by the inverse of
Kolmogorov-Sinai entropy of the Gibbs measure.
This fact was established by showing that the
Gibbs measure associated to a H\"older potential
admits a Dixmier trace representation, which leads to an identity similar to \eqref{eq:noncommutative-integral}, but
without the factor two. One important ingredient in their proof is the existence of a Haar basis for the $ {L}^2$ space of a Gibbs measure.

In here, we focus on these questions associated to a class of weakly expanding maps, which allows to simultaneously study subshifts of finite type, expanding local diffeomorphisms, some connected fractal sets like the Sierpiński gasket or weakly expanding maps like the Manneville-Pomeau family at the same time. This class is defined as
follows and is a weakly expansive version of the class of expanding maps defined in \cite{Ruelle--The-Thermodynamic-Formalism-For--CMP1989}. 

\begin{definition}[Ruelle expansive] 
Assume that $(\Omega,d)$ is a compact metric space of diameter 1, that $r>0$ and that $(c_n :[0,1] \to [0,\infty))$ is a sequence of functions such that $c_n(0)=0$, $t \mapsto c_n(t)$ is  
strictly increasing for all $n \in \N$,  
$n \mapsto c_n(t)$ is decreasing  and 
$\lim_{n \to \infty} c_n(t) =0$ for all $t \in [0,1]$. 
We then refer to $T: X \to X$ as \emph{$(c,r)$-Ruelle expansive} 
if for all $n \in \N$, $x, {y}, \tilde{x} \in \Omega$  
with $d(x, {y})<r$ and  {\text{$T^n(\tilde{x})=x$}},
there exists a unique $\tilde{y}\in \Omega$  with $T^n(\tilde{y})={y}$  and 
\[ d(T^k(\tilde{x}), T^k(\tilde{y})) \leq  c_{n-k} (d(x,y)) \quad \forall k =0,1,\ldots ,n .\]
\end{definition}
{
A standard example for  elements in this class are open maps which uniformly expand distances by $\lambda^{-1}$ for $\lambda \in (0,1)$ as the unicity of the preimage is a consequence of the openness whereas the uniform expansion guarantees that $c_n(t) \ll t \lambda^{n}$. This class of examples was introduced in \cite{Ruelle--The-Thermodynamic-Formalism-For--CMP1989} and is referred to as \emph{Ruelle expanding} in the literature (see Example \ref{ex:ruelle-expanding} for the definition given in \cite{Ruelle--The-Thermodynamic-Formalism-For--CMP1989}).} A further example is the class of Manneville-Pomeau maps as introduced in \cite{Liverani-Saussol-Vaienti--A-Probabilistic-Approach-To--ETDS1999}. In this case, it is shown in Example \ref{ex:Manneville} below that $c_n(t) \ll  (n+t^{-\delta})^{-1/\delta}$.

For a given function $J : \Omega \to (0,1)$, we are now interested in establishing a differential calculus on $\Omega$ with respect to the \emph{potential} $\log J$ by following ideas from noncommutative geometry. This is of special relevance as, among other things, the theory of noncommutative geometry aims to unify the differential calculus for smooth and fractal spaces in terms of their $C^\ast$-algebras. This unifying object, referred to as \emph{spectral triple}, is defined as follows.
 
\begin{definition}\label{def-spec-triples}
A spectral triple is an ordered triple $(A,H,D)$, where
\begin{enumerate}
\item
 $H$ is a Hilbert space;
\item
 $A$ is a $C^{*}$-algebra and there exists a faithful representation $a \to L_a$ from $A$ to the bounded  operators of $H$;
\item
 $ D$ is an essentially self-adjoint  unbounded linear operator on $H$ with
compact resolvent such that $\{a\in A: \|[D,L_a]\|<+\infty \}$ is dense in $A$. In here, 
$[D, L_a]$ refers to the commutator operator and  $D$
is called \emph{Dirac operator}.
\end{enumerate}
\end{definition}
Observe that a spectral triple a priori is not of dynamical nature as $D$ is a differential operator and $[D,L_a]$ might be seen as a  noncommutative version of the derivative of $a$ (for a list of examples, see \cite{Samuel--A-Commutative-Noncommutative-Fractal--2010} and Remark \ref{remark:s_0} below). Furthermore, each spectral triple gives rise to a pseudo-metric on the set of states of $A$, known as \emph{spectral metric} or \emph{Connes' metric} (cf. Section \ref{subsec:spectral-metric}). 

On the other hand, the relation of the Dirac operator to thermodynamical formalism is revealed by the representation of the Dixmier trace of $L_a|D|^{-1}$ as the integral with respect to the equilibrium state of $\log J$.
{
Recall that in order to construct the  Dixmier trace  $\textrm{Tr}_{\omega}$, one has to choose a positive, normalized linear functional $\omega$ defined on the Banach space of bounded sequences with some additional conditions (see \cite{Dixmier--Existence-De-Traces-Non--CRASPSB-1966}). Then, the Dixmier trace is defined by
\[  \textrm{Tr}_{\omega}(B) := \omega\left( \left(  \tfrac{1}{\log (n)}  \textstyle\sum_{k=1}^n \beta_k \right)_n \right),\]
where $\beta_k$ are the eigenvalues with multiplicities of the positive operator $|B| = \sqrt{B^* \, B}$ in descending order. However, if $ \lim_{n \to \infty} \frac{1}{\log (n)}  \sum_{k=1}^n \beta_k$ exists, then $B$ is referred to as \emph{measurable},
\[  \textrm{Tr}_{\omega}(B) = \lim_{n \to \infty} \frac{1}{\log (n)}  \sum_{k=1}^n \beta_k,\]
and $\textrm{Tr}_{\omega}(B) $ is  called the \emph{noncommutative  expectation} of $B$.
}

We now present our main results and relate them to the state of art. In Section \ref{sec-preliminares}, we provide the necessary results on the asymptotic behaviour of the action of the  
Ruelle operator, defined by  
\[
\mathcal{L}_a (f) (x)  
= 
\sum_{y\in T^{-1}(x) } 
e^{a(y) } f(y),
\] 
on a suitable function space. Namely, we obtain polynomial 
contraction rates of $\mathcal{L}_a^n$ towards a limiting 
distribution if either $c_n(t) \ll \lambda^{-n}(t)$ and $a$ 
is of weak regularity or $c_n(t)$ decays polynomially and $a$ is H\"older. 
As these results are new, without doubt of independent 
interest and refine recent results by Kloeckner in \cite{Klo}, 
we now state a  simplified version of the main result of this first part.        

In order to do so, we have to introduce a version of weak, 
but uniform regularity as follows. For $t> 0$, set
$\omega_{\alpha, \beta \textrm{log}} (t) := t^\alpha (\log t_0  -\log t)^{- \beta}$, where $t_0 > 0$ 
is chosen such that
$\omega_{\alpha, \beta \textrm{log}}$ is concave. 
We then refer to $f: \Omega  \to \R$  as 
$ \omega_{\alpha, \beta \textrm{log}}$-Hölder continuous 
if there is $\textrm{Höl}_{\omega_{\alpha, \beta \textrm{log}}}(f)> 0$ with
\[
|f (x) - f (y)| \leq \textrm{Höl}_{\omega_{\alpha, \beta \textrm{log}}}(f)   \omega_{\alpha, \beta \textrm{log}}(d(x,y))
\]
In particular, $ \omega_{\alpha,0 \textrm{log}}$-Hölder continuity 
coincides with the classical $\alpha$-Hölder continuity, 
whereas $\omega_{0, \beta \textrm{log}}$-Hölder continuous 
functions are less regular.

{
As it will turn out below, it is possible to handle pairs $(T,a)$ in a unified way, provided that $T$ is either Ruelle expanding and $a$ is  $\omega_{0, \beta \textrm{log}}$-Hölder continuous or $T$ is Ruelle expansive and $a$ is Hölder continuous. That is, we consider the following two cases.
\begin{description}
 \item[Case 1] $T$ is Ruelle expanding and $a$ is $\omega_{0, \beta \textrm{log}}$-Hölder continuous for some $\beta > 1$. In this situation, $\gamma := \beta -1$, $\omega^\dagger := \omega_{0, \beta \textrm{log}}$ and, for $n \in \mathbb{N}$ and $t \in (0,1]$,
$\omega_n(t):= (|\log t| + n)^{-\gamma}$.
 \item[Case 2] $T$ is  Ruelle expansive with contraction rate $(t^{-1/\beta} + n)^{-\beta}$ for some $\beta > 1$ and $a$ is  $ \omega_{\alpha,0 \textrm{log}}$-Hölder continuous for some $\alpha > 1/\beta$. In this situation, $\gamma := \alpha \beta -1$,  $\omega^\dagger := \omega_{\alpha - 1/\beta, 0\textrm{log}}$ and,  for $n \in \mathbb{N}$ and $t \in (0,1]$, $\omega_n(t):= (t^{-1/\beta} + n)^{-\gamma}$.
\end{description}
Observe that, in both cases, provided that $n$ is sufficiently large, $\omega_n$ is concave and therefore gives rise to the notion of $\omega_n$-Hölder continuity and, in particular, to the norm $\|f\|_{\omega_n} := \|f\|_{\infty} + \textrm{Höl}_{\omega_n}(f)$. In both cases, we show below the following contraction result (for $\gamma > 0 $, see  Corollary \ref{cor:decay-to-invariante-measure} and for $\gamma> 1$, see Corollary \ref{cor:decay-of-Hoelder-coefficients}), which is an immediate consequence of Theorem \ref{theo:main-contraction-result}.

\smallskip

\begin{theorem}
Assume that $T$ is topologically mixing, that $\mathcal{L}_a(\mathbf{1})=\mathbf{1}$ and that we are either in Case 1 or Case 2.
\begin{enumerate}
  \item  If $\gamma>0$, then there exists a probability measure $\mu$ and $C> 0$ such that for any $\omega^\dagger$-Hölder continuous function $f: \Omega \to \R$ and $n \in \N$, $ \| \mathcal{L}^n_a(f) - \mu(f) \|_\infty \leq C \textrm{Höl}_{\omega^\dagger}(f) n^{-\gamma}$.
  \item  If $\gamma>1$, then there exists a probability measure $\mu$ and $C> 0$ such that for any $\omega^\dagger$-Hölder continuous function  $f: \Omega \to \R$ and $n \in \N$,  $ \| \mathcal{L}^n_a(f) - \mu(f) \|_{\omega_n} \leq C \textrm{Höl}_{\omega^\dagger}(f)$.
\end{enumerate}
\end{theorem}
}

It is worth noting here that the proof is based on ideas from optimal transport
(cf. Theorem \ref{theo:main-contraction-result}) and allows to obtain 
a complete description on the asymptotic behaviour of 
$\mathcal{L}^n_a$, even if  $\mathcal{L}_a(\mathbf{1}) \neq \mathbf{1}$
 (Corollary \ref{cor:decay-to-conformal-measure}).
In particular, we obtain an extension of the main 
contraction result in \cite{Klo} to Ruelle expanding, Ruelle expansive maps
and Ruelle operators with $\mathcal{L}_a(\mathbf{1}) \neq \mathbf{1}$. 
We also would like to point out that the method of proof is different 
from the one in \cite{Klo} as it does not make use of half-times. 
{ However, we have to admit that the we do not yet fully understand the transition at $\gamma =1$. However, when applying our approach to the Manneville-Pomeau family, this corresponds to the parameter $\frac{1}{2}$, where it is known from the work of Gouëzel (\cite{Gouezel--Sharp-Polynomial-Estimates-For--IJM2004}) that there is a transition of the decay rate at this parameter.
Furthermore, the approach in here provides direct access to
stability under perturbation of the potential: the equilibrium state, the conformal measure,
the eigenfunction and the leading eigenvalue vary Hölder 
continuously as functions of $a$. That is, we obtain the following from Corollary\ref{cor:uniqueness} and Theorem \ref{theo:continuity}.

\begin{theorem}
Assume that $T$ is topologically mixing, that  $\gamma>0$ and that we are either in Case 1 or 2.
\begin{enumerate}
 \item  There exists a unique continuous and strictly positive function $h_a$ such that
 $\mathcal{L}_a(h_a)=\lambda^{(1)}_a h_a$ for some $\lambda^{(1)}_a > 0$. Moreover, $h_a$ is $\omega^\dagger$-Hölder continuous.
 \item There exists a a unique probability measure $m_a$ with
 $\int \mathcal{L}_a(f) dm_a = \lambda_a^{(2)} m_a(f)$ for any continuous function $f:\Omega \to \R$ and some  $\lambda^{(2)}_a > 0$.
 \item We have $\lambda^{(1)}_a =\lambda^{(2)}_a$, for the constants given by (1) and (2).
 \item For $M > 0$, there exists $C>0$ such that the following holds. If $a,b$ with $\textrm{Höl}_{(\ast)}(a), \textrm{Höl}_{(\ast)}(b)< M$,  where $(\ast)$ stands for either $\omega_{0, \beta \textrm{log}}$ in Case 1 or $\omega_{\alpha, 0\textrm{log}}$ in Case 2, respectively, then, with $W_{\omega^\dagger}$ referring to the Wasserstein metric associated with $\omega^\dagger$,
 \[ W_{\omega^\dagger}(m_a,m_b) , \|h_a - h_b\|_\infty, |\lambda_a - \lambda_b| \leq C \|a-b\|_\infty^{\frac{\gamma}{\gamma + 1}}. \]
%
\end{enumerate}
\end{theorem}
}

In Section \ref{mainsec-1}, we adapt the approach by Sharp in \cite{MR2990125} 
and construct a spectral triple for Ruelle expanding maps and a 
given continuous function $J$ (see Proposition \ref{prop:spectral-triple}). 
By regarding a specific example on the torus, 
it then becomes obvious that this spectral triple is a classical 
directional derivative and that its spectral metric is 
degenerated (see Remark \ref{remark:s_0}). 
However, by considering a direct sum of these spectral 
triples and provided that the potential $J$ reflects the 
metric structure of $\Omega$, we then obtain in 
Theorems \ref{theo:spectral-metric} and \ref{theo:spectral-metric-sft} 
that the spectral metric associated to this new spectral triple 
in fact is Lipschitz equivalent to the Wasserstein metric 
on probability measures on $\Omega$. { We would like to point out that the approach in here is different
to the original motivation by Connes, which is based on the Dirac operator on smooth spinors (see \cite{Lord-Rennie-Varilly--Riemannian-Manifolds-In-Noncommutative--JGP-2012,Varilly--An-Introduction-To-Noncommutative---2006}). However, we will only make use of Connes' abstract definition of a spectral triple in order to be able to handle spaces without differentiable structure and discuss these in Remark \ref{remark:analysis on fractals} for the Sierpiński gasket.}

It is worth noting that, as we are able to cover subshifts of 
finite type and weakly regular potential functions, our approach 
is applicable to any fractal which is the attractor of a 
(uniformly contracting) graph directed conformal Markov system 
with respect to a finite alphabet 
(see \cite{Mauldin-Urbanski--Graph-Directed-Markov-Systems--2003}), 
e.g., the classical Cantor set or the limit set  of  a 
Schottky group acting on hyperbolic $n$-space. 
In particular, Theorem \ref{theo:spectral-metric-sft} 
implies that the spectral metric is equivalent to the 
Euclidean metric, which generalizes, e.g., 
Theorem 4.5 in \cite{Guido-Isola--Dimensions-And-Singular-Traces--JFA2003}.

On the other hand, if the fractal under consideration is 
not totally disconnected, a suitable interplay between 
the potential and the inverse branches of $T^k$ are sufficient 
to obtain the analogue statement in this case 
(Theorem \ref{theo:spectral-metric}). We then show how to apply 
the result to the Sierpiński gasket (see Subsection \ref{subsec-sierpinski}).
This application is of particular interest as it provides 
a construction of a spectral triple, which does not makes 
use of infinitesimal curves like in \cite{Christensen-Ivan-Lapidus--Dirac-Operators-And-Spectral--AM2008} 
or the removed parts, the \emph{lacunas}, 
of the gasket as in \cite{Cipriani-Guido-Isola--Spectral-Triples-For-The--JFA2014} (for more details on the analysis on the Sierpiński Gasket, see Remark \ref{remark:analysis on fractals}).

Thereafter, in Section \ref{sec-zeta-dixmier}, 
we then discuss the relation of these operators to 
thermodynamic formalism through the Diximer trace $\textrm{Tr}_{\omega}$ of $L_a|D|^{-1}$.
Here , we want to point out that there are several constructions of spectral triples in the literature associated to iterated functions systems based upon the prior definition of a $C^\ast$-algebra (\cite{MR2261610,Christensen-Ivan-Lapidus--Dirac-Operators-And-Spectral--AM2008,Samuel--A-Commutative-Noncommutative-Fractal--2010,MR3084488}) as well as the direct construction by Sharp (\cite{MR2990125}) for subshifts of finite type. Moreover, independent of the construction of the spectral triple, 
it was shown by several authors (see 
\cite{Christensen-Ivan-Lapidus--Dirac-Operators-And-Spectral--AM2008,Samuel--A-Commutative-Noncommutative-Fractal--2010,MR2990125,MR3084488}) that
$L_a|D|^{-1}$ is measurable and that
the noncommutative expectation satisfies
\begin{equation} \label{eq:noncommutative-integral} 
\textrm{Tr}_{\omega}(L_a|D|^{-1}) =  \frac{1}{\mathfrak{h}(\mu)}\,\,\int_{\Omega}a\, d\mu, 
\end{equation}
where $\mu$ refers to the equilibrium state associated to the Hölder continuous potential $\log J$ and $\mathfrak{h}(\mu)$ to its the Kolmogorov-Sinai entropy. In here, our results in Section \ref{sec-preliminares} allow us to push these results to spectral triples associated to Ruelle expansive maps and potentials of lower regularity. That is, we show in Theorem \ref{theo:dixmier-for-Ruelle-expansive-and-expanding} that, up to a factor the identity in \eqref{eq:noncommutative-integral}
holds for either Ruelle expanding maps and potentials of lower regularity or Ruelle expansive maps and Hölder potentials.
In fact, we show more. Namely, we only require that $\gamma>0$ and, in particular, that the underlying potential is not necessarily normalized. However, this requires to add additional terms (see Theorem \ref{def-spec-triples}).       
The mechanism behind the proof of this identity is an application of a version of the Hardy-Littlewood Tauberian theorem (cf. Th. \ref{theo:Tauber}). Namely, Theorem \ref{theo:dixmier-for-Ruelle-expansive-and-expanding} becomes a corollary of the differentiability of a dynamically defined $\zeta$-function at a critical parameter $\rho$ (see Lemma \ref{lem:differentiability}). In other words, one has to show that 
\begin{equation} \nonumber \label{kw}
\lim_{s\to \rho^{+}}(s-\rho)\,\,\sum_{k=0}^\infty \mathcal{L}_{sa}^k (f) (x)\,
=  
\frac{h_{\rho a}(x)}{\mathfrak{h}(\mu_{\rho a})}\,\,\int_{\Omega}f\, dm_{\rho a},
\end{equation}
which is equivalent in showing that the conformal measure $m_s$ associated to $J^s$ varies continuously at $s=\rho$. Among other things, the lack of results in this generality in the literature lead us to investigate the asymptotical behaviour of $\mathcal{L}^n_a$ in Section \ref{sec-preliminares} as the iterates of its dual provide access to the regularity of $\mu_s$ as a function of $s$.   

In the subsequent section, we then present applications, examples and counterexamples. We begin with a brief discussion of the relation between the Dixmier trace representation, the construction of spectral triples and the topological dimension of $\Omega$. In particular, it turns out that in the totally disconnected case, it is possible to construct spectral triples, whose spectral metric is a metric and which comes with a Dixmier trace representation with respect to the same exponent. However, in case of a connected space, this seems to be impossible. After that, we consider
a specific potential function $\Phi$ associated with the  Dyson model of ferromagnetism (cf. \cite{MR0436850,MR0295719}). It turns out that $\Phi$
is a natural example for an application of the above contraction result, but also shows a very different behaviour than the one known from conformal dynamics as its pressure function is strictly positive and increasing (cf. Lemma \ref{lem:dyson-pressure-is-increasing}). In particular, in order to obtain a spectral triple with a thermodynamic representation of the Dixmier trace, it is necessary to consider potentials of the form $\Phi - t$, for $t > \max \Phi$ (cf. Proposition \ref{prop:dixmier-triple}). After that, we present in Section \ref{exampand} examples within the Walters family of potentials defined on $\Omega$ (see \cite{MR2342978}) whose dynamical $\zeta$-functions are differentiable or non-differentiable at the critical parameter. We are also able to present explicit expressions for the associated zeta functions.

In Section \ref{sec-top-markov-chains} 
we then show how to transfer these results to topological Markov chains and Hölder continuous potentials and give some remarks on the difficulty in the context of uncountable alphabets.
In the Appendices, we present explicit computations for the zeta function of potentials in Walters' family (Appendix \ref{sec112}) as well as a short discussion of the construction of a spectral triple in \cite{MR3084488} in Appendix \ref{Haar} and an explicit representation for the equilibrium states (Appendix \ref{expWal}). The latter is of interest as this allows to obtain an explicit representations of the Dirac operator defined in  \cite{MR3084488}.

\section{Asymptotics of  Ruelle's operator}
\label{sec-preliminares} 
In this section we describe the asymptotic behaviour of Ruelle's operator and derive a stability result which will turn out to be essential for the representation of the Dixmier trace as an integral. The main tool for revealing the asymptotics of the Ruelle operator is the construction of a coupling which depends on the interplay between the contraction properties of $T$ with the regularity of the potential function. As mentioned above, we are also interested in Ruelle expansive maps, which have weaker contraction properties and whose definition is repeated for convenience.   

\begin{definition}[Ruelle expansive] 
Assume that $(\Omega,d)$ is a compact metric space of diameter 1, that $r>0$ and that $(c_n :[0,1] \to [0,\infty))$ is a sequence if functions such that $c_n(0)=0$, $t \mapsto c_n(t)$ is  
strictly increasing for all $n \in \N$,  
$n \mapsto c_n(t)$ is decreasing  and $\lim_{n \to \infty} c_n(t) =0$   
for all $t \in [0,1]$. We then refer the map 
$T: X \to X$ as \emph{$(c,r)$-Ruelle expansive} 
if for all $n \in \N$, $x, {y}, \tilde{x} \in \Omega$  
with $d(x, {y})<r$ and  {\text{$T^n(\tilde{x})=x$}},
there exists a unique $\tilde{y}\in \Omega$  
with $T^n(\tilde{y})={y}$  and 
\[ d(T^k(\tilde{x}), T^k(\tilde{y})) \leq  c_{n-k} (d(x,y)) \quad \forall k =0,1,\ldots ,n .\]
\end{definition}
As a first consequence of the definition, observe that $T^n$ is a local homeomorphism with inverse 
\[ \tau_{\tilde{x}} : B(x,r) \to X, y \mapsto \tilde{y},\]
where $T^n (\tilde{x}) = x$, $\tilde{y}$ is given by expansiveness and $B(x,r)$ is the open ball of radius $r$ with center $x$. Furthermore,  $d(\tau_{\tilde{x}}(y),\tau_{\tilde{x}}(y')) \leq c_n(d(y,y'))$.
 
We now introduce some notation. We will write $a \asymp b $  and $a \ll b$ whenever there exists $C\leq 1$ such that $C^{-1} b \leq a \leq Cb$ and $a \leq Cb$, respectively. Moreover, we will refer to $C$ as the \emph{implicit constant} in the estimate.  

{
\begin{example} \label{ex:ruelle-expanding}
We now recall the definition of expanding maps from \cite{Ruelle--The-Thermodynamic-Formalism-For--CMP1989}. That is, for a compact metric space $(\Omega,d)$,
the map $T: \Omega \to \Omega$ is \emph{$(\lambda,r)$-Ruelle-expanding} if there exists $\lambda \in (0,1)$ and $r> $ such that for any $x, {y}, \tilde{x} \in \Omega$  with $d(x, {y})<r$ and $T(\tilde{x})=x$, there exists a unique $\tilde{y}\in \Omega$  with $T(\tilde{y})={y}$ and $d(\tilde{x}, \tilde{y})<\lambda r$. Furthermore,  $ d(\tilde{x}, \tilde{y}) \leq  \lambda d(x,y) $.

We now remark that it easily can be shown that $T^n$ is Ruelle-expanding with parameters $(\lambda^n,r)$.  Or in other words, $T$ is Ruelle expansive with contraction rate $c_n(t) = t \lambda^n$. The relevance of the class of Ruelle expanding maps stems from the simple fact that subshifts of finite type as well as expanding, local diffeomorphisms on closed manifolds are in this class, and that the preimage structure, in general, allows to avoid the construction of Markov partitions. For an example defined on the Sierpiński gasket, see Section \ref{subsec-sierpinski}.
\end{example}
}

\begin{example} \label{ex:Manneville}
On the other hand, the family of Pomeau-Manneville maps as defined in \cite{Liverani-Saussol-Vaienti--A-Probabilistic-Approach-To--ETDS1999} provides examples of Ruelle expansive maps which are not Ruelle expanding. Recall that this family is defined by, for  
$0 < \delta < 1$,  
\[ T: [0,1] \to [0,1], 
x \mapsto 
\begin{cases}
 x(1 + (2x)^\delta) &: x \in [0,1/2]\\
 2x -1 &: x\in (1/2,1]
\end{cases}
\]
In order to show that $T$ is Ruelle expansive, note that for $x\in [1/2,1]$, the $n$-th inverse branch towards the neutral fixed point in $0$ satisfies $x_n \asymp  n^{-1/\delta}$ (see, e.g.  \cite{Sarig--Subexponential-Decay-Of-Correlations--IM2002}). Therefore, if $x \in (0,1/2)$ is of distance of order $k^{-1/\delta}$, then its $n$-th inverse branch towards $0$ is of order $(n+k)^{-1/\delta}$. By identifying $0$ and $1$, one then obtains a continuous map on the circle which is expansive with respect to $c_n(t) \asymp (n+ t^{-\delta})^{-1/\delta}$.
\end{example}

 The next ingredient is a control of the regularity of the potential through its \emph{modulus of continuity} as defined in \cite{Klo}. A function $\omega : [0,1] \to [0,\infty ) $ is referred to as  a modulus of continuity if $\omega$ is a continuous, increasing and  concave function such that $\omega(0)  = 0$. One then refers to  $ f:  \Omega \to \R$ as $\omega$-Hölder if there is $C_f> 0$ with  
\[
|f (x) - f (y)| \leq  C_f \omega(d(x, y)) \quad \forall x, y \in \Omega. 
\]
For example, if  $\omega(t) := t^\alpha$ for $\alpha \in (0,1]$,
then the $\omega$-Hölder property corresponds to the classical  $\alpha$-Hölder property. However, as we are also interested in potentials of lower regularity, we will consider functions with a modulus of continuity of the form, as introduced in \cite{Klo}, 
\[ \omega_{\alpha, \beta \textrm{log}} (t) = \frac{C t^\alpha}{(\log t_0 - \log t)^\beta }, \]
for $C,\alpha, \beta > 0$ and $t_0$ sufficiently large such that $\omega_{\alpha, \beta \textrm{log}}$ is concave on $[0,1]$. We now analyse the interplay 
between the contraction and  $\omega_{\alpha, \beta \textrm{log}}$ with respect to the Birkhoff sum
\[a_n:= a + a\circ T + \cdots + a\circ T^{n-1} \] 
for $a : \Omega \to \R$ and $n \in \N$. As we have to show that the thermodynamical quantities are continuous with respect to perturbations of $a$, we now take special care of constants which depend on $a$.    
If $T$ is Ruelle expanding (that is, $c_n(t)  \leq D s^n t$ for some $s\in (0,1)$) and $a$ is $\omega_{0,\beta \textrm{log}}$-Hölder with $\beta > 1$, then, for $x,y$ with $d(x,y)< r$ and $\tilde x \in  T^{-n}(\{x\})$ and $\tilde y := \tau_{\tilde x}(y)$,  
\begin{align}
\label{eq:distortion-expanding}
|a_n (\tilde x) - a_n (\tilde y)| 
& \leq C_a \sum_{j=1}^{n} \omega_{0,\beta \textrm{log}} (c_j(d(x,y)))   
\\ 
\nonumber
& \leq C_a \sum_{j=1}^{n} \frac{C}{(\log (t_0/D)  + |\log d(x,y))| + j |\log s |  )^\beta}  \\
\nonumber
& \leq \frac{C_a}{\beta |\log s|} \frac{C}{(\log (t_0/D) + |\log s | - \log d(x,y))    )^{\beta -1}  }  \\
\nonumber & \ll {C_a}  \omega_{0,(\beta -1) \textrm{log}} ( d(x,y)),  
\end{align}
where the implicit constant in $\ll$ only depends on $D$ and $s$. On the other hand, if $c_n(t) \leq D t(1 + n t^{1/\beta})^{-\beta}$ and $a$ is $\omega_{\alpha, 0 \textrm{log}}$-Hölder for some $\alpha,\beta > 0$ with $\alpha > 1/\beta$, then 
\begin{align}
\nonumber |a_n (\tilde x) - a_n (\tilde y)| &
\leq \sum_{j=1}^{n} C_a D^\alpha d(x,y)^\alpha ( 1 + j (d(x,y))^{1/\beta} )^{-\alpha \beta} \\
\label{eq:distortion-expansive}
& \leq \frac{ C_a D^\alpha}{\alpha \beta -1} d(x,y)^{\alpha - 1/\beta} \ll C_a \omega_{\alpha - 1/\beta, 0 \textrm{log}} (d(x,y)),
\end{align}
where the implicit constant in $\ll$ only depends on $D$, $\alpha$ and $\beta$. In particular, it follows under both hypotheses that, for 
$x,y \in \Omega$ with  $d(x,y) < r$, 
\begin{align*}
\left| \frac{\mathcal{L}^n_a (1)(x)}{\mathcal{L}^n_a (1)(y)}-1 \right|
& \leq \frac{1}{{\mathcal{L}^n_a (1)(y)}}\sum_{T^n(\tilde x)=x} \left| e^{a_n(\tilde x)} -  
e^{a_n(\tau_{\tilde x}(y))} \right| \\
&
\leq
\sup_{T^n(\tilde x)=x}   \left| e^{a_n(\tilde x)  - a_n(\tau_{\tilde x)}(y))}  -1 \right| \\
&\ll \left(e^{C_a} - 1\right)
\begin{cases}
\omega_{0, (\beta -1) \textrm{log} }(d(x,y)) &: T 
\textrm{ expanding}\\      
\omega_{\alpha - 1/\beta, 0 \textrm{log} }(d(x,y)) &: T \textrm{ expansive}.      
     \end{cases}
\end{align*} 
In particular, $\mathcal{L}^n_a (1) (x) \asymp \mathcal{L}^{n}_a (1) (y)$, provided that $d(x,y)<r$. {  In order to have a more streamlined notation at hand, provided that $\alpha,\beta$ are known, set
\begin{equation} \label{def:omega-dagger}
\omega^\dagger(t) :=
 \begin{cases}
 \omega_{0, (\beta -1) \textrm{log} }(t) &: T
\textrm{ expanding}\\
\omega_{\alpha - 1/\beta, 0 \textrm{log} }(t) &: T \textrm{ expansive}.
     \end{cases}
\end{equation}}
In order to extend this estimate to any $x,y \in \Omega$, we recall that $T$ is called \emph{topologically mixing} if for all open sets $U,V \subset \Omega$ there exists $k\in \mathbb{N}$ such that $T^{-n}(U) \cap V \neq \emptyset$ for all $n > k$. In particular, as the inverse branches of $T$ contract with rate $c_n$, it follows that there exists $p\in \mathbb{N}$ such that for each pair $x,y$, there exists at least one element in $T^{-p}(\{x\})$ which has at most distance $r$ to $y$. Hence, it follows for $n \geq 0$ and $\mathfrak{m}_{k}:= \max_x  \hbox{card}(T^{-p} (\{x\}))$ that
\begin{align*}
e^{p \min a } \mathcal{L}^{n}_a (1) (y) & \leq  e^{C_a} \mathcal{L}^{n+p}_a (1) (x)  \leq e^{C_a} \| \mathcal{L}^{p}_a (1) \|_\infty   \mathcal{L}^{n}_a (1) (x) \\
& \leq e^{C_a} e^{k_0 \max a }  \mathfrak{m}_{p}  \mathcal{L}^{n}_a (1) (x)
\end{align*}
As $a$ is $\omega$-Hölder and the $\omega$-diameter of $\Omega$ is smaller than 1, we obtain that 
\[ \mathcal{L}^n_a (1) (x)  \leq e^{(k_0 +1)C_a} \mathfrak{m}_{k_0} \mathcal{L}^{n}_a (1) (y) \]
for all $x,y \in \Omega$ and  $n \in \mathbb{N}$. Observe that this implies that the limit 
 \[ \rho_a := \lim_{n\to \infty} \sqrt[n]{\mathcal{L}^n_a (1) (x) } \]
exists, by almost submultiplicativity, and is independent of $x$. 
 
A first corollary of these estimates is the following result. As the proof of Proposition 3.1 in \cite{Kloeckner-Lopes-Stadlbauer--Contraction-In-The-Wasserstein--N2015} applies in verbatim, we do not give the proof in here.  
 
\begin{proposition}\label{prop:eigenfunction}
Assume that $T$ is a topological mixing and Ruelle expansive map.
\begin{enumerate}
 \item If $c_n(t)  \ll s^n t$ for some $s\in (0,1)$ and $a$ is  $\omega_{0,\beta \textrm{log}}$-Hölder continuous for some $\beta > 1$, then there exists a $\omega_{0, (\beta -1)\textrm{log}}$-Hölder continuous, strictly positive function $h_a: \Omega \to \mathbb{R}$ such that $\mathcal{L}_a(h_a) = \rho_a h_a$. 
 \item If $c_n(t) \ll t(1 + n t^{1/\beta})^{-\beta}$ and $a$ is $\omega_{\alpha, 0 \textrm{log}}$-Hölder continuous for some $\alpha,\beta > 0$ with $\alpha > 1/\beta$, then  there exists a $\omega_{\alpha - 1/\beta,0\textrm{log}}$-Hölder continuous, strictly positive function $h_a: \Omega \to \mathbb{R}$ such that $\mathcal{L}_a(h_a) = \rho_a h_a$. 
\end{enumerate}
Moreover, if $C_a$ refers to the $\omega$-Hölder coefficient of $a$, we have, in both cases, with respect to implicit constants only depending on $T$, that 
$h_a(x)/h_a(y) \ll e^{p C_a}$ for all $x,y \in \Omega$. Moreover, if $d(x,y) < r$, then
$|h_a(x)/h_a(y) -1| \ll  e^{C_a}  \omega^\dagger (d(x,y))$ { (for the defintion of $\omega^\dagger$, see \eqref{def:omega-dagger})}.
\end{proposition}
 
As an immediate application of Proposition \ref{prop:eigenfunction}, one obtains that $\mathcal{L}^n_a(1)$ behaves asymptotically like $\rho_a^n$ as  
$ \mathcal{L}^n_a(1) \asymp e^{\pm p C_a} \mathcal{L}^n_a(h_a) = e^{\pm p C_a} \rho_a^n h $.
Moreover, for $\rho_a$ and $h_a$ as above, it follows for $\overline{a} := a + \log h - \log h \circ T - \log \rho_a$ that
\[\mathcal{L}^n_{\overline{a}} (f) =  \rho^{-n} \mathcal{L}^n_a(f h/h \circ T^{n}) = \frac{\mathcal{L}^n_a(fh)}{\rho_a^n h} \]
for any $f:\omega \to [0,\infty)$ and $n \in \mathbb{N}$. In particular, $\mathcal{L}_{\overline{a}} (1) = 1$. 
In other words, $\overline{a}$ is a \emph{normalized} potential.  

We now employ $\overline{a}$ in order to define a sequence of probabilities on $\Omega$ by
\[m_x^n := \sum_{T^n(z)=x} e^{\overline{a}_n}\delta_z = (\mathcal{L}_{\overline{a}}^n)^\ast(\delta_x),\]
where $\delta_z$ refers to the Dirac measure {supported}  on $z \in \Omega$ and $(\mathcal{L}_{\overline{a}}^n)^\ast$ to the dual of the action of $\mathcal{L}_{\overline{a}}^n$ on the space of continuous functions. Note that $m_x^n(\Omega)=1$ as $\overline{a}$ is normalized. In order to determine the asymptotics and the dependence on $x \in \Omega$ we now make use of couplings. Recall that a probability measure $P$ on $\Omega \times \Omega$ is referred to as a \emph{coupling} of the probabilities $m_1,m_2$ on $\Omega$ if $m_i = \pi_i^\ast(P)$ where $\pi_i$ is the canonical projection on the $i$-th coordinate and $\pi_i^\ast$ the corresponding action on measures. Furthermore, for {an} $\omega$-Hölder continuous function $f$, we write $\|f\|_\omega := \|f\|_\infty + \textrm{Höl}_\omega(f)$, for
\[\textrm{Höl}_\omega(f) = \sup \{ |f(x)-f(y)|/\omega(d(x,y)) : 0 < d(x,y) < r \}.\] 
We now construct couplings of $m_x^n$ and $m_x^n$ which show a polynomial contraction. Moreover, we take special care on the continuity of the involved constants with respect to $\|a\|_\omega$.

\begin{theorem} \label{theo:main-contraction-result}
 Assume that $T$ is a Ruelle expansive and topologically mixing map of the compact metric space $(\Omega,d)$ of diameter 1 and that $a$ is $\omega$-Hölder continuous. 
\begin{enumerate} 
\item Assume that $c_n(t)  \ll s^n t$ for some $s\in (0,1)$ and that $\omega =  \omega_{0,\beta \textrm{log}}$ for some $\beta > 1$. Then there exists $\kappa_a \geq 1$ depending on $a$ and $T$ such that for any $n > 0$ there exists a  
a coupling $P^n_{x,y}$ of $m_x^n$ and $m_y^n$ such that 
 \[\int d(\tilde{x},\tilde{y}) d  P^n_{x,y} (\tilde{x},\tilde{y}) \leq \kappa_a n^{1-\beta}\]
If $\beta > 2$
then there exist $\Delta_0, n_0\geq 1$ depending on $a$, and $\kappa$ depending on $T$ and $\beta$ such that for all $x,y \in \Omega$, $t \geq d(x,y)$, $\Delta_{t} := \Delta_0 + \log t /\log s$,  the following holds. 
There exists {a} coupling of  $P^n_{x,y}$ of $m_x^n$ and $m_y^n$ with
 \[ 
 \int d(\tilde{x},\tilde{y}) dP^n_{x,y}(\tilde{x},\tilde{y}) 
 \leq  {\kappa}_a
\left( \Delta_t+ \frac{n}{2}\right)^{1 - \beta},
\] 
provided that $n \geq n_0$ and $n$ is sufficiently large such that  $s^n(\Delta_t + n/2)^{\beta-1} \leq 1$. Moreover, for $n\geq n_0$,
\[
\int  \omega_{0, (\beta -1) \textrm{log}}
(d(\tilde{x},\tilde{y})) dP^n_{x,y}(\tilde{x},\tilde{y})  \leq {\kappa}_a   {\left( \frac{\log d(x,y)}{\log s} +  \frac{\Delta_0 + n}{2} \right)^{1 - \beta} }.
\]

 \item Assume that  $c_n(t) \ll t(1 + n t^{1/\beta})^{-\beta}$ and $\omega = \omega_{\alpha, 0 \textrm{log}}$ for some $\alpha,\beta > 0$ with $\alpha \beta > 1$ and $\beta > 1$.  
 Then there exists $\kappa_a \geq 1$ depending on $a$ and $T$ 
 such that for any $n > 0$ there exists {a} coupling $P^n_{x,y}$ of $m_x^n$ and $m_y^n$ such that
\[ \int d(z,\tilde{z}) d  P^n_{x,y} (z,\tilde{z}) \leq \kappa_a n^{1-\alpha\beta}.\] 
If, in addition, $\alpha \beta > 2$  then there exist $\Delta_0,n_0\geq 1$ depending on $a$ and $\kappa\geq 1$ depending on $T$ and $\alpha,\beta$ such that for all $x,y \in \Omega$, there exists  a coupling $P^n_{x,y}$ of $m_x^n$ and $m_y^n$ such that for all $n \geq  n_0$, 
 \[ \int d(\tilde{x},\tilde{y}) dP^n_{x,y}(\tilde{x},\tilde{y}) \leq {\kappa}_a   {\left( d(x,y)^{-1/\beta} +  \frac{\Delta_0 + n}{2} \right)^{1 - \alpha\beta} }. \]
Moreover, for $n\geq n_0$,
\[
\int  \omega_{\alpha - 1/\beta, 0 \textrm{log}}
(d(\tilde{x},\tilde{y})) dP^n_{x,y}(\tilde{x},\tilde{y})  \leq {\kappa}_a   {\left( d(x,y)^{-1/\beta} +  \frac{\Delta_0 + n}{2} \right)^{1 - \alpha\beta} }.
\]
\item In both cases, the constants $\kappa_a,\Delta_0$ and $n_0$ might be chosen in such way that they vary continuously in $\|a\|_\omega$.  
\end{enumerate}
\end{theorem}

Before giving the proof, we would add some remarks and observations to this theorem. 

\begin{remark}
We did not include the case of an expanding map and Hölder continuous potentials, as this is Ruelle's Perron-Frobenius theorem as given, e.g., in \cite{Ruelle--The-Thermodynamic-Formalism-For--CMP1989}. Recall that in this setting, the above decay is exponential. However, we would like to note that the method of proof given below also is applicable. In fact, it suffices to replace the term $(\Delta_t + n)^{-\gamma}$ by $t \lambda^n$ and apply basic estimates for the geometric series.
\end{remark}

\begin{remark} 
Firstly, if $1< \beta \leq 2$ in the expanding case or $1< \alpha\beta \leq 2$ in the expansive case, respectively, then we were only able to obtain a decay rate  of $\int d(\tilde{x},\tilde{y}) dP^n_{x,y}(\tilde{x},\tilde{y})$ which is independent of $d(x,y)$. Secondly, if $ \beta > 2$ or $\alpha\beta > 2$, respectively, then the decay rate with respect to the modified distance $\omega_{\cdot\cdot} \circ d$, that is of $\int \omega_{\cdot\cdot} \circ d(\tilde{x},\tilde{y}) dP^n_{x,y}(\tilde{x},\tilde{y})$ is a function of $n$ and $d(x,y)$. However, in case of an expanding map and weak regularity, the decay with respect to $d$, that is of $\int d(\tilde{x},\tilde{y}) dP^n_{x,y}(\tilde{x},\tilde{y})$, only becomes effective after a certain waiting time, which also depends on $n$ and $d(x,y)$.

Note that these different types of decay have several interpretations in terms of the action of $\mathcal{L}_a$ on Lipschitz and $\omega_{\cdot\cdot}$-Hölder functions or in terms of the Wasserstein distance as defined below (Def. \ref{def:wasserstein}). For example, the asymptotics of $\mathcal{L}_a$ given in Corollaries \ref{cor:decay-to-invariante-measure} and \ref{cor:decay-of-Hoelder-coefficients} are immediate consequences of these two types of decay. In terms of Wasserstein distances, the first case provides us with a uniform decay of 
$W_d((\mathcal{L}_a^n)^\ast(\mu_1),(\mathcal{L}_a^n)^\ast(\mu_2))$ for arbitrary probability measures $\mu_i$. The second case then refines this result by providing decay and a modulus of continuity of the map $\mu \mapsto (\mathcal{L}_a^n)^\ast(\mu)$ with respect to  $W_{\omega_{\cdot\cdot}\circ d}$.

Unfortunately, the authors were not able to verify if it is possible 
to extend the stronger version of decay to all parameters. 
In this case, the transition at 2 only would be a consequence 
of our method of proof. However, a transition at this parameter 
is plausible as it corresponds to the known transition in 
the Manneville-Pomeau family for the parameter $\delta = 1/2$. 
For the transition, we refer to 
\cite{Gouezel--Sharp-Polynomial-Estimates-For--IJM2004} 
whereas it follows from Remark 
\ref{rem:examples-for-decay:Manneville-Pomeau-and-Dyson} 
below, that $\delta < 1/2$ corresponds to $\alpha\beta > 2 $.
\end{remark}

\begin{remark}
{  Due to the similarity of the statements, we discuss the differences to the results by Kloeckner in  \cite{Klo}. First of all, we point out that the results in in  \cite{Klo}  require that the potential $a$ is \emph{flat} with respect to a family of couplings $\{P^n_{x,y}\}$ of $m_x^n$ and $m_y^n$, that is the estimates \eqref{eq:distortion-expanding} or \eqref{eq:distortion-expansive} hold without the restriction that $d(x,y)< r$, but only with respect to $P^n_{x,y}$-almost every pair $(\tilde{x},\tilde{y})$ and $n \in \N$. Kloeckner then shows for example in Theorem E in there, that a for a flat, normalized and $\omega_{0, \beta \textrm{log}}$-Hölder potential with $\beta > 1$ and a uniformly expanding map $T$, $\|\mathcal{L}_a^n(f) - \int f d\mu\|_\infty \ll n^{1-\beta}$, where  $\mu$ is the unique $(\mathcal{L}_a)^\ast$-invariant measure.

However, if $r$ is smaller than the diameter of $\Omega$, flatness is a non-canonical condition due to the following technical reason.
If $Q$ is a measure on $\Omega^2$ such that $Q(A \times \Omega)\leq m_x^n(A) $ and $Q(\Omega \times A)\leq m_y^n(A)$ for all measurable sets $A$ in $\Omega$, then
$Q$ is referred to as a \emph{subcoupling} of $m_x^n$ and $m_y^n$. Moreover, any subcoupling can be extended to a coupling by adding the product of the remaining mass  (see \eqref{eq:extension-to-coupling} in the proof below). The relevance of this extension relies on the fact that there is a canonical construction of subcouplings such that $a$ is flat. However, if \eqref{eq:distortion-expanding} or \eqref{eq:distortion-expansive} do not hold globally, then the potential $a$ no longer necessarily is flat with respect to the extension. In particular, in this situation, there is no canonical way to construct couplings such that a given potential is flat. However, flatness is not necessary for having a contraction result at hand as shown in Theorem \ref{theo:main-contraction-result}).

A well studied class of examples, where  $r$ might be chosen to be strictly bigger than the diameter of $\Omega$, are maps with full branches. That is, if a map $T$ is a local homeomorphism equipped with a finite cover $\mathcal{U}$ of $\Omega$ into closed sets such that, for all $A \in \mathcal{U}$, $T(A) = \Omega$, $T(\partial A)$ is nowhere dense and $T$ restricted to the interior of $A$ is injective, then $T$ is said to have \emph{full branches}. Well known examples in this class are local homeomorphisms of the circle, a full shift over a finite number of symbols or linear endomorphisms of the torus, that is, the action of an invertible matrix with integer coefficients on the $d$-dimensional torus.

In particular, Part (1) of Theorem \ref{theo:main-contraction-result} above is partially stronger as we only require a local version of flatness. That is, our result is without further adaptions applicable to topologically mixing subshifts of finite type or the map defined in Section \ref{subsec-sierpinski}. Moreover, the method in here allows to obtain an estimate for the decay of  $\int d(x,y) dP^n_{x_1,x_2}$ depending on $d(x_1,x_2)$ for $\beta > 2$.}

Furthermore, Theorem  \ref{theo:main-contraction-result} also has analogies and applications to Theorem 4.1 in \cite{Klo}. As we already discussed the expanding case above, we will focus on the expansive, polynomially decaying case. In there, for a given sequence of couplings $M^n_{x_1,x_2}$ of probability measures $m^n_{x_i}$ supported  on $T^{-n}(\{x_i\})$ and a flat and normalized potential $a$, Kloeckner constructs couplings $P^n_{x_1,x_2}$ of  $(\mathcal{T}_a^n)^\ast(\delta_{x_i}) $, for $\mathcal{T}_a (f)(x) := \int f dm^n_{x}$ with the following property. If
$$\int d(x,y) dM^n_{x_1,x_2} \ll  (n+ d(x_1,x_2)^{-1/\alpha})^{-\alpha}$$
for some $\alpha > 0$,
then $P^n_{x_1,x_2}$ shows the same decay. However, as  $\mathcal{T}_a = \mathcal{L}_{\overline{a}}$,  Theorem \ref{theo:main-contraction-result} also can be applied and gives the same result for $\alpha\beta > 2$.

Finally, we also would like to remark that we found {an error}
in the proof of the main contraction result in \cite{Klo}
as Lemma 2.14 in there does not hold. 
After pointing out this problem to Kloeckner, 
he provided a partial solution by imposing a stronger 
but still sufficiently weak hypothesis (see \cite{Kloeckner--An-Optimal-Transportation-Approach--ETDS2022}). 
However, due to global flatness, his results are not applicable in our setting.
\end{remark}

\begin{proof} In order to present the main argument for the contraction, we introduce the following objects and notations. First of all, observe that it follows from topological mixing and the uniform contraction of the inverse branches that there exists $p\in \mathbb{N}$ such that for any pair $U,V \in \mathcal{U}$, there exist $u,v \in \Omega$ with $T^p(u) \in U$, $T^p(v) \in V$ such that $\diam (\tau_u(U) \cup \tau_v(V)) < r$, where $\tau^p_u,\tau^p_v$ refer to the inverse branches of $T^p$ given by the orbits of $u,v$.  
For $x \in U$ and $y \in V$, we now define $u_{x,y} := \tau_u(x)$ and $v_{x,y} :=\tau_v(y)$. As $d(u_{x,y},v_{x,y}) < r$ it follows that for each $z$ with $T^n(z) = u_{x,y}$, $v_{x,y}$ is in the domain of $\tau_z^n$.

Moreover, we refer to a finite measure $Q$ on $\Omega^2$ as a \emph{subcoupling} of two Borel probability measures $m_1,m_2$ on $\Omega$ if $\pi_i^\ast(Q)(A) \leq m_i(A)$ for all Borel sets $A$ and $i=1,2$. Finally,  in order to keep the notation simple, we will use $\kappa$ for a constant which does not depend on $a$, and $\omega^\dagger$ for $\omega_{0, (\beta -1) \textrm{log} }$ or $     
\omega_{\alpha - 1/\beta, 0 \textrm{log} }$, depending on the application of \eqref{eq:distortion-expanding} or \eqref{eq:distortion-expansive}, respectively.

\medskip
\noindent\textsc{Step 1.} We begin with the construction of the underlying subcouplings. For $x,y \in \Omega$ and $k \geq p$, set $\varphi_{z_1,z_2}^{k}:= \min \left\{ e^{\overline{a}_k(z_i)}: i=1,2 \right\}$ and
\[
Q^{k}_{x,y} :=
\begin{cases}
\sum_{T^{k}(z)=x} 
\varphi_{z,\tau_z^k(y)}^{k} 
\delta_{z} \otimes \delta_{\tau_z^k(y)} & : d(x,y) < r,
\\
\sum_{T^{k-p}(z)=u_{x,y}} 
\varphi_{z,\tau_z^{k-p}(v_{x.y})}^{k} 
\delta_{z} \otimes \delta_{\tau_z^{k-p}(v_{x.y})} & : d(x,y) \geq r.   
\end{cases}
\] 
Note that $z = \tau_z^k(x)$ in the first case and that $z = \tau_z^{k-p}(u_{x.y})$ in the second case. In particular, there is a symmetry in the above definition by choosing preimages with respect to choosing the preimages of $x$ or $y$ and $u_{x.y}$ or $v_{x.y}$, respectively.

As $T$ is expansive, it follows in the first case, for $Q^{k}_{x,y}$-a.e. $z_1,z_2$,  that $d(z_1,z_2) \leq c_k(d(x,y))$. By the same argument, one obtains for the second case that 
\[d(z_1,z_2) \leq c_{k-p}(d(u_{x,y},v_{x,y}) \leq c_{k-p}(r) \leq  c_{k-p}(d(x,y)) \ll c_k(d(x,y)),\] 
where the implicit constant in the estimate does not depend on $k$ as $p$ is fixed. That is, $d(z_1,z_2)  \ll c_{k}(d(x,y))$ for $Q^{k}_{x,y}$-a.e. $z_1,z_2$ and all $x,y \in \Omega$ and $k> p$.

\medskip
\noindent\textsc{Step 2.}
We now determine $\pi_i^\ast Q_{x,y}^k$ and obtain lower bounds for $Q_{x,y}^k(\Omega^2)$. If $d(x,y) < r$, then the preimages of $x$ and $y$ come in pairs. 
Hence, it follows from \eqref{eq:distortion-expanding}, \eqref{eq:distortion-expansive} and Proposition \ref{prop:eigenfunction} for $f:\Omega \to [0,\infty)$ that
\begin{align*}
0 & \leq \int f dm_x^k - \int f(\tilde{x}) dQ_{x,y}^k(\tilde{x},\tilde{y})  
 = \sum_{T^k(z)=x} f(z) \left(e^{\overline{a}_k(z)}   - \varphi_{z,\tau_z^k(y)}^{k} \right)  \\  
&   = \sum_{T^k(z)=x} f(z) e^{\overline{a}_k(z)} \left( 1 - \min_{t = z,\tau_z^k(y)} e^{\overline{a}_k(t) - \overline{a}_k(z)}  \right) \\
& \leq \int f dm_x^k \left( 1 - e^{-\kappa C_a \omega^\dagger(d(x,y))}\left( 1 - \kappa e^{C_a} \omega^\dagger(d(x,y))  \right)^2  \right)\\
& \leq \int f dm_x^k \left( 
1 - e^{-\kappa C_a \omega^\dagger(d(x,y))} + 2 \kappa e^{C_a} \omega^\dagger(d(x,y)) \right) \ll e^{\kappa C_a} \omega^\dagger(d(x,y)).
\end{align*}
It follows now from the symmetry in the construction of $Q_{x,y}^k$ that the same estimate holds with respect to $m_x^k$. Hence, $Q_{x,y}^k$  is a subcoupling of $m_x^k$ and $m_y^k$. Furthermore, it follows from the above that
\[ 
0 \leq 1 -  Q_{x,y}^k (\Omega^2) = 1 - \int 1 dQ_{x,y}^k 
\ll e^{\kappa C_a} \omega^\dagger(d(x,y)). \]
Note that these estimates only are effective for $x,y$ close to each other. However, for $x,y$ in arbitrary position with $d(x,y)< r$, the estimates in  \eqref{eq:distortion-expanding}, \eqref{eq:distortion-expansive} imply that    
$ Q_{x,y}^k (\Omega^2) \gg e^{-\kappa C_a}$.  

If $d(x,y) \geq r$, then it follows from the same argument that $Q_{x,y}^k$  is a subcoupling of $m_x^k$ and $m_y^k$. Moreover, by  \eqref{eq:distortion-expanding}, \eqref{eq:distortion-expansive} and Proposition \ref{prop:eigenfunction},
\begin{align*}
 Q_{x,y}^k(\Omega^2) & \geq e^{\min\{ \overline{a}_p(u_{x,y}),\overline{a}_p(v_{x,y})\}} e^{- \kappa C_a \omega^\dagger(r)} \mathcal{L}_{\overline{a}}^{k-p}(1) \\
 & \gg \rho_a^p  e^{p \min_z a(z)} e^{-2pC_a}.  
\end{align*}
Finally, it follows from  $\mathcal{L}_{a}^p(h_a) = \rho_a^p h_a$ that $\rho_a^p \gg 
e^{p \min_z a(z) - 2 pC_a}$. Hence, there exists $\kappa \geq  1$, 
only depending on $T$, such that  
\begin{eqnarray} \nonumber 
  Q_{x,y}^k(\Omega^2) & \geq &  
  \begin{cases}
   \frac{1}{\kappa} e^{\kappa(\min_{z\in \Omega} a(z) - C_a)} & : d(x,y) \geq r \\
   \max\left\{ \frac{1}{\kappa} e^{-\kappa C_a} , 1 -  \kappa e^{\kappa C_a} \omega^\dagger(d(x,y)) \right\}  &:   d(x,y) <  r,
  \end{cases} \\
  & \geq & \max  \left\{ \theta_a^{-1}, 1 - \theta_a \omega^\dagger(d(x,y)) \right\} 
  \label{eq:global-estimate-for-Q-from-below}
\end{eqnarray}  
where $\theta_a^{-1}:= \min \left\{ \kappa^{-1} e^{\kappa(\min_{z\in \Omega} a(z) - C_a)},\kappa^{-1} e^{- \kappa C_a)} \right\}$.

\medskip
\noindent\textsc{Step 3.} We now extend $Q^k_{x,y}$ and give the main estimate for the contraction statement of the theorem. In order to do so, 
set $A_{-1} = \{(z_1,z_2) \in \Omega^2: d(z_1,z_2)\geq r\}$ and, for $j \geq  0$, 
\begin{align*}
 A_j :=  \left\{ (z,\tau_{z}^j(\tilde{z})) ,\in \Omega^2: d(T^j(z),\tilde{z}) < r,d(T^{j+1}(z),T(\tilde{z})) \geq r \right\}. 
\end{align*}
Note that it follows from the construction that the support of $Q^k_{x,y}$ is contained in $A_{j+k}$ for $(x,y) \in A_j$ for $j \geq 0$, and in $A_{k-p}$
for $(x,y) \in A_{-1}$, respectively. Furthermore, as it easily can be verified, 
$Z_{x,y}^k = Q_{x,y}^k + R_{x,y}^k$ is a coupling of $m^k_x$ and $m^k_y$, for  
\begin{equation} \label{eq:extension-to-coupling}
R_{x,y}^k : = \frac{1}{1 - Q_{x,y}^k(\Omega^2)} (\pi_1^\ast Q_{x,y}^k - m^k_x)\otimes (\pi_2^\ast Q_{x,y}^k - m^k_y).
\end{equation}
Finally, it follows from the fact that $T^k$ is a local homeomorphism that the support of $R_{x,y}^k$ is contained in $\bigcup_{i<j+k}A_i$ for $x,y \in A_j$, $j\geq 0$, and in $\bigcup_{i<k}A_i$ for $x,y \in A_{-1}$, respectively.  

We are now in position to prove the main estimate. In order to do so, set $t:= d(x,y)$ and assume that $n \in \mathbb{N}$, that $B_{j} \subset A_j$ for $j = 0,\ldots n$, $B_{-1} \subset \Omega^2$, and that $B_{\ast} \subset \bigcup_{j\geq n}A_j$ for $t<r$ and $B_{\ast} = \emptyset$ for $t \geq r$, respectively. We now construct $k$ and $B^+_j$ for  $j =-1, \ldots, n+k$ and $B^+_\ast$ with the these properties with respect to $n+k$. That is, we consider 
\begin{align*}
B^+_{-1}  &:= \bigcup_{(z_1,z_2)\in B_j+k: j= -1,\ldots m, \ast } \supp(R^{k}_{z_1,z_2}), 
\quad B^+_{k-p}  := \bigcup_{(z_1,z_2)\in B_{-1}} \supp(Q^{k}_{z_1,z_2}), \\
B^+_\ast & := \bigcup_{(z_1,z_2)\in B_\ast}\supp(Q^{k}_{z_1,z_2}), 
\quad B^+_{j} := \emptyset\,  \hbox{ for } j =0, \ldots k-1, j \neq k-p,  
.\\
B^+_{j} &:= \bigcup_{(z_1,z_2)\in B_{j-k}} \supp ( Q^{k}_{z_1,z_2}) \, \hbox{ for } j =k, \ldots n+k.
\end{align*}
As a consequence of the construction of $Q^k$, we have that $B^+_j \subset A_j$ for $j=0, \ldots n+k$ and $B^+_\ast \subset \bigcup_{j\geq n+k}A_j$ or $B^+\ast =\emptyset$, depending on the value of $t$. 

\medskip
\noindent\textsc{Step 4. Polynomial decay.} 
Choose $n_0$ such that $1 - \theta_a c_{n_0+1}(r) \geq \theta_a^{-1}$ and assume that 
for some function $\Delta:[0,r] \to [1,\infty)$ and $\gamma > 1$,
\[\omega^\dagger \circ c_n(t)  \leq \kappa(\Delta(t)+n)^{-\gamma}.\]
Furthermore, assume that $M$ is a probability measure on $B_\ast \cup \bigcup_{j\leq n} B_{j}$ such that, for some  $\Delta_0 \geq  0$, $\delta>0$,
\begin{align}  \label{eq:estimate-for-the-iteration-polynomial}
M(B_j)  &= 0  \hbox{ for }  j=0, \ldots,n_0,\\
\nonumber M(B_j)  &\leq \delta (\Delta_0 + \Delta(t) + n - j)^{-\gamma} \hbox{ for }  j -1,n_0+1, \ldots  ,n.
\end{align}
For $M^+ := Z^k_{z_1,z_2}dM(z_1,z_2)$ it then follows that $M^+(B^+_{j}) \leq M(B_{j-k})$ for $j \neq -1$, which proves  \eqref{eq:estimate-for-the-iteration-polynomial} with respect to $M^+$, $n+k$ and $j=0, \ldots n+k$. So it remains to construct $k$, $\delta$ and  $\Delta_0$, independent of $n$  such that 
\begin{equation} \label{eq:estimate-to-show}
M^+(B^+_{-1}) \leq \delta (\Delta_0 + \Delta(t) + n + k - j)^{-\gamma}. 
\end{equation}
For ease of notation, set $\Delta_t:= \Delta_0 + \Delta(t)$. Moreover,  choose $n_0$ such that $1 - \theta_a \omega^{\dagger}(c_{n_0+1}(r)) \geq \theta_a^{-1}$ and $n_0 < m_0 < n$.   
By \eqref{eq:global-estimate-for-Q-from-below}, 
\begin{align*}
M^+(B^+_{-1}) 
 \leq & \sum_{j=-1}^{n_0} (1 - \theta_a^{-1}) M(B_j) + 
\theta_a \sum_{j=n_0+1}^n  \omega^\dagger(c_j(r)) M(B_j)  + \theta_a c_n(t) M(B_\ast)\\
 \leq &   
\frac{(1 - \theta_a^{-1}) \delta}{(  \Delta_t + n + 1)^{\gamma}} + 
\sum_{j=n_0+1}^{m_0} \frac{\theta_a \kappa \delta}{((\Delta(r) + j)(  \Delta_t + n  - j))^{\gamma}} \\
&  +  \sum_{j=m_0+1}^{n}
\frac{\theta_a \kappa \delta}{((\Delta(r) + j)( \Delta_t + n  - j))^{\gamma}}
+ \theta_a c_n(t) M(B_\ast)\\
\leq &  \frac{(1 - \theta_a^{-1}) \delta }{(  \Delta_t+ n )^{\gamma}}
+ \frac{\theta_a \kappa \delta}{(  \Delta_t + n  - m_0)^{\gamma}} 
\frac{(\gamma -1)^{-1}}{(\Delta(r) + n_0 )^{\gamma - 1} }
\\ & +  \frac{\theta_a \kappa \delta}{(\Delta(r) + m_0)^{\gamma}} \frac{(\gamma -1)^{-1}}{( \Delta_t-1 )^{\gamma - 1} } + \theta_a c_n(t) M(B_\ast).
\end{align*}
We now choose $m_0$. If $n \leq  \Delta_t -\Delta(r)$, set $m_0 = n$. Otherwise, we choose $m_0$ such that $ \Delta_t + n  - m_0 = \Delta(r) + m_0 \pm 1$.  In  both cases, one obtains that
\begin{align*}
M^+(B^+_{-1})   
\leq &   \frac{(1 - \theta_a^{-1}) \delta}{( \Delta_t + n)^{\gamma}} 
+ \frac{ \frac{2\theta_a \kappa \delta}{\gamma -1}\left( n_0^{1 - \gamma}  + (\Delta_0 -1 )^{1- \gamma}   \right) }{(\Delta_t + \Delta(r) + n)^\gamma} 
+ \frac{\theta_a \kappa}{(\Delta(t)+n)^\gamma}.
\end{align*} 
For $\epsilon > 0$, we may enlarge $n_0$ such that 
$2\theta_a \kappa  n_0^{1 - \gamma}(\gamma - 1)^{-1} < \epsilon$. For $\Delta_0 := k/\epsilon$, it then follows from $k \geq n_0+p$ that 
$2\theta_a \kappa  (\Delta_0 -1 )^{1 - \gamma}(\gamma - 1)^{-1} < \epsilon$. 
Moreover, this choice also ensures that, for $\delta:= \kappa\theta_a /\epsilon$,  
\begin{align*}
M^+(B^+_{-1}) \leq & \frac{\delta (1+\epsilon)^\gamma }{(\Delta_t + n + k)^{\gamma}}
\left( (1 - \theta_a^{-1}) + 2\epsilon  + \frac{\kappa \theta_a}{\delta} \frac{(\Delta_t + n)^\gamma}{(\Delta(t) + n)^\gamma}   \right)\\
\leq &  \frac{\delta (1+\epsilon)^\gamma }{(\Delta_t + n + k)^{\gamma}}
\left( 1 - \theta_a^{-1} + 4\epsilon \right).
\end{align*}
The estimate in \ref{eq:estimate-to-show} then follows e.g. for $\epsilon^{-1} := \max \{  8\theta_a ,   2^{\gamma}\theta_a \}$.  

We are now in position to obtain the asymptotic behaviour through induction. Assume that  
that $n = q (n_0 + p) +  r $ for $q,r \in \mathbb{N}$ and $0 \leq r < n_0 + p$. 
{We} may apply the above iteration $q-1$-times for $k = n_0+p$ and once for $ k = n_0 + p + r$ to the Dirac measure $M_0:= \delta_{(x,y)}$ and $B_\ast^0 := \{(x,y)\}$ if $t:= d(x,y) < r$ and
$B_{-1}^0 := \{(x,y)\}$ and $t:=r$, otherwise. Hence, 
\[ 
dP^n_{x,y}(z,\tilde{z}) :=   dZ^{n_0 + p+ r}_{z_{q-1},\tilde{z}_{q-1} } (z,\tilde{z})  
\cdots Z^{n_0 + p}_{z_1,\tilde{z}_1}(z_2,\tilde{z}_2)  dZ^{n_0 + p}_{x,y}(z_1,\tilde{z}_1)  
\]
is a coupling of $m^n_x$ and $m^n_y$. Moreover, if $B^q_j$ and $B^q_\ast$ refer to the set obtained by induction, it follows that 
\begin{align*}
(\ast) := & \int d(z,\tilde{z}) dP^n_{x,y}(z,\tilde{z}) \\ = & \sum_{j ={-1}}^{n} \int_{B^q_j}    d(z,\tilde{z}) dP^n_{x,y}(z,\tilde{z}) + \int_{B^q_\ast} d(z,\tilde{z}) dP^n_{x,y}(z,\tilde{z}) \\
 \leq &  \sum_{j =-1}^{n}  \frac{c_{j}(r) }{\delta (\Delta_t + n -j)^{\gamma}} + c_{n}(t) M_0 (B_\ast^0).
\end{align*}

\medskip
\noindent\textsc{Step 4. Polynomial decay and weak regularity.} We now give the estimate for $\int \omega^\dagger(d)  dP^n_{x,y}$. Observe that it follows from Step 3 that,  for $-1 < n_0< m$,  
\begin{align*}
\int \omega^\dagger(d)  dP^n_{x,y}
\leq &
  \sum_{j =-1}^{m_0}  \frac{\kappa (\Delta(t) + j)^{-\gamma} }{\delta (\Delta_t + n -j)^{\gamma}} +   \sum_{j =m_0 +1}^{n}  \frac{\kappa (\Delta(t) + j)^{-\gamma} }{\delta (\Delta_t + n -j)^{\gamma}} + 
  \kappa (\Delta(t) + n)^{-\gamma}  M_0 (B_\ast^0)\\
  \leq  & 
  \frac{\kappa (\Delta(t) -1)^{1- \gamma} }{\delta (\gamma -1) (\Delta_t + n -m_0)^{\gamma}} 
  + 
    \frac{\kappa (\Delta(t) + m_0)^{- \gamma} }{\delta (\gamma -1) (\Delta_t  -1)^{\gamma -1}} +    \kappa (\Delta(t) + n)^{-\gamma}. 
\end{align*}
For $m_0$ with $\Delta_t + n -m_0 = \Delta(t)  + m_0$, we hence obtain that 
\[
\int \omega^\dagger(d)  dP^n_{x,y} \leq \frac{\kappa}{\delta(1-\gamma)}
 \left(\Delta(t) + \frac{\Delta_0 + n}{2}   \right)^{-\gamma}. 
\] 

\medskip
\noindent\textsc{Step 5. Expanding maps.} If $c_n(t) \leq \kappa s^n t$, then $\omega^\dagger =  \omega_{0,(\beta -1)\textrm{log}}$. Hence, $\gamma = \beta-1$ and $\Delta(t) = ( - \log t  - \kappa )/\log s$ and, in particular, $\beta > 2$. It then follows by separating the sum  at $m_0 = n/2$ that     
\begin{align} \nonumber  
(\ast) \leq &  \frac{   \frac{\kappa r}{s(1-s)} }{\delta (\Delta_t + n/2)^{\beta -1}}
+  \frac{\kappa r s^{n/2} \Delta_t^{2 - \beta}}{\delta|2 - \beta|}  + \kappa s^n t
\end{align}
 Then the middle term is bounded, up to constants depending on $T$ and $\beta$, by $s^{n/2}\Delta_t^{2-\beta}$. Hence, for $n$ with $s^{n/2} (\Delta_t +  {n}/{2})^{\beta -1} \leq 1$, it follows by enlarging $\kappa$, if necessary, that  
\begin{align} \nonumber 
(\ast) \leq & \frac{\kappa}{\delta}  {\left( \Delta_t + \frac{n}{2}\right)^{1 - \beta}}  
. 
\end{align}
 
\medskip  
\noindent\textsc{Step 6. Expansive maps.} 
If $c_n(t) \leq \kappa t(1 + n t^{1/\beta})^{-\beta} = \kappa( t^{-1/\beta}+ n )^{-\beta} $ and $\omega = \omega_{\alpha, 0 \textrm{log}}$ for some $\alpha> 0$ and $\beta > 1$ with $\alpha \beta  >  1$, then $\gamma = \alpha \beta -1 $ and $\Delta(t) = t^{-1/\beta}$. In particular, $\alpha \beta > 2$. 
Now choose $m_0$ such that $\Delta_0 + t^{-1/\beta} + n - m_0  = t^{-1/\beta} + m_0 $. It then follows as above with  $\Delta_t = \Delta_0 + t^{-1/\beta}$ that
\begin{align*}
(\ast) \leq & \frac{ \kappa(t^{-1/\beta})^{-\beta + 1}}{\delta (\Delta_t + n - m_0)^{\alpha\beta - 1}} +  \frac{\kappa(t^{-1/\beta} + m_0)^{-\beta} }{\delta (\alpha \beta -2) } \Delta_t^{\alpha\beta - 2}
 + \kappa(t^{-\frac1{\beta}} + n)^{-\beta} \\
 \ll & \frac{1}{\left( t^{-1/\beta} +  \frac{\Delta_0 + n}{2} \right)^{\alpha\beta  -1} },
\end{align*}
where we have used that $ \alpha \beta - 1 <   \beta $.

\medskip  
\noindent\textsc{Step 7. The case $\gamma >0$.} Now assume that $\omega^\dagger \circ c_n(t) \leq \kappa(\Delta(t)+n)^{-\gamma}$  for some $\gamma > 0$. Furthermore, assume that $M$ is a probability measure on $B_\ast \cup \bigcup_{j\geq n} B_{j}$ where $n$ is chosen in such way that $\omega^\dagger( c_n(r)) \leq \theta_a^{-2} M(B_{-1})/2$. Then, by the same argument and with the same notation as in Step 4, it follows for $k\in \mathbb{N}$ that
\begin{align*}
\int P^m_{x,y}(B_{-1}^+) dM(x,y) \leq & (1 - \theta_a^{-1}) M(B_{-1}) + \theta_a \omega^\dagger(c_n(r)) (1 - M(B_{-1})) 
\\
\leq  &   (1 - \theta_a^{-1}) M(B_{-1}) + \frac{1}{2}\theta_a^{-1}M(B_{-1}  = \left(1 - \frac{1}{2\theta_a} \right) M(B_{-1}).
\end{align*}   
Hence, in order to be in position to apply the estimate again, it suffices to choose $m$ such that $\omega^\dagger( c_{m-p}(r)) \leq \rho \theta_a^{-2} M(B_{-1})/2$, for $\rho:= 1 - \frac{1}{2\theta_a}$. In other words, for $n_k$ with $\omega^\dagger( c_{n_k -p}(r)) \leq \rho^k \theta_a^{-2} /2$, for $k=0,1, \cdots$, it follows for 
\[  Q_k :=  P_{x_k,y_k}^{n_k} \cdots  dP^{n_0}_{x_0,y_0} (x_1,y_1)    dP_{x,y}^{n_0}(x_0,y_0)
\]
that $Q_k(B_{-1}^k) \leq \rho^k$ and $Q_k(B_{j}^k)=0$ for $k=0,\ldots,m-p-1$, where $B^k_\cdot$ refers to the set obtained by applying $B_\cdot \mapsto B_\cdot^+$ $k$-times. In particular,
\[ \int d(z,\tilde{z}) dQ^k \leq \rho^k + c_{n_k-p}(r)
.
\]      
As  $\omega^\dagger \circ  c_j(r) \leq \kappa (\Delta(r)+j)^{-\gamma}$ it follows that one may choose $n_k$ to be smallest integer bigger than $ p - \Delta(r) + (2\kappa\theta_a^2)^{1/\delta} \rho^{-k/\gamma}$. For this choice of $n_k$, one then obtains that $\sum_{j=0}^{k} n_j \sim (2\kappa\theta_a^2)^{1/\delta} \rho^{-(k+1)/\gamma}$. The statement on the decay for $\gamma > 0$ follows from this.  

\medskip  
\noindent\textsc{Step 8. Continuity of $\Delta_0$ and $n_0$.} First of all, note that  \eqref{eq:global-estimate-for-Q-from-below} holds for $\theta_a := \kappa \exp (\kappa \| a\|_\omega)$ for some $\kappa$ only depending on $T$ and $\omega^\dagger$. In particular, by eventually enlarging $\kappa$, we may assume that the parameter $\epsilon$ at the end of Step 4 might be chosen as $\epsilon := \kappa^{-1} \exp ( - \kappa \| a\|_\omega)$. In particular,  $\delta > 0 $ in fact might be chosen independently from $a$ and that 
$\Delta_0 := 2 n_0 \kappa \exp (\kappa \| a\|_\omega)$ satisfies the required properties. As $n_0$ might be chosen to be the minimal value such that 
$\omega^\dagger(c_{n_0 + 1}(r)) \leq \theta_a^{-1} - \theta_a^{-2}$ and $2 (\kappa \theta_a)^2 n_0^{1-\gamma} < \gamma -1$, it follows that $\Delta_0$ and $n_0$ vary continuously with $\|a\|_\omega$.  
\end{proof}

We now discuss applications of Theorem \ref{theo:main-contraction-result} with respect to the Hölder continuity of the equilibrium state with respect to $a$. In order to do so, we first introduce several new objects in order to be able to treat both cases in Theorem \ref{theo:main-contraction-result} simultaneously. That is, in the expanding case, we set 
\begin{equation}  \label{def:omega-n-expanding}
\gamma := \beta -1, \quad \omega^\dagger := \omega_{0, \beta -1 \textrm{log}}, \quad
\omega_n(t) := {\kappa}   {\left(  {\log t}/{\log s} +   {\Delta_0/2  + n}/{2} \right)^{1 - \beta} }.
\end{equation}
On the other hand, if $T$ is expansive,  we consider
\begin{equation}  \label{def:omega-n-expansive}
\gamma := \alpha\beta -1, \quad \omega^\dagger:=\omega_{\alpha - 1/\beta, 0 \textrm{log}}, \quad
\omega_n(t) :=  {\kappa}   {\left( t^{-1/\beta} +  {\Delta_0/2 + n/2}  \right)^{1 - \alpha\beta} }.
\end{equation}
A further key ingredient is the Wasserstein distance, which is defined as follows. 
\begin{definition}\label{def:wasserstein}
Assume that $\mu_1,\mu_2$ are probability measures on $\Omega$. Then 
\[W(\mu_1,\mu_2) := \inf \left\{ \textstyle \int d(x,y) dQ(x,y) : Q \hbox{ is  a coupling of }\mu_1,\mu_2   \right\} \]
refers to the Wasserstein distance of $\mu_1$ and $\mu_2$. 
\end{definition}
As it is well known, $W$ is a metric on the space of Borel probability measures which is compatible with the weak$^\ast$-topology. Moreover, by Kantorovich's duality,
\[
W_d(\mu_1,\mu_2) = \sup\{ \textstyle \int f d(\mu_1 - \mu_2): 
\textrm{Lip}(f) \leq 1 \},
\]
where $\textrm{Lip}(f)$ refers to the Lipschitz constant of $f$ (i.e. $\textrm{Lip}(f) = \textrm{Höl}_{\omega_{1,0 \textrm{log}}}(f)$). By having a look to Theorem \ref{theo:main-contraction-result}, it is natural to consider an equivalent metric on $\Omega$. In order to do so, note that for any concave and strictly increasing function $\omega:[0,1] \to [0,\infty)$ with $\omega(0)=0$, it follows from the subadditivity of concave functions that $\omega(d)$ is a metric. Moreover, as $\omega$ is invertible, it follows that the topologies generated by $d$ and $\omega(d)$ coincide. With respect to this new metric, we will write  
\[W_\omega(\mu_1,\mu_2) := \inf \left\{ \textstyle \int \omega(d(x,y)) dQ(x,y) : Q \hbox{ is  a coupling of }\mu_1,\mu_2   \right\}. \]

We begin with an analysis of the asymptotics of $(\mathcal{L}^n_{\overline{a}})^\ast(\mu_i)$ with respect to $W_d$, where $(\mathcal{L}_{\overline{a}})^\ast$ is the dual of the action  of $\mathcal{L}_{\overline{a}}$ on the space of continuous functions. By Kantorovich's duality, it suffices in fact to analyse the action of $\mathcal{L}_{\overline{a}}$ on Lipschitz continuous functions (i.e. $\omega_{1,0 \textrm{log}}$-Hölder continuous functions). So assume that  $Q$ is a coupling of  $\mu_1$ and $\mu_2$  and that Theorem \ref{theo:main-contraction-result} is applicable. Then $P_{x,y}^ndQ(x,y)$ defines a coupling of the $(\mathcal{L}^n_{\overline{a}})^\ast(\mu_i)$, where $P^n_{x,y}$ is as in Theorem \ref{theo:main-contraction-result}. Hence, if $f$ is Lipschitz continuous, then  
\begin{align*}
\left| \int \mathcal{L}^n_{\overline{a}}(f) d\mu_1  - \mathcal{L}^n_{\overline{a}}(f) d\mu_2 \right|
& = \left| \int \mathcal{L}^n_{\overline{a}}(f)(x)   - \mathcal{L}^n_{\overline{a}}(f)(y) 
 dQ(x,y) \right|
\\ & 
 = \left| \int f(\tilde{x}) - f(\tilde{y})  dP^n_{x,y}(\tilde{x},\tilde{y}) 
 dQ(x,y) \right| 
 \\ & 
\leq \textrm{Lip}(f)  \int d(\tilde{x},\tilde{y}) dP^n_{x,y}(\tilde{x},\tilde{y}) dQ(x,y) \ll  \omega_n(r).
\end{align*}
That is, $W_d((\mathcal{L}^n_{\overline{a}})^\ast(\mu_1),(\mathcal{L}^n_{\overline{a}})^\ast(\mu_2)) \ll  \omega_n(r)$ and, in particular,  $((\mathcal{L}^n_{\overline{a}})^\ast(\mu_1): n\in \mathbb{N})$ is a Cauchy sequence and therefore converges to a probability measure $\mu$. It then follows from continuity that $((\mathcal{L}_{\overline{a}})^\ast(\mu) = \mu$. By applying the estimate to $\delta_x$ and $\mu$ and after changing the metric, { one then obtains in both cases the following corollary. In order to give unified statements, we now will make use of $\omega^\dagger$, $\gamma$ and $\omega_n(r)$ as defined in \eqref{def:omega-n-expanding} or \eqref{def:omega-n-expansive}.
}

\begin{corollary} \label{cor:decay-to-invariante-measure} 
Under the assumptions of Theorem \ref{theo:main-contraction-result}, assume that $\gamma > 0$.
Then there exists a unique Borel probability measure $\mu$ with $((\mathcal{L}_{\overline{a}})^\ast(\mu) = \mu$. Moreover, for any Lipschitz continuous function $f$ and $n > n_0$,
$\| \mathcal{L}^n_{\overline{a}}(f)  - \textstyle \int f d\mu \|_\infty  
\ll \textrm{Lip}(f) \omega_n(r).$
If $f$ is $\omega^\dagger$-Hölder continuous, then
$\| \mathcal{L}^n_{\overline{a}}(f)  - \textstyle \int f d\mu \|_\infty  
\ll \textrm{Höl}_{\omega^\dagger}(f) \; \omega_n(r)$. 
\end{corollary}
In order to obtain an estimate on the Hölder regularity of $ \mathcal{L}^n_{\overline{a}}(f)$, it suffices to consider $\mu_i := \delta_{x_i}$ for $i=1,2$ in the above estimate and apply Theorem \ref{theo:main-contraction-result} for the case $\gamma > 1$. Hence, 
\[
\left|   \mathcal{L}^n_{\overline{a}}(f)(x_1)  - \mathcal{L}^n_{\overline{a}}(f)(x_2) \right|
  \ll \textrm{Lip}(f) \omega_n(d(x_1,x_2))
\]
for $n$ sufficiently large. That is, if $T$ is expanding, then $\gamma =\beta -1$ and $n$ has to be {bigger than} some lower bound which depends on $d(x,y)$. Therefore, Theorem \ref{theo:main-contraction-result} does not provide an answer in this case. However, by considering functions which are Lipschitz continuous with respect to the metric $\omega^\dagger(d)$, the estimate on $\int \omega^\dagger(d) dP^n_{x,y}$ provides a uniform upper bound for $n \geq n_0$.
That is, one immediately obtains the following result {(for the definitions of $\gamma$, $\omega^\dagger$ and $\omega_n$, see \eqref{def:omega-n-expanding} or \eqref{def:omega-n-expansive})}.
\begin{corollary} \label{cor:decay-of-Hoelder-coefficients} Under the assumptions of Theorem \ref{theo:main-contraction-result}, assume that $\gamma > 1$ and that $f$ is $\omega^\dagger$-Hölder continuous.
We then have with respect to the measure $\mu$ given by Corollary \ref{cor:decay-to-invariante-measure}, for $n \geq 0$,   that
$
\left\| \mathcal{L}^n_{\overline{a}}(f)  - \textstyle \int f d\mu \right\|_{\omega_n}  
\ll  \textrm{Höl}_{\omega^\dagger}(f).
$
\end{corollary}

We  now discuss how these results allow to obtain the asymptotics of $\mathcal{L}_a^n$ 
and the continuity of $\mu_a$, $\lambda_a$ and $h_a$ as functions of $a$. 
The key observation in here is that       
$\mathcal{L}^n_{\overline{a}}(f/h_a) = {\mathcal{L}^n_{a}(f) }/{\rho_a^n h_a}$ which allows to apply Corollary \ref{cor:decay-to-invariante-measure} to  $\mathcal{L}^n_{a}$. 
In particular, if we assume that $\mu(h^{-1}) =1$, we have that 
\begin{equation} \label{eq:definition-h}
\left\|  \rho^{-n} \mathcal{L}^n_{a}(\mathbf{1})  - h  \right\|_\infty \leq \kappa e^{\kappa C_a}  (c_a + n )^{-\gamma}.
\end{equation}
For  $dm_a:= \mu_a(h_a^{-1})^{-1} h_a^{-1}d\mu =h_a^{-1}d\mu  $, we then obtain the following. 

\begin{corollary}\label{cor:decay-to-conformal-measure} Assume that $\gamma > 0$ and $f$ is $\omega^\dagger$-Hölder continuous. Then, with  respect to  $C_a,C_b$ as in \eqref{eq:distortion-expanding} and \eqref{eq:distortion-expansive}, $c_a$ depending continuously on $\|a\|_\omega$ and $\kappa$ depending only on $T$ and $\gamma$, 
  \[ \left\|  \frac{\mathcal{L}^n_{a}(f)(x)}{\mathcal{L}^n_{a}(\mathbf{1})(x)} - m_a(f) \right\|_\infty \leq \kappa e^{\kappa C_a} \|f\|_{\omega^\dagger} (c_a + n )^{-\gamma}.
  \]
\end{corollary}

\begin{proof} It follows from Corollary \ref{cor:decay-to-invariante-measure} that
 \begin{align*}
&   \left|  \frac{\mathcal{L}^n_{a}(f)(x)}{\mathcal{L}^n_{a}(\mathbf{1})(x)} - m_a(f) \right| 
= \left|  \frac{\mathcal{L}^n_{\overline{a}}(f/h_a)(x)}{\mathcal{L}^n_{\overline{a}}(\mathbf{1}/h_a)(x)} - m_a(f) \right| \\
\leq &  
  \frac{1}{{\mathcal{L}^n_{\overline{a}}(h_a^{-1})(x)}} 
  \left( \left|  \mathcal{L}^n_{\overline{a}}(f h_a^{-1})(x) - \mu_a(f h_a^{-1}) \right| + \left|  
  \mu_a(f h_a^{-1}) - m(f) \mathcal{L}^n_{\overline{a}}(h_a^{-1})(x)  \right| 
  \right) \\
\leq & \kappa e^{pC_a} \left( \textrm{H\"ol}_{\omega^\dagger}(fh_a^{-1}) + \|f\|_\infty \textrm{H\"ol}_{\omega^\dagger}(h_a^{-1}) \right) (c_a + n)^{-\gamma} \\
\leq & 2 \kappa e^{pC_a}  \|f\|_{\omega^\dagger} \|h_a^{-1}\|_{\omega^\dagger}   (c_a + n)^{-\gamma}.   
 \end{align*}
Hence, it remains to remark that $\|h_a^{-1}\|_{\omega^\dagger} \leq 2 e^{(p+1)C_a}$ by Proposition \ref{prop:eigenfunction}  
\end{proof}

Note that the above statement is independent from $h_a$ and $\rho_a$. In particular, it immediately follows that $a \to m_a$ and $a \mapsto \rho_a = \int \mathcal{L}(1) dm_a$ are continuous. 
{
Before analysing the regularity of $m_a$, $h_a$ and $\rho_a$ in detail, we show how Corollary \ref{cor:decay-to-conformal-measure} implies that these objects are unique.

\begin{corollary}\label{cor:uniqueness} If $\gamma > 0$, then the following holds.
\begin{enumerate}
 \item Assume that $\tilde{h}: \Omega \to [0,\infty)$ is a continuous function such that $\|\tilde{h}\|_\infty > 0$ and $\mathcal{L}_{a}(\tilde{h}) = \tilde{\rho} \tilde{h}$ for some $\tilde{\rho} > 0$. Then $\tilde{h} / h_a$  is a positive constant and $\tilde{\rho} = \rho_a$.
 \item Assume that $\tilde{m}$ is a probability measure such that $\int \mathcal{L}^n_{a}(f) d\tilde{m} = \tilde{\rho} \int f d\tilde{m}$ for some $\tilde{\rho} > 0$ and any continuous $f$. Then, $\tilde{m}=m_a$ and $\tilde{\rho} = \rho_a$.
\end{enumerate}
\end{corollary}

\begin{proof}
 We begin with the proof of the first assumption. In order to do so, note that topological transitivity implies that $\mathcal{O}_-(x) := \bigcup_{n= 1}^\infty T^{-n}(\{ x\}\}$ is dense in $\Omega$. In particular, if $\tilde{h}(x)=0$, then $\mathcal{L}_{a}^n(\tilde{h})(x) = \tilde{\rho}^n \tilde{h}(x) =0$ implies that $\tilde{h}(y) =0$ for all $y \in \mathcal{O}_-(x)$, a contradiction to $\|\tilde{h}\|_\infty > 0$.

 Hence, $\tilde{h}$ is bounded away from $0$ and infinity. Therefore, there exists $C> 0$ with $C^{-1} h_a < \tilde{h}  < C h_a$. Hence, by positivity of $\mathcal{L}_{a}$, we obtain for all $n \in \mathbb{N}$ that
 \[  C^{-1} \rho_a^n  h_a = C^{-1} \mathcal{L}_{a}^n(h_a)(x)  <   \mathcal{L}_{a}^n(\tilde{h})(x) = \tilde{\rho}^n \tilde{h} < C  \mathcal{L}_{a}^n(h_a)(x) =  C  \rho_a^n h_a. \]
 This implies that $\tilde{\rho} = \rho_a$. In order to show the remaining assertion of the first part, note that it follows from the fact that the Lipschitz continuous functions are dense in the space of continuous functions and Corollary \ref{cor:decay-to-conformal-measure} that
 $ \lim_{n \to \infty} \mathcal{L}_{a}^n(\tilde{h})/ \mathcal{L}_{a}^n(\mathbf{1})  = m_a(\tilde{h})$. Hence, $\tilde{h} / h_a = m_a(\tilde{h})/m_a({h})$.

 For the second assertion, note that, $\int \mathcal{L}_{a}^n(\mathbf{1}) d \tilde{m} = \tilde{\rho}^n$ by iteration. However, it follows from \eqref{eq:distortion-expanding} and \eqref{eq:distortion-expansive} that the left hand side is comparable, up to some multiplicative constant, to $\mathcal{L}_{a}^n(\mathbf{1})(x)$, independent of $x \in \Omega$. By repeating the argument with respect to $m_a$, it follows that $\tilde{\rho} = \rho_a$.

 By applying Corollary \ref{cor:decay-to-conformal-measure} to $h_a$, it follows that $ \mathcal{L}_{a}^n(\mathbf{1}) / \rho_a^n  \to 1$ uniformly as $n \to \infty$. This implies that $\| \mathcal{L}_{a}^n(\mathbf{1}) / \rho_a^n \|_\infty$ is uniformly bounded. Therefore, one obtains from Corollary \ref{cor:decay-to-conformal-measure} by multiplying with $ \mathcal{L}_{a}^n(\mathbf{1}) / \rho_a^n$ and using again the density of Lipschitz functions that
 \[ \lim_{n \to \infty} \|\rho_a^{-n} \mathcal{L}_{a}^n(f)  - \rho_a^{-n} m_a(f) \mathcal{L}_{a}^n(\mathbf{1}) \|_\infty =0,\]
 for any  continuous $f$. The assertion follows by integrating with respect to $\tilde{m}$.
\end{proof}

We now close this section by proving that these objects depend Hölder continuously on $a$.
}

\begin{theorem} \label{theo:continuity}
 Assume that the conditions of Theorem \ref{theo:main-contraction-result} hold with respect to the potentials $a$, $b$ and assume that $\gamma>0$. Then, for some $\kappa$ only depending on $T$ and $\gamma$ and with $C_a,C_b$ as in \eqref{eq:distortion-expanding} and \eqref{eq:distortion-expansive},
  \begin{align}
\label{eq:Hoelder-continuity-equilibrium-conformal-lyapunov}
  W_{\omega^\dagger}(\mu_a,\mu_b),  W_{\omega^\dagger}(m_a,m_b), |\rho_a - \rho_b|, \|h_a - h_b\|_\infty  \\
  \nonumber <   e^{\| a-b\|_\infty} \kappa \min\left\{e^{\kappa C_a}, e^{\kappa C_b}\right\}  \| a-b\|_\infty^{\frac{\gamma}{\gamma + 1}}.  
\end{align}  
%
\end{theorem}

\begin{proof} 
Assume that $h_a$ and $h_b$ are the functions given by \eqref{eq:definition-h} with respect to $a$ and $b$ and observe that the constants {in \eqref{eq:definition-h}} are of the form $e^{p C_a}$ and $e^{C_a}$. Furthermore, observe that
$ \mathcal{L}^n_{\overline{a}}(f/h_a) = {\mathcal{L}^n_{a}(f) }/{\rho^n h_a}$,
Hence, by Corollary \ref{cor:decay-to-invariante-measure} and Proposition \ref{prop:eigenfunction}, 
 \begin{align*}
\|\rho_a^{-n}  \mathcal{L}^n_{a}(1)   - h_a \mu(h_a^{-1}) \|_\infty \leq  \|h_a\|_\infty  \textrm{Höl}_{\omega^\dagger}(1/h_a)  (c_a + n)^{-\gamma} \leq e^{(2p+1)C_a} (c_a + n)^{-\gamma}. 
 \end{align*}
Moreover, $c_a$ depends continuously on $\|a\|_\omega$ by Theorem \ref{theo:main-contraction-result}. In particular, $\lim \rho_a^{-n}  \mathcal{L}^n_{a}(1) =h_a$. 
By combining \eqref{eq:distortion-expanding}, \eqref{eq:distortion-expansive}, the continuity properties of $h_a$ and the above convergence, one then obtains that $ \| \mathcal{L}^n_{\overline{a}}(1/h_a) \|_{\omega^\dagger} \leq  e^{\kappa C_a}$ with respect to some $\kappa$ only depending on $T$. As $\textrm{Höl}_{\omega^\dagger}(fg) \leq \|f\|_{\omega^\dagger} \|g\|_{\omega^\dagger}$, by possibly enlarging $\kappa$,
\begin{align*}
& \left\|\frac{\mathcal{L}^n_a(f \mathcal{L}^n_{a}(1))}{\rho_a^{2n} h_a} - \mu_a(f) \mu_a(h_a^{-1})\right\|_\infty 
= \left\| \mathcal{L}^n_{\overline{a}}(f \mathcal{L}^n_{\overline{a}}(h_a^{-1}))   - \mu_a(f) \mu_a(h_a^{-1})\right\|_\infty \\
  \leq & 
 \left\| \mathcal{L}^n_{\overline{a}}(f \mathcal{L}^n_{\overline{a}}(h_a^{-1}))   - \mu_a(f \mathcal{L}^n_{\overline{a}}(h_a^{-1}))\right\|_\infty +  \left| \mu_a(f (\mathcal{L}^n_{\overline{a}}( h_a^{-1})  -\mu(h_a^{-1})) )\right|
 \\
 \leq & \kappa \frac{\textrm{Höl}_{\omega^\dagger}(f \mathcal{L}^n_{\overline{a}}(h_a^{-1}) )}{(c_a + n)^{\gamma}}  + \|f\|_\infty \|\mathcal{L}^n_{\overline{a}}( h_a^{-1})  - \mu_a(h_a^{-1}) \|_\infty
 \\
 \leq & \kappa \frac{\|f\|_{\omega^\dagger} \| \mathcal{L}^n_{\overline{a}}(h_a^{-1}) \|_{\omega^\dagger}  }{(c_a + n)^{\gamma}} + 
  \kappa \frac{\textrm{Höl}_{\omega^\dagger}(h_a^{-1}) \|f\|_\infty}{ (c_a + n)^{\gamma}} \leq   \frac{2 \kappa e^{\kappa C_a}}{ (c_a + n)^{\gamma}} \|f\|_{\omega^\dagger}.    
\end{align*}
We are now in position to obtain an approximation which is independent from $\rho_a$ and $h_a$: 
\begin{align*}\nonumber 
& \left\|\frac{\mathcal{L}^n_a(f \mathcal{L}^n_{a}(1))}{\mathcal{L}^{2n}_{a}(1)} - \mu_a(f) \right\|_\infty 
= \left\| \frac{\mathcal{L}^n_{\overline{a}}(f \mathcal{L}^n_{\overline{a}}(h_a^{-1}))}{\mathcal{L}^{2n}_{\overline{a}}(h_a^{-1})} - \mu_a(f) \right\|_\infty \\
  \nonumber 
  \leq &
  \left\| \frac{\mathcal{L}^n_{\overline{a}}(f \mathcal{L}^n_{\overline{a}}(h_a^{-1}))}{\mathcal{L}^{2n}_{\overline{a}}(h_a^{-1})} - \frac{\mu_a(f) \mu(h_a^{-1})}{\mathcal{L}^{2n}_{\overline{a}}(h_a^{-1})} \right|_\infty
   + |\mu_a(f)|  \left\| \frac{ \mu(h_a^{-1})}{\mathcal{L}^{2n}_{\overline{a}}(h_a^{-1})}   -1 \right\|_\infty 
\\
\label{eq:decay-of-quotient}
\leq & \left\| \frac{1}{\mathcal{L}^{2n}_{\overline{a}}(h_a^{-1})} \right\|_\infty \left( 
 \frac{2 \kappa e^{\kappa C_a}}{ (c_a + n)^{\gamma}} \|f\|_{\omega^\dagger} +  \frac{ e^{\kappa C_a}}{ (c_a + 2n)^{\gamma}} \| f \|_\infty   \right) 
 \ll   \frac{\kappa e^{\kappa C_a}}{ (c_a + n)^{\gamma}} \|f\|_{\omega^\dagger}
 {.}
\end{align*}
We now compare the action of $\mathcal{L}^n_a$ and $\mathcal{L}^n_b$ on $\omega^\dagger$-Hölder functions. Note that

\begin{align*}
& \left| \frac{\mathcal{L}^n_a(f\mathcal{L}^n_a(1)) }{\mathcal{L}^{2n}_a(1)} -  \frac{\mathcal{L}^n_a(f\mathcal{L}^n_b(1)) }{\mathcal{L}^{2n}_b(1)} \right| \\
  \leq & \left| \frac{\mathcal{L}^n_a\left(f \mathcal{L}^n_a(1)\left(1 - e^{b_n - a_n} \frac{\mathcal{L}^n_a(  e^{b_n - a_n} )}{\mathcal{L}^n_a(1)} \right)\right)}{\mathcal{L}^{2n}_a(1)}  \right| 
  + \left|  \frac{\mathcal{L}^n_a(f\mathcal{L}^n_b(1)) }{\mathcal{L}^{2n}_b(1)}  \right|
    \left|  \frac{\mathcal{L}^{2n}_b(1)}{\mathcal{L}^{2n}_a(1)} -1\right|
 \\
 \leq & \|f\|_\infty \left|e^{2   \| a_n-b_n\|_\infty} -1\right| + 
 \|f\|_\infty \left|e^{ \| a_{2n}-b_{2n} \|_\infty} -1\right| \leq 2 \|f\|_\infty \left|e^{2 n \| a-b\|_\infty} -1\right|
\end{align*}
We now give an estimate $\mu_a(f) - \mu_b(f)$. As this quantity  is invariant under adding a constant to $f$, we may assume without loss of generality that $\min f =0$. Hence, $\|f\|_{\omega^\dagger} \leq (1 + \omega^\dagger(1)) \textrm{Höl}_{\omega^\dagger}(f)$. For  $n = \| a-b\|_\infty^{-1/(\gamma + 1)}$, one then obtains from the last two estimates that 
\begin{align*}
 |\mu_a(f) - \mu_b(f)| & \ll   \frac{\textrm{Höl}_{\omega^\dagger}(f) e^{\kappa C_a}}{(c_a + n)^\gamma}   + e^{\| a-b\|_\infty} \textrm{Höl}_{\omega^\dagger}(f)  \| a-b\|_\infty^{\frac{\gamma}{\gamma + 1} } + \frac{\textrm{Höl}_{\omega^\dagger}(f) e^{\kappa C_b}}{(c_b + n)^\gamma}  \\
 & \ll \textrm{Höl}_{\omega^\dagger}(f) e^{\| a-b\|_\infty} \min\left\{e^{\kappa C_a}, e^{\kappa C_b}\right\} \| a-b\|_\infty^{\frac{\gamma}{\gamma + 1} }.
\end{align*}
The first estimate in \eqref{eq:Hoelder-continuity-equilibrium-conformal-lyapunov} and the second estimate 
follows by the same lines. For the continuity of $\rho_a$, note that 
\begin{align*}
|\rho_a - \rho_b| & = \left| \int \mathcal{L}_a(1) dm_a +  \int \mathcal{L}_b(1) dm_b \right| \\
& \leq  \left| \int \mathcal{L}_a(1) - \mathcal{L}_b(1)   \right| dm_a  +  \left| \int \mathcal{L}_b(1)  d(m_a - m_b)  \right|\\
& \leq  \left(e^{\|b-a\|_\infty} -1\right) \rho_a +  e^{\| a-b\|_\infty} \min\left\{e^{\kappa C_a}, e^{\kappa C_b}\right\}  \| a-b\|_\infty^{\frac{\gamma}{\gamma + 1}} \textrm{Höl}_{\omega^\dagger}(\mathcal{L}_b(1)) \\
& \ll e^{\|b-a\|_\infty} \min\left\{e^{\kappa C_a}, e^{\kappa C_b}\right\}
\| a-b\|_\infty^{\frac{\gamma}{\gamma + 1}}. 
\end{align*}
The final estimate in \eqref{eq:Hoelder-continuity-equilibrium-conformal-lyapunov} follows by similar arguments, based on the fact that $\| \rho_a^n \mathcal{L}^n_a(\mathbf{1}) - h_a \|_\infty \ll \|h_a\|_\infty  (c_a + n)^\gamma$.
\end{proof}

\begin{remark} \label{rem:examples-for-decay:Manneville-Pomeau-and-Dyson}
We now discuss two classes of examples where these conditions turn out to be natural. We begin with the Pomeau-Manneville family in Example \ref{ex:Manneville}. For $\delta \in (0,1)$, this map is Ruelle expansive with contraction rate $c_n(t) \ll t(1+ nt^{\delta})^{-1/\delta}$. Hence, the above results with respect to the $\sup$-norm hold with respect to potential functions $a$, which are $\alpha$-Hölder continuous for $\delta < \alpha $.  As $\alpha$-Hölder continuous functions on $\mathbb{R}$ for $\alpha> 1$ are constant, the natural range for the Hölder exponent of $a$ is 
$\delta < \alpha \leq 1$. 

Furthermore, we have that $\alpha \delta^{-1} > 2$ if $\alpha > 2 \delta$. Therefore, in order to apply results like Corollary \ref{cor:decay-of-Hoelder-coefficients} for the decay of regularity, one has to assume that $\delta < 1/2$ and $2 \delta < \alpha \leq 1$. Observe that a change of behaviour at $\delta = 1/2$ is reasonable as it is known from the decay of correlation of the invariant probability measure which is absolutely continuous with respect to Lebesgue measure, that this is a critical value (see \cite{Sarig--Subexponential-Decay-Of-Correlations--IM2002,Gouezel--Sharp-Polynomial-Estimates-For--IJM2004}). However, the results in here do not cover this measure as it is associated to the geometrical potential $-\log T'$, which is $\delta$-Hölder. As a concluding remark with respect to this example, we would like to remark that the expansive behaviour of $T$ is only due to the indifferent fixed point at zero, i.e. consequence of an isolated phenomena. This can be used in fact to obtain exponential decay based on the combinatorial observation that there are more expanding than expansive branches (see \cite{Klo}). 

The second class of examples is related to the Dyson model of a ferromagnet with binary spins and polynomially decaying interactions, which is the physically relevant scale. In terms of thermodynamic formalism, this translates to considering  the shift on $\{-1,1\}^\N$ and the potential 
\[ a((x_1,x_2,\ldots )) = \sum_{n=1}^\infty x_1x_{n+1} n^{-\alpha}.\]
As it easily can be seen, this potential is $\omega_{0,(\alpha -1) \textrm{log}}$-Hölder continuous (cf. Section \ref{subsec:dyson}) and that his weak regularity is not an isolated phenomenon as above. Furthermore, it follows that there is polynomial decay of correlation whenever $\alpha> 2$. This adds a new detail to this well-studied model as so far it only was known that there exists more than one equilibrium state for $\alpha =2$ (see \cite{Frohlich-Spencer--The-Phase-Transition-In--CMP1982}) and a unique one for 
$\alpha > 2$, as in this case, the unicity follows, e.g., from the observation that the potential satisfies Walters condition (see \cite{MR2342978}).

\end{remark}

\section{Spectral Triples and their spectral metrics} \label{mainsec-1}
In this section we adapt the construction of the spectral triple  by R. Sharp in \cite{MR2990125,MR3485414} to our more general setting of Ruelle expansive maps. Moreover, we provide criteria such that Connes' pseudo-metric is comparable to the Wasserstein distance on probability measures.

\subsection{Spectral triples}
\label{sec-spe-trip-norm-potentials}
We now begin with the construction of the spectral triples. In order to do so, we will make use of the fact that preimages of Ruelle expanding maps come in pairs. Therefore recall that for any   
$(r,\lambda)$-expanding map $T$ the map $T^n$ is $(r/2,\lambda^n)$-expanding. In particular, there exists an open cover $\mathcal{U}$ of $\Omega$ such that     
 for any $n \in \mathbb{N}$, $U \in \mathcal{U}$, $x, {y} \in U$ and 
 $\tilde{x} \in \Omega$  with  $T^n(\tilde{x})=x$, there exists a unique $\tilde{y}\in \Omega$  with $T^n(\tilde{y})={y}$ and 
 $d(T^k(\tilde{x}), (T^k(\tilde{y})) \leq  \lambda^{n-k} d(x,y)$ for $k = 0,\ldots, n$. We will refer to $\tilde{y}$ as the $n$-the preimage of $y$ adapted to $\tilde{x}$ and $U$ and set   
\[T^{-n}_{\tilde{x}}: U \to \Omega, \quad    T^{-n}_{\tilde{x}} (y) := \tilde{y} .\] 
In order to parametrize the set of all inverse branches, fix $x_U \in U$, define
\[W^*_U := \{(x,n) \in \Omega \times \mathbb{N} : T^n(x) = x_U \} \]
and set $l(w) := n$ and $\tau_w = T_x^{-n}$,  for $w = (x,n) \in W^*$.  

The construction of the spectral triple starts 
with the $C^*$-algebra $A = C(\Omega,\mathbb{C})$, the space of all continuous complex valued functions defined on $\Omega$
and the Hilbert space 
\[ H_U = \ell^2(W^*_U)\oplus \ell^2(W_U^*),\]
for some fixed $U \in \mathcal{U}$. 
{We} will write a generic element in $H$ as
\[
\bigoplus_{w \in W^*_U}  \binom{\epsilon_1(w) }{\epsilon_2(w) },
\]
where $\epsilon_i: W^{*}_U\to \mathbb{C}$
are complex valued functions with 
$\sum_{w \in W^*_U} |\epsilon_i(w)|^2 <\infty$.
We now fix two arbitrary elements $x, y \in U$.
For $a\in A$, the operator $L_a: H \to H$ is defined by
\[
L_a  (
	\bigoplus_{w \in W^*}  \binom{\epsilon_1(w) }{\epsilon_2(w) }
 )
=
\bigoplus_{w \in W^*}
\binom{a( \tau_w(x))\, \epsilon_1(w) }{a(\tau_w(y))\,\epsilon_2(w) }.
\]
To construct the Dirac operator $D_s$, we fix a continuous  function $J:\Omega \to (0,\infty)$, $ s\geq 0$ and set, for $w \in W^*_U$,
\[\mathbb{J}^s_w(x) := \prod_{k=0}^{l(w)-1}  \left(J(T^{k}(\tau_w(x)))\right)^s .\]
The operator $D_s$ is then defined by
\begin{align} \label{eq:Definition of D_s}
D_s :&\,   \hbox{Dom}(D_s) := \left\{ \bigoplus_{w \in W^*}  \binom{\epsilon_1(w) }{\epsilon_2(w)} : \sum_w \frac{|\epsilon_i(w)|^2}{\mathbb{J}^s_w(x)^2}  < \infty  \right\} \to H_U, \\
\nonumber
& \,
 \bigoplus_{w \in W^*}  \binom{\epsilon_1(w) }{\epsilon_2(w) }  \mapsto 
\bigoplus_{w \in W^*} \frac{1}{\mathbb{J}^s_w(x)} \,\binom{\epsilon_2(w)}{\epsilon_1(w)}.
\end{align}
It then follows immediately from the definitions that 
$\hbox{Dom}(D_s)$ is dense in $H_U$ and that $\langle D_s x,y\rangle = \langle x, D_s y\rangle$ for all $x,y \in \hbox{Dom}(D_s)$. That is, $D_s$ is symmetric. By choosing an orthonormal basis of $H_U$, it is then easy to see that $\hbox{Dom}(D_s) = \hbox{Dom}(D_s^\ast)$, implying that $D_s$ is self-adjoint.
Moreover,  
\[ [D_s,L_a] ( \bigoplus_{w \in W^*}  \binom{\epsilon_1(w) }{\epsilon_2(w) } )
=
\bigoplus_{w \in W^*}
\frac{ a(\tau_w (x))- a(\tau_w (y)) }{\mathbb{J}^s_w(x)}
\binom{\,- \epsilon_2(w) }{\,\epsilon_1(w) }.
\]
 
{
\begin{proposition} \label{prop:spectral-triple}
Assume that $T$ is Ruelle expanding and that $J:\Omega \to (0,\infty)$ is a continuous function with  
\[\lim_{n \to \infty} \sup \{\|\mathbb{J}_w \|_\infty:  \ell(w) =n\} =0,\]
and that $x,y \in U$, $U \in \mathcal{U}$. Then there exists $s_0> 0$ such that $(H_U,A,D_s)$ is a spectral triple for $0 < s <  s_0$.
\end{proposition}

\begin{proof} It is only left to show that $D_s$ has a compact resolvent and that there exists $s_0 > 0$ such that   $\{a\in A: \|[D,L_a]\|<+\infty \}$ is dense in $A$ for $0 \leq s < s_0$.

In order to do so, observe that for $t \in \mathbb{C}$, the product structure of $D_s$ implies that $D_s - t \hbox{id}$ is invertible if and only if the the linear map $A_w: (z_1,z_2) \mapsto \mathbb{J}^{-s}_w(x)(z_2,z_1) - t (z_1,z_2)$ is invertible for all $w \in W_U^\ast$. Or equivalently,
 $t^2 \neq \mathbb{J}^{-2s}_w(x)$ for all $w \in W_U^\ast$ and $A_w$ has eigenvalues  $\pm \mathbb{J}^{-s}_w(x) - t \neq 0$.  As $\mathbb{J}^{s}_w(x)  \to 0$ as $\ell(w) \to \infty$, the eigenvalues of $A_w^{-1}$ tend to zero as $\ell(w) \to \infty$. This implies that $(D_s - t \, \hbox{id})^{-1}$ is compact for $t$ in the resolvent set of $D_s$. 
 
It remains to analyse $[D_s,L_a]$. In order to do so, choose $n_0 \in \mathbb{N}$ such that $\| \mathbb{J}_w \|_\infty < 1$ for all $\omega$ with $\ell(\omega) \geq n_0$. Furthermore, set $C_1:= \sup \{  \|(\mathbb{J}_\omega)^{-1}\|_\infty  :  \ell(\omega) \} < n_0$ and $C_2:= \inf \{  \|(\mathbb{J}_\omega)^{-1}\|_\infty  :  \ell(\omega) =  n_0  \}$. Now observe that $C_1 < \infty$ as $J$ is continuous and that $C_1 > 1$ by construction.
In particular, if $a$ is Lipschitz continuous and $\ell(\omega) = p n_0 + q$ for $p,q \in \mathbb{N}\cup \{0\}$ and $q  < n_0$, then
 \[
 \frac{| a(\tau_w (x))- a(\tau_w (y))| }{\mathbb{J}^s_w(x)}
 \leq C_1 C_2^{ps}  \hbox{Lip}(a)  d(\tau_w (x),\tau_w (y)) \leq C_1 C_2^{ps} \hbox{Lip}(a) \lambda^{\ell (w)} \leq C_1 \hbox{Lip}(a) (C_2^s\lambda^{n_0})^p.
 \]
 In particular, if $0 \leq s < s_0 := - n_0 \log \lambda / \log C_2$, then $[D_s,L_a]$ is a bounded operator. The remaining assertion follows from the fact that the Lipschitz continuous functions are dense in $C(\Omega)$.
 \end{proof}
}

\begin{remark} \label{remark:s_0}
Even though the proof of Proposition \ref{prop:spectral-triple} is elementary, the result gives rise to interesting observations with respect to the dimension of the ambient space $\Omega$ and the relation of $D_s$ to the usual derivative. Firstly, if $\Omega \subset \{1, 2, \ldots k\}^\mathbb{N}$ is a subshift of finite type with metric 
\[
d((x_i),(y_i)) := \lambda^{ \text{ { $\min \{i : x_i \neq y_i\}$ } }   },
\]
then $\mathcal{U}$ consists of the cylinders of length one and the expansion rate of the shift $T$ is equal to $\lambda$. In particular, a change of metric corresponds to a change of $\lambda$ and a brief analysis of the above proof shows that $(H_U,A,D_s)$ is a spectral triple for any $s > 0$.

On the other hand, if $\Omega$ is not zero-dimensional, then $s_0$ plays an essential rôle as indicated in the following basic  example. Let $T: \mathbb{T}^2 \to \mathbb{T}^2$, $(x,y) \mapsto (2x \mod 1, 2y \mod 1)$ and $J = 1/4$. Then $J$ is normalized 
and $d(\tau_w(x),\tau_w(y))/\mathbb{J}^s_w(x) = 2^{ (2s - 1)|w|}  d(x,y)$. Hence, this quotient is uniformly bounded if and only if $s \in [0,1/2]$. Moreover, 
 if $a \in C^1(\mathbb{T})$ and $h_w := \tau_w(x) - \tau_w(y)$, then 
 \begin{align*} 
 \frac{a(\tau_w(x))- a(\tau_w(y))}{\mathbb{J}_w(x)}
 & = 
 2^{ (2s - 1)|w|} d(x,y) 
 \left( a'(\tau_w(x))\left( h_w/|h_w| \right)
+ o(1)  \right)\\
& = 2^{ (2s - 1)|w|}  
 \left( a'(\tau_w(x))\left( x-y \right)
+ d(x,y) o(1)  \right)
,\end{align*}
where $a'$ denotes the usual derivative. In particular, if $s = 1/2$, this establishes a connection between $[D_s,L_a]$ and $a'$. Moreover, note that $1/2$ is the   natural parameter in this example: the equilibrium state associated to $J$ is the two-dimensional {Lebesgue} measure whereas $J^{1/2}$ corresponds to the arclength.

Now assume that, after identifying $\mathbb{T}^2$ with $[0,1]^2/_\sim$, that $x  - y$ is colinear to $(1,0)$ and that  $a:\mathbb{T}^2 \to \mathbb{R}$ is of the form $(z_1,z_2) \mapsto f(z_2)$, where $f : [0,1] \to \mathbb{R}$ is continuous and $f(0)=f(1)$. Then $a \in A$, $[D_s,L_a] = 0$ and there is no uniform bound on $\max a - \min a = \max f - \min f$. Hence, the image of 
\[ \{ a \in A: \|[D_s,L_a]\| \leq 1  \}  \]
in $A/z\mathbf{1}$ is unbounded and, therefore, a result due to Rieffel and Pavlović implies that Connes' pseudo-metric is not a metric (for a brief exposition of Connes' pseudo-metric, see \cite[Section 2.1]{MR3084488}). However, by considering a finite direct sum of the spectral triple, this problem can be resolved as shown in Theorem \ref{theo:spectral-metric} below.    
\end{remark}

\subsection{Spectral triples and their spectral metrics} \label{subsec:spectral-metric}
In this section, we modify the above construction in order to obtain a spectral triple such that the topology of $X$ can be recovered by Connes' pseudo-metric. In order to do so, we assume that $T$ is a Ruelle expansive map such that the associated open cover $\mathcal{U}$ has a certain overlap, that is, for each $U \in  \mathcal{U}$, there exists $V \in   \mathcal{U}$ with $U \neq V$ and $U \cap V \neq 0$.  
Furthermore, we now fix $J: \Omega \to (0,1)$ and assume that there exists a finite set  $R \subset \bigcup_{U \neq V} U\cap  V$ such that the following conditions are satisfied.
\begin{itemize}
 \item[(C0)] $R \cap U \neq \emptyset $ for all  $U \in  \mathcal{U}$.
 \item[(C1)] There exists $s> 0$ and $C_1>0$ such that for all $x,y \in U \cap R$, $w \in W_U^\ast$ and  $U \in \mathcal{U}$,
\[ \frac{1}{C_1}< \frac{d(\tau_w(x),\tau_w(y))}{\mathbb{J}^s_w(x)} < C_1\]
\item[(C2)] There exists $C_2$ such that, for all $x,y\in R$ and  $u,v \in W^n$, there exist $k \in \mathbb{N}$, $z_1, \ldots z_k \in R$, $U_1,\ldots U_k$ and $w_1 \in W^n_{U_1}, \ldots w_k \in W^n_{U_k}$ with $z_1 = x$, $w_1 = u$  and  
$z_k = y$, $w_k = v$ such that 
\begin{enumerate}
 \item for each $j = 1, \ldots k-1$, 
 $\tau_{w_{j+1}}(z_{j+1}) \in \tau_{w_j}(U_j)$,
 \item $\sum_{j=1}^{k-1} d(\tau_{w_{j}}(z_{j}), \tau_{w_{j+1}}(z_{j+1})) \leq C_2 d(\tau_u(x),\tau_v(y))$.
\end{enumerate}
\end{itemize}
The motivation behind these conditions is the following. Assume that $\Omega$ is pathwise connected, that $\gamma$ is a curve in $\Omega$ and that $n \in \mathbb{N}$. Then any cover of $\gamma$ by elements of $T^{-n}(\mathcal{U})$ gives rise to a sequence of points in $T^{-n}(R)$ such that any two neighbouring points are in $\tau_\omega(U)$ for some $U$ and $\omega$. 
{This allows us to} identify those $a \in C(\Omega)$ such that  $[D^s,L_a]$ extends to a bounded operator.

The spectral triple under consideration is now a finite sum of those in Proposition \ref{prop:spectral-triple}, provided that $\mathcal{U}$ and $R$ are as above.  

\begin{definition} For $x,y \in U \in \mathcal{U}$ and $s> 0$, let $D_s^{x,y}$ refer to the Dirac operator as in \eqref{eq:Definition of D_s}. We then refer to $(H_U,A,D^s_{x,y})$ as the partial spectral triple, to $D_s^{x,y}$ as the partial Dirac operator and to 
\[ (H,A,\mathcal{D}_s) := 
\left(\bigoplus_{U \in \mathcal{U}} \bigoplus_{x,y \in R \cap U, x \neq y} H_U,A, \bigoplus_{U \in \mathcal{U}} \bigoplus_{x,y \in R \cap U, x \neq y} D_s^{x,y}\right) \]
as the global spectral triple.  
\end{definition}

With (C0) to (C2) at hand, it is then possible to relate Connes' pseudo-metric and the Wasserstein distance $W$ as defined in Definition \ref{def:wasserstein}. For a spectral triple $(H,A,D)$ and states $p,q$ of $A$, the Connes (pseudo-)metric is defined by
\begin{align*}
d_{\hbox{\tiny spec} }(p,q) := 
\sup \left\{
 p(a)-q(a)  : a \in A, [D,L_a] \hbox{ densely def.,} \|[D,L_a]\|\leq 1
\right\}.
\end{align*}
However, if $A$ is the $C^\ast$-algebra of continuous functions on a compact metric space, then the set of states of $A$ coincides with the set of Borel probability measures by the Riesz representation theorem. Also note that in this situation, the set of Lipschitz continuous functions is dense in $A$ and that it follows from Kantorovich's duality that 
\begin{align*}
W(p,q) := 
\sup \left\{ \textstyle 
 \int a dp - \int a dq : a \in A, \hbox{Lip}(a)\leq 1 
\right\}.
\end{align*}
Furthermore, observe that the metric on $X$ can be recovered through $d(x,y) = W(\delta_x,\delta_y)$, where $\delta_x,\delta_y$ refer to the Dirac measures in $x$ and $y$, respectively. 

\begin{theorem} \label{theo:spectral-metric} Assume that $T:X \to X$ is a Ruelle expansive map, $X$ is a compact, metric space and $J : X \to (0,\infty)$ is continuous. If (C1) holds, then 
$(H,A,\mathcal{D}_{s_0})$ is a spectral triple. If (C1) and (C2) hold, then
\[ C_1^{-1}  W  \leq  d_{\hbox{\tiny spec}}  \leq C_1C_2 W.\]
\end{theorem}

\begin{proof} It follows from (C1) and the expansiveness of $T$ that 
\[ \mathbb{J}^s_w(x)  \leq C_1 d(\tau_w(x),\tau_w(y))  \to 0  \hbox{ as } \ell(w)\to \infty.\]
It follows from this{,} as in the proof of  Proposition \ref{prop:spectral-triple}{,}
that $\mathcal{D}_{s_0}$ has a compact resolvent. 
So it remains to analyse $[\mathcal{D}_{s_0},L_a]$. 
As it easily can be seen,  
\[ \| [\mathcal{D}_s,L_a]\|  = \sup \left\{ \frac{|a(\tau_w(x)) - a(\tau_w(y))|}{\mathbb{J}^s_w(x)} : x,y \in U \cap R, w \in W^\ast, U \in \mathcal{U} \right\}.   \]
Hence, if $a$ is Lipschitz continuous then (C1) implies that for all $s \in(0, s_0]$,  $ x,y \in U \cap R $, $w \in W^\ast_U$ and $U \in \mathcal{U}$,  
\[\frac{|a(\tau_w(x)) - a(\tau_w(y))|}{\mathbb{J}^s_w(x)} 
\leq  \hbox{Lip}(a) \frac{d(\tau_w(x),\tau_w(y))}{\mathbb{J}^s_w(x)} 
\leq C_1 \hbox{Lip}(a),\]
where $\hbox{Lip}(a)$ refers to the Lipschitz constant of $a$. Hence, $\| [\mathcal{D}_s,L_a]\| \leq C_1 \hbox{Lip}(a)$. This proves that 
$(H,A,\mathcal{D}_{s_0})$ is a spectral triple. 

On the other hand, if $a \in C(\Omega)$ is such that $[\mathcal{D}_s,L_a]$ is a bounded operator, then  
\begin{equation}\label{eq:spectral-radius-of-[DL]}
|a(\tau_w(x)) - a(\tau_w(y))| \leq   \|[D^s,L_a] \| \mathbb{J}^s_w(x)
\end{equation}
for all $x,y \in U \cap R$, and $w \in W^\ast_U$ and $U \in \mathcal{U}$. So assume that $x,y \in \Omega$ and that $\epsilon > 0$. By uniform continuity of $a$, there exists $\delta >0$ such that $d(z,\tilde{z})< \delta$ implies $|a(z)-a(\tilde{z})|< \epsilon$.  
Now choose $n \in \mathcal{N}$ such that $\diam(\tau_w(U)) \leq \delta$ for all $w$ with $\ell(w) =n$. It then follows from condition (C2) there are  $z_1, \ldots z_k$ and $w_1, \ldots w_k$ satisfying the conditions in (C2) and such that $d(x,z_1), d(z_k,y) < \delta$. Hence, 
\begin{align*}
 d(x,y) & \geq  d(\tau_{u_1}(z_1),\tau_{u_k}(z_k)) - 2 \delta 
 \geq \frac{1}{C_2}\sum_{j=1}^{k-1} d(\tau_{w_{j}}(z_{j}), \tau_{w_{j+1}}(z_{j+1}))  - 2 \delta\\
 & \geq \frac{1}{C_1C_2}\sum_{j=1}^{k-1} \mathbb{J}^{s_0}_{w_{j}}(z_{j})   - 2 \delta
 {.}
\end{align*}
The estimate combined with \eqref{eq:spectral-radius-of-[DL]} then implies that 
\begin{align*}
| a(x) -  a(y)| 
\leq &|a(x) - a(\tau_{u_1}(z_1))| + 
\sum_{j=1}^{k-1} |a(\tau_{w_{j}}(z_{j}))- a(\tau_{w_{j+1}}(z_{j+1}))
 |  \\ & + | a(\tau_{u_1}(z_1)) - a(y)| \\
\leq &  \|[\mathcal{D}_s,L_a] \| \sum_{j=1}^{k-1} \mathbb{J}^{s_0}_{w_{j}}(z_{j}) + 2 \epsilon \leq C_1 C_2 (d(x,y) + 2\delta)    + 2 \epsilon
.
\end{align*}
As $\epsilon$ can be arbitrarily chosen, we obtain that $\hbox{Lip}(a) \leq C_1C_2$. 
\end{proof}
 
Note that the above theorem is applicable to the second example in Remark \ref{remark:s_0}. 
Namely it suffices to define $\mathcal{U}$ as the set open squares of the form $(x,x+ 1/2) \times (y, y + 1/2)$ for $x, y \in \{0,1/4,1/2,3/4\}$ and $R = \{(k/8, l/8): k,l = 1,3,5,7\}$. For a more interesting example, we refer to section \ref{subsec-sierpinski}.

\subsection{Subshifts of finite type} \label{subsec:sft}
The motivation of this section is to obtain an analogous result for totally disconnected spaces. The applications we have in mind are conformal graph directed Markov systems (see \cite{Mauldin-Urbanski--Graph-Directed-Markov-Systems--2003}) with a totally disconnected limit set. {For ease of exposition,} we only formulate the result in terms of shift spaces.
Therefore, we now recall the definition of a topologically mixing subshift of finite type.
\begin{definition}
Assume that $\mathcal{A} = \{1, \ldots, k\}$ for some $k> 1$, $B= (a_{ij})\in \{ 0,1\}_{k\times k}$ and set
\[\Omega   :=  \{(x_i:i \in \N \cup \{0\}) : x_n \in \mathcal{A}, \mathbf{1}_{i}(x_n)\mathbf{1}_{j}(x_{n+1}) \leq a_{ij} \, \forall i,j \in \mathcal{A}, n\in \N \cup \{0\} \}.
\]
We then say that $(\Omega,T)$ is a \emph{subshift of finite type}, where $T$ is the shift map $T(x_1,x_2,\ldots ):= (x_2,x_3,\ldots )$. In this setting, we refer to $\mathcal{A}$ as the \emph{alphabet} or the \emph{set of states}. 
Moreover, we say that the subshift of finite type is \emph{aperiodic} if there exists  $n_0 \in \N$  such that all coordinate of $B^{n_0}$ are strictly positive.    
\end{definition}
Recall that, as it is well known, $T$ is topologically mixing if and only if $A$ is aperiodic. Furthermore, $d((x_i: i \geq 0),(y_i :i \geq 0)) := 2^{- \min\{i: x_i \neq y_i\}}$ defines a metric on $\Omega$ such that $(\Omega,d)$ is compact and totally disconnected and $T: \Omega \to \Omega$ is Ruelle expanding with parameter $1/2$. 

We now fix $J: \Omega \to (0,\infty)$ such that $g:= \log J$ is $\omega_{0, \beta \textrm{log}}$-Hölder for some  $\beta > 1$ and that $\rho_g=1$, with $\rho_g$ given by Proposition \ref{prop:eigenfunction}. We now use the results in Section \ref{sec-preliminares} in order to define a new metric. That is, with $m_g$ given by Corollary \ref{cor:decay-to-conformal-measure}, assume that  $x,y \in \Omega$ with $d(x,y) = 2^{-k}< 1$. That is, the first $k$ coordinates of $x$ and $y$ coincide whereas the $(k+1)$-th are different. In particular, 
\[ d_{m_g}(x,y) := m_g(\{ (z_i) : z_i = x_i \hbox{ for }i = 0,1, \ldots,k-1  \}) = m_g(\{ z: d(x,z) \leq 2^{-k}\}) \] 
satisfies $d_{m_g}(x,y) = d_{m_g}(y,x)$. Furthermore, if $d(x,y) =1$, set $d_{m_g}(x,y) =1$. It now follows immediately from the definition that $d_{m_g}$ satisfies the triangle inequality. A further important ingredient is \eqref{eq:distortion-expanding} which implies that $g_k(x) \asymp g_k(y)$ whenever $d(x,y) \leq 2^{-k}$. Hence, as $m_g$ is a conformal measure, it follows that 
\[ m_g( T^k (A)) = \int_A \frac{1}{g_k(x)} dm_g  \]
for any $A$ such that $T^k$ is bimeasurable on $A$. In particular, \eqref{eq:distortion-expanding} implies that 
\begin{equation} \label{eq:new-metric-shift}
m_g(\{ z: d(x,z) \leq 2^{-k}\}) \asymp e^{g_k(x)}.
\end{equation}
Finally, as $\mathcal{L}^{n_0}_{\overline{g}}(\mathbf{1}) = \mathbf{1}$, it follows that $\sup \overline{g}_{n_0} < 0$. In particular, $\sup g_k \to - \infty$, which implies that $d(x,y) = 0 $ if and only if $x=y$. Hence, $d_{m_g}$ is a metric. With respect to this metric, the following holds. 

\begin{theorem} \label{theo:spectral-metric-sft}
 Assume that $(\Omega,T)$ is a topologically mixing subshift of finite type, that $\log J$ is a $\omega_{0, \beta \textrm{log}}$-Hölder continuous potential for some $\beta > 1$ and that $\rho_{\log J}=1$. Moreover, assume that $R \subset \Omega$ is a finite set such each cylinder of length (i.e. ball of radius 1/2) contains precisely two elements.       
Then $(H,A,\mathcal{D}_{1})$ is a spectral triple and $ d_{\hbox{\tiny spec}} \asymp W_{d_{m_g}}$,  where  $W_{d_{m_g}}$ refers to the Wasserstein distance with respect to $d_{m_g}$.
\end{theorem}

\begin{proof} Observe that \eqref{eq:new-metric-shift} implies Condition (C1) of Theorem \ref{theo:spectral-metric}. Hence, 
 $(H,A,\mathcal{D}_{1})$ is a spectral triple. We omit the proof of the remaining statement as it is significantly simpler than the one in Theorem \ref{theo:spectral-metric}. 
\end{proof}

\subsection{A spectral triple associated to the Sierpiński gasket}
\label{subsec-sierpinski}

We now give an example of a spectral triple associated to a connected fractal set, for which Connes' pseudometric in fact is a metric, in contrast to the example given in Remark \ref{remark:s_0}. 

Instead of working directly with the Sierpiński gasket, we consider 4 copies of this well known fractal, placed on the regular octohedron as illustrated in Figure \ref{fig:sierpinski}. The reason for this procedure stems from the fact that we want to define the Sierpiński gasket as a repeller of a Ruelle expanding map which only is possible after taking care of the three extremal points of the standard Sierpiński gasket.
\begin{figure}[htb]
\centering
\def\svgwidth{0.8\textwidth}
\begingroup%
  \makeatletter%
  \providecommand\color[2][]{%
    \errmessage{(Inkscape) Color is used for the text in Inkscape, but the package 'color.sty' is not loaded}%
    \renewcommand\color[2][]{}%
  }%
  \providecommand\transparent[1]{%
    \errmessage{(Inkscape) Transparency is used (non-zero) for the text in Inkscape, but the package 'transparent.sty' is not loaded}%
    \renewcommand\transparent[1]{}%
  }%
  \providecommand\rotatebox[2]{#2}%
  \newcommand*\fsize{\dimexpr\f@size pt\relax}%
  \newcommand*\lineheight[1]{\fontsize{\fsize}{#1\fsize}\selectfont}%
  \ifx\svgwidth\undefined%
    \setlength{\unitlength}{370.60513559bp}%
    \ifx\svgscale\undefined%
      \relax%
    \else%
      \setlength{\unitlength}{\unitlength * \real{\svgscale}}%
    \fi%
  \else%
    \setlength{\unitlength}{\svgwidth}%
  \fi%
  \global\let\svgwidth\undefined%
  \global\let\svgscale\undefined%
  \makeatother%
  \begin{picture}(1,0.67942116)%
    \lineheight{1}%
    \setlength\tabcolsep{0pt}%
    \put(0,0){\includegraphics[width=\unitlength,page=1]{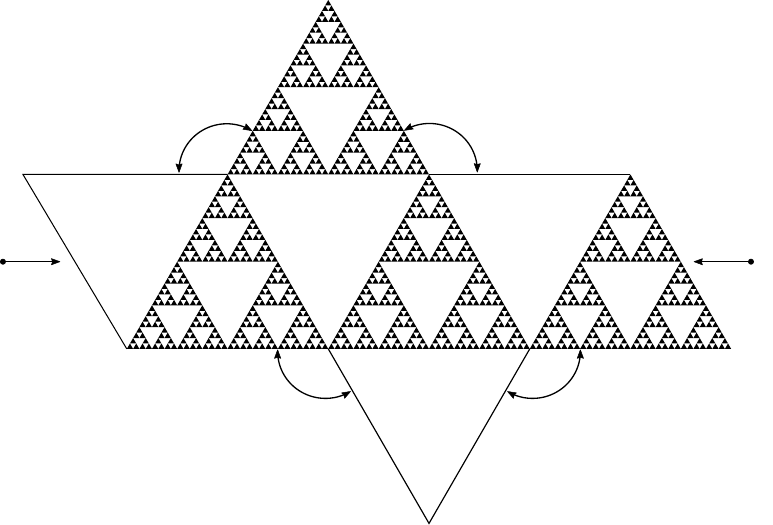}}%
    \put(0.27768164,0.15638876){\makebox(0,0)[lt]{\lineheight{1.25}\smash{\begin{tabular}[t]{l}$\triangle_2$\end{tabular}}}}%
    \put(0.54076482,0.15638876){\makebox(0,0)[lt]{\lineheight{1.25}\smash{\begin{tabular}[t]{l}$\triangle_3$\end{tabular}}}}%
    \put(0.79980063,0.15638876){\makebox(0,0)[lt]{\lineheight{1.25}\smash{\begin{tabular}[t]{l}$\triangle_4$\end{tabular}}}}%
    \put(0.54076482,0.56922706){\makebox(0,0)[lt]{\lineheight{1.25}\smash{\begin{tabular}[t]{l}$\triangle_1$\end{tabular}}}}%
    \put(0.84404425,0.42977003){\makebox(0,0)[lt]{\lineheight{1.25}\smash{\begin{tabular}[t]{l}$x_1$\end{tabular}}}}%
    \put(0.58312232,0.42976193){\makebox(0,0)[lt]{\lineheight{1.25}\smash{\begin{tabular}[t]{l}$x_2$\end{tabular}}}}%
    \put(0.32628019,0.42976193){\makebox(0,0)[lt]{\lineheight{1.25}\smash{\begin{tabular}[t]{l}$x_3$\end{tabular}}}}%
    \put(0.95737241,0.2031137){\makebox(0,0)[lt]{\lineheight{1.25}\smash{\begin{tabular}[t]{l}$x_4$\end{tabular}}}}%
    \put(0.69645049,0.20310558){\makebox(0,0)[lt]{\lineheight{1.25}\smash{\begin{tabular}[t]{l}$x_6$\end{tabular}}}}%
    \put(0.45175066,0.20310558){\makebox(0,0)[lt]{\lineheight{1.25}\smash{\begin{tabular}[t]{l}$x_5$\end{tabular}}}}%
    \put(0.19052118,0.20310558){\makebox(0,0)[lt]{\lineheight{1.25}\smash{\begin{tabular}[t]{l}$x_4$\end{tabular}}}}%
    \put(0.45267618,0.65574368){\makebox(0,0)[lt]{\lineheight{1.25}\smash{\begin{tabular}[t]{l}$x_1$\end{tabular}}}}%
  \end{picture}%
\endgroup%
\caption{An expanding map on the union of 4 Sierpiński gaskets}
\label{fig:sierpinski}
\end{figure}
That is, with $\triangle_1$ to $\triangle_4$ referring to the copies of the standard Sierpiński gasket and endpoints $x_1$ to $x_6$ as indicated in Figure \ref{fig:sierpinski}, we define $T : X\to X$, where $X:= \bigcup_{i=1}^4 \triangle_i$, to be the piecewise, orientating preserving similarity which expands distances by 2 such that
\[ 
x_1 \mapsto x_3,\; x_2 \mapsto x_2,\; x_3 \mapsto x_1,\; 
x_4 \mapsto x_4,\; x_5 \mapsto x_6,\; x_6 \mapsto x_5. 
\]
In terms of the action on $\triangle_i$, this means that
\begin{align*}
 \textstyle \triangle_1 \mapsto \bigcup_{i \neq 1} \triangle_i, \; \triangle_2 \mapsto \bigcup_{i \neq 4} \triangle_i, \; \triangle_3 \mapsto \bigcup_{i \neq 3} \triangle_i, \; \triangle_4 \mapsto \bigcup_{i \neq 2} \triangle_i.
\end{align*}
It is now straightforward to check that, that $T$ is continuous and that each point has three preimages. We now show how to apply Theorem \ref{theo:spectral-metric}. Let $d$ refer to the geodesic metric {on} $X$, that is the distance between two points is the Euclidean length of the shortest path between these points (as in, e.g., \cite{Christensen-Ivan-Lapidus--Dirac-Operators-And-Spectral--AM2008}). Furthermore, let $U_i$ be the $\epsilon$-neighbourhood of $\triangle_i$ for $\epsilon$ sufficiently small such that $T|_{U_i}$ is a homeomorphism. In particular, the connected components of $U_i \cap U_j$ are then open balls of radius $\epsilon$ whose centers are elements of $R = \{x_1,\ldots x_6\}$.

We now consider the constant function $ J = 1/3$ and the action of $\mathcal{L}_{\log J}$ on $C(X)$. As each $x\in X$ has three preimages, we have that  $\mathcal{L}_{\log J}(\mathbf{1}) = \mathbf{1}$. Moreover, as $\tau_w$ contracts by $2^{-\ell(w)}$ and $J$ is constant, it follows that 
\[ d(\tau_w(x) , \tau_w(y) ) = d(x,y) 2^{-\ell(w)} = d(x,y) \mathbb{J}^{\log 2 / \log 3}_w(x) \] 
for any $x,y \in R \cap U_i$ and $w \in W_{U_i}^\ast$ and $i=1, \ldots 4$. Hence, hypothesis (C1) is satisfied with respect to $s = \log 2 / \log 3$. In order to verify (C2), we make use of the fact that the Sierpiński gasket is a geodesic space.
That is, the distance between two points is defined as infimum over all path lengths of paths from $x$ to $y$ and this infimum is realized by a curve, referred to as geodesic, from $x$ to $y$. In particular, , for $x,y \in T^{-n}(R)$, there exists a curve in $X$ from $x$ to $y$ which realizes the (geodesic) distance between $x$ and $y$. 
Moreover, this curve has subsequent visits to elements of $T^{-n}(R)$. This then defines a sequence $z_1,\ldots z_k \in  T^{-n}(R)$ and $w_1, \ldots w_k \in W^n$  such that in fact
\[  d(x,y) =  \sum_{j=1}^{k-1} d(\tau_{w_{j}}(z_{j}), \tau_{w_{j+1}}(z_{j+1})).\]         
Hence, Theorem \ref{theo:spectral-metric} is applicable and one obtains that the spectral metric associated to the global spectral triple is comparable to the Wasserstein distance.  

\begin{remark} \label{remark:analysis on fractals}
{As pointed out by Connes, a spectral triple establishes a differential calculus on the underlying space by considering $[D_s,L_a]$ as a replacement for the gradient of $a$. However, in the context of the Sierpiński gasket, a differential calculus historically was first established through the construction of Laplacians and Dirichlet forms as a by-product of the construction of the Brownian motion on the gasket.

The first breakthrough with respect to a detailed analysis of a diffusion is due to Barlow and Perkins in \cite{Barlow-Perkins--Brownian-Motion-On-The--PTRF-1988}, who refined the known construction as a limit of a simple random walk on rescaled copies of the gasket such that they were able to provide precise asymptotics for the tails of the transition probabilities in terms of the Hausdorff dimension and the dimension of the walk. Thereafter, the community started to analyse the concept behind their work in terms of Dirichlet forms, which is the natural choice in the context of diffusions. It turned out that this Dirichlet form can be constructed as a the limit of finite dimensional Dirichlet forms, and that the limit satisfies a fixed point equation, which lead to elegant criteria for the uniqueness of the diffusion by Sabot in  \cite{Sabot--Existence-And-Uniqueness-Of--ASENS4-1997}. In particular, this justified to speak about the Brownian motion on Sierpiński's gasket.

On the other hand, Kigami and Lapidus already started in the early 1990ies to work towards a differential calculus on fractals by analysing Weyl's formula for the Laplacian in \cite{Kigami-Lapidus--Weyl-S-Problem-For--CMP-1993} and its relation to non-commutative geometry in terms of the representation of the volume as a Dixmier trace (\cite{Lapidus--Analysis-On-Fractals-Laplacians--TMNA-1994,Kigami-Lapidus--Self-Similarity-Of-Volume--CMP-2001}) and spectral triples (\cite{Christensen-Ivan-Lapidus--Dirac-Operators-And-Spectral--AM2008}).

We now give a brief comparison of the constructions of Christensen, Ivan and Lapidus in \cite{Christensen-Ivan-Lapidus--Dirac-Operators-And-Spectral--AM2008} and of Cipriani, Guido, Isola and Sauvageot in \cite{Cipriani-Guido-Isola--Spectral-Triples-For-The--JFA2014} with the one in here. In \cite{Christensen-Ivan-Lapidus--Dirac-Operators-And-Spectral--AM2008}, the authors start with the construction of a spectral triple on the unit interval based on the observation that the usual derivative acts as a multiplication operator on a Fourier series. After that, the spectral triple on the gasket is obtained as an infinite direct product of scaled copies of this initial triple. In \cite{Cipriani-Guido-Isola--Spectral-Triples-For-The--JFA2014}, the focus of the authors is on the holes or \emph{lacunas} of the gasket. That is, they first construct a spectral triple for quasicircles and then again consider an infinite direct product of these triples.

The approach in here is different as it is based on a difference operator instead of a derivative (or multiplication operator in terms of Fourier series). Hence, it is natural to work here with Lipschitz functions instead of square integrable functions. Due to Kantorovich's duality, the proof that the spectral metric is equivalent to the ambient metric becomes elementary (cf. Theorems \ref{theo:spectral-metric} and \ref{theo:spectral-metric-sft}). Furthermore, it follows as in \cite{Christensen-Ivan-Lapidus--Dirac-Operators-And-Spectral--AM2008} that the Hausdorff measure can be recovered by identifying the Dixmier trace of $L_a|D_0|^{-1}$ with the integral of $a$ (see Corollaries \ref{cor:Dixmier-negative-potential} and \ref{cor:Dixmier-normalized-potential} below). It is also worth pointing out that Cipriani, Guido, Isola and Sauvageot show how to reconstruct the standard Dirichlet form on the Sierpiński from their spectral triple (Corollary 4.24 in \cite{Cipriani-Guido-Isola--Spectral-Triples-For-The--JFA2014}).
}
\end{remark}

\section{Dixmier Trace Representation of Gibbs Measures}
\label{sec-zeta-dixmier}

We now analyse spectral triples for a given function $J: \Omega \to (0,1)$ of higher regularity and determine the Dixmier trace of $L_a|D_s|^{-1}$, with $L_a$ and $D_s$ as defined in \eqref{eq:Definition of D_s}. In here, we do not necessary require that $L_a$ and $D_s$ belong to a spectral triple, which allows to include the following combinations of Ruelle expansive maps and functions $\log J$ of lower regularity. 

\begin{enumerate}
\item[(E)] Assume that $T:\Omega \to \Omega$ is topologically mixing. Furthermore, assume that $T$ is either Ruelle expanding and  $\log J$ is $\omega_{0,\beta \textrm{log}}$-Hölder continuous for some $\beta > 1$, or  $T$ is Ruelle expansive with contraction rate $t(1 + nt^{-1/\beta} + n)^{-\beta}$  and $\log J$ is $\omega_{\alpha,0\textrm{log}}$-Hölder with $\alpha\beta > 1$, respectively.  
\end{enumerate}

Note that it follows from the results in Section \ref{sec-preliminares} that under these conditions the pressure function
\[ P(s) := \lim_{n \to \infty} \frac{1}{n} \log \mathcal{L}^n_{s \log J}(\mathbf{1})(z_0),\] 
is well defined and does not depend on $z_0 \in \Omega$. Moreover, if $P(s)$ is finite, then $\rho_s := \exp(P(s))$ is the spectral radius of the action of $L_{s\log J}$ on the space of $\omega^\dagger$-Hölder continuous functions with $\omega^\dagger = \omega_{0,(\beta -1)\textrm{log}}$ or $\omega^\dagger = \omega_{\alpha - 1/\beta,\textrm{log}}$, respectively.  
 
Now assume that  $h_s$, $\mu_s$ and $m_s$ are constructed as in \eqref{eq:definition-h} and Corollaries \ref{cor:decay-to-invariante-measure} and \ref{cor:decay-to-conformal-measure} for the potential $ s \log J$. 
As $T$ is locally invertible and the preimages of the covering $\mathcal{U}$ generate the Borel $\sigma$-algebra, it follows from the Rokhlin formula for the  Kolmogorov-Sinai entropy $\mathfrak{h}(\mu_s)$ that 
\begin{align*}
 \mathfrak{h}(\mu_s) &= \int \log \frac{d\mu_s\circ T}{d\mu_s} d\mu_s = 
\int - s\log (J)  - \log h_s + \log h_s \circ T + P(s) d\mu_s\\ &= P(s)  -  s\int \log (J) d\mu_s.
\end{align*}
\begin{theorem}
 \label{theo:dixmier-for-Ruelle-expansive-and-expanding} Assume that (E) holds and that, for some ${\delta}>0$, $P(\delta)=0$ and $\mathfrak{h}(\mu_{\delta})> 0$. Then  
 for any $\omega^\dagger$-Hölder continuous and positive function $a$ and $x\in \Omega$,
\[   
\textrm{Tr}_{\omega}( L_a \, |D_{\delta}|^{-1}) =  \frac{2 h_{\delta}(x)}{\delta \mathfrak{h}(\mu_{\delta})}   \int a \,  d m_{\delta} .
\]
\end{theorem}

Note that the Rokhlin formula provides a simple condition for $\mathfrak{h}(\mu_{\delta})> 0$. Namely, as $P(\delta)=0$, it follows that $\delta\int \log J d\mu = - \mathfrak{h}(\mu_{\delta})$. In particular, $\mathfrak{h}(\mu_{\delta})> 0$ whenever $J < 1$. The proof of the following result is based on this observation and also indicates that the noncommutative integral only is non-trivial for a single choice of the parameter $s$. 

\begin{corollary} \label{cor:Dixmier-negative-potential} 
Assume that (E) holds and that $J < 1$. Then there exists a unique $\delta_0>0$ such that $P(\delta_0)=0$. Moreover, for any $\omega^\dagger$-Hölder continuous function $a$ with $a \geq 0$, $a\neq 0$, $x\in \Omega$ and $s \geq 0$, 
\[\textrm{Tr}_{\omega}( L_a \, |D_{s}|^{-1}) = 
\begin{cases}
\infty &: \, 0 \leq s < \delta_0,\\
 \frac{2 h_{\delta}(x)}{\delta \mathfrak{h}(\mu_{\delta})}   \int a dm_{\delta}    &: \,   s = \delta_0,\\
0      &: \,   s > \delta_0.
\end{cases}
\]
\end{corollary}
  
\begin{proof}  We first prove that $T$ is non-invertible. So assume that $T$ is invertible. It then follows from the contraction along preimages, that 
\[\diam(\Omega) = \diam (T^{-n}(\Omega)) \leq \sum_{U \in \mathcal{U}} \diam (T^{-n}(U)) \xrightarrow{n \to \infty}  0, \]
{which is a contradiction}. Hence, $T$ is non-invertible and there  exists $x\in \Omega$ with at least two preimages. Furthermore, as $T$ is topologically mixing, there is $n_0$ such that $\# T^{-n_0}(\{ x\})\geq 2$ for all $x \in \Omega$, which implies that $P(0) \geq \log 2/n_0> 0$.

Moreover, it follows from Lemma \ref{lem:differentiability} that $P'(s) \leq \max (\log J) < 0$. Hence, there exists a unique $\delta_0>0$ with $P(\delta_0)=0$. As $\mathfrak{h}(\mu_{\delta_0})> 0$ by Rokhlin's formula, Theorem \ref{theo:dixmier-for-Ruelle-expansive-and-expanding} is applicable and provides the statement for $s = \delta_0$.  

It remains to prove the cases for $s \neq \delta_0$. In order to do so, note that for $s > \delta$, it follows from $\rho^s < 1$ that  $\sum_{n=1}^\infty \mathcal{L}_{s \log J}^n(a) < \infty$. By applying Theorem \ref{theo:Tauber}, it then follows as in the proof of Theorem \ref{theo:dixmier-for-Ruelle-expansive-and-expanding} that $\textrm{Tr}_{\omega}( L_a \, |D_{s}|^{-1}) =0$ for $s > \delta$. 

On the other hand, note that it follows from the fact that the support of $m_\delta$ is $\Omega$, that $\int a dm_\delta > 0$. Hence, by Theorem \ref{theo:dixmier-for-Ruelle-expansive-and-expanding}, 
$\textrm{Tr}_{\omega}( L_a \, |D_{s}|^{-1})  \geq \textrm{Tr}_{\omega}( L_a \, |D_{\delta}|^{-1}) > 0$ for $s < \delta_0$.
It immediately follows from this that  
$\textrm{Tr}_{\omega}( L_a \, |D_{s}|^{-1}) = \infty$ for $s < \delta_0$.
\end{proof}
  
We give a further application of Theorem \ref{theo:dixmier-for-Ruelle-expansive-and-expanding} to normalized potentials, that is $\mathcal{L}_{\log J}(\mathbf{1})=\mathbf{1}$. In the next 
   corollary $\mu$ refers to the equilibrium state associated to  $\log J$.

\begin{corollary} 
\label{cor:Dixmier-normalized-potential} Assume that (E) holds and that $\log J$ is normalized. Then $\delta =1$ is the unique root of $P(s)$ for $s \geq 0$, the non-commutative integral $a \to \textrm{Tr}_{\omega}( L_a \, |D_{s}|^{-1})$ is non-trivial only for $s=1$ and,
for any $\omega^\dagger$-Hölder continuous and positive function $a$ and $x\in \Omega$,
\[   
\textrm{Tr}_{\omega}( L_a \, |D_{1}|^{-1}) =  \frac{2}{\mathfrak{h}(\mu)}   \int a d\mu . 
\]
\end{corollary}
  
\begin{proof} 
Observe that it follows from the proof of Corollary \ref{cor:Dixmier-negative-potential} that there exists $n_0$ such that $\# T^{-n_0}\{ x\}\geq 2$ for all $x \in \Omega$. As $\log J$ is normalized, this implies that 
$\prod_{k=0}^{n_0-1}  J \circ T^k(x) < 1$ for all $x \in \Omega$. The result then follows from applying Corollary \ref{cor:Dixmier-negative-potential} to $T^{n_0}$.   
\end{proof}

\subsection{Proof of Theorem \ref{theo:dixmier-for-Ruelle-expansive-and-expanding}}

The remaining part of this section is devoted to the proof of Theorem \ref{theo:dixmier-for-Ruelle-expansive-and-expanding}. 
For the proof of this theorem, it is crucial to understand the spectrum of $L_{a}|D_s|^{-1}$. 
As the absolute value of $(\epsilon_1,\epsilon_2) \mapsto (\epsilon_2,\epsilon_1)$ is the identity,  it follows that
\[ L_{a}|D_s|^{-1}  (\bigoplus_{w \in W^*}  \binom{\epsilon_1(w) }{\epsilon_2(w) })
=
\bigoplus_{w \in W^*}
\binom{a( \tau_w(x)) \mathbb{J}^s_w(x)\, \epsilon_1(w) }{a(\tau_w(y))\mathbb{J}^s_w(x)\,\epsilon_2(w) }.
\]
In particular, $L_{a}|D_s|^{-1}$ has pure point spectrum and each element of the spectrum is of the form $a(\tau_w(x)) \mathbb{J}^s_w(x)$ or $a( \tau_w(y)) \mathbb{J}^s_w(x)$. In order to determine the Dixmier trace, we consider the  $\zeta$- functions
\[
\zeta_+ (s)
=
\sum_{k=1}^\infty
\sum_{ \substack{ w\in W^* \\ l(w)=k }}
a(\tau_w(x))\mathbb{J}(\tau_w(x))^{s}
,\quad
\zeta_{-} (s)
=
\sum_{k=1}^\infty
\sum_{\substack{ w\in W^* \\ l(w)=k }}
a(\tau_w(y))\mathbb{J}(\tau_w(x))^{s}.
\]
The relation to the Dixmier trace is now given by the following version of the Hardy-Littlewood
Tauberian theorem which can be found in  \cite{MR1303779},  Chapter IV.2, Prop. 4.  
\begin{theorem}\label{theo:Tauber}
Assume that $(b_n)$ is a decreasing sequence of positive numbers. Then the following statements are equivalent.
\begin{enumerate}
 \item $ \lim_{s \to 1+} (s-1) \sum_{k=1}^\infty b_k^s = 1$.
 \item $ \lim_{n \to \infty}
\frac{1}{\log  n }
\sum_{k=1}^n b_k   = 1$.
\end{enumerate}
\end{theorem}
In particular, we have for  $c> 0$ and $\delta>0$ that  
\begin{enumerate}
 \item $ \lim_{s \to \delta+} (s-\delta) \sum_{k=1}^\infty b_k^s = c \delta $,
 \item $ \lim_{n \to \infty}
\frac{1}{\log  n }
\sum_{k=1}^n b_k^{\delta}   = c$
\end{enumerate}
are equivalent. We now assume that $\min a \geq 1$. As this implies that $\|a^{s-\delta} -1 \|_\infty \to 0$ as $s \to \delta+$, it follows that        
$\lim_{s \to \delta+} (s-\delta) (\zeta_{+} (s) + \zeta_{-} (s))  = c$ for some $c  > 0$ implies that 
\[   \textrm{Tr}_{\omega}( L_{a}  |D_{\delta}|^{-1})
= \lim_{n \to \infty} \frac{1}{\log n} \sum_{k=1}^{n}b_k  = \frac{c}{\delta},\]
where the $b_n$ refer to the $n$-th largest eigenvalue of $L_{a} |D_{\delta}|^{-1}$ up to multiplicity. Hence, it suffices to determine $c$. 
However, as  
\begin{align*}
\zeta_+(s) + \zeta_- (s) & = \sum_{\substack{ w\in W^*  }}
( a(\tau_w(x)) + a^s(\tau_w(y)))\mathbb{J}(\tau_w(x))^{s}\\
& = 2  \sum_{n=1}^\infty \mathcal{L}^n_{s \log J} (a)(x) + 
\sum_{\substack{ w\in W^*  }}
( a(\tau_w(y)) - a(\tau_w(y)))\mathbb{J}(\tau_w(x))^{s} \\
& =: 2  \sum_{n=1}^\infty \mathcal{L}^n_{s \log J} (a)(x)  + R(s),
\end{align*}
it remains to understand the asymptotics of $(s-\delta)\sum_n \mathcal{L}^n_{s \log J}(a)$ and $(s-\delta)R(s)$.

\medskip
\noindent\textsc{Step 1. The asymptotics of $(s-\delta)\zeta_+(s)$}.  
In order to do so, we apply the results from Section \ref{sec-preliminares}. Set $\mathcal{L}_s:= \mathcal{L}_{s \log J}$ and let $\rho_s$, $h_s$, $m_s$ and $\mu_s$ refer to the objects given by Proposition \ref{prop:eigenfunction}, \eqref{eq:definition-h} and Corollary \ref{cor:decay-to-invariante-measure}. In particular, it follows from Theorem \ref{theo:continuity} that they vary continuously in $s$ and that for any  $n$ sufficiently large,      
\[ \left| \frac{\mathcal{L}_{s}^n(a)}{\rho_s^n h_s} -  m_s(\varphi) \right| \leq \kappa \frac{\textrm{Höl}_{\omega^\dagger}(a)}{(c_s + n)^\gamma}, \]
where $\kappa$ only depends on $T$ and $c_s > 1$ is continuous in $s$ by Theorem \ref{theo:main-contraction-result}. 
We now show that $\rho_s < 1$ for $s> \delta$ sufficiently close to $\delta$.  

\begin{lemma} \label{lem:differentiability}
Assume that $\gamma > 0$. Then $P(s)$ is differentiable and $dP(s)/ds = \int (\log J) d\mu_s$.
\end{lemma}

\begin{proof} We now write  $A:=\log J$.
It follows from Corollary \ref{cor:decay-to-conformal-measure} that for each $\omega^\dagger$-Hölder continuous function $f$,   
\[ \| \mathcal{L}_{s}(a)/\mathcal{L}_{s}(\mathbf{1}) -  m_s(f) \|_\infty \leq \kappa \textrm{Höl}_{\omega^\dagger}(f) e^{\kappa C_{sA}} n^{-\gamma}.\]
 The uniform estimate above now implies that 
\begin{align*}
\frac{d}{ds} \log \frac{1}{n}\mathcal{L}_{s}^n(\mathbf{1})(z_0) & = \frac{1}{n} \frac{\sum_{\sigma^n(x) = z_0} A_n(x) e^{sA_n(x)} }{\mathcal{L}_{sA}^n(\mathbf{1})(z_0)} = \frac{ \mathcal{L}_{s}^n(\frac{1}{n}  A_n)(z_0) }{\mathcal{L}_{s}^n(\mathbf{1})(z_0)}\\
& = \frac{1}{n} \int A_n dm_s  \pm \kappa n^{-\gamma}  \textrm{Höl}_{\omega^\dagger}(A) \\
& = \frac{1}{n} \sum_{j=0}^{n-1} \int A \circ T^j   dm_s  \pm \kappa n^{-\gamma}  \textrm{Höl}_{\omega^\dagger}(A)
\\
& = \frac{1}{n} \sum_{j=0}^{n-1} \int A \rho_s^{-j} \mathcal{L}_{sA}^j(\mathbf{1}) dm_s 
 \pm \kappa n^{-\gamma}  \textrm{Höl}_{\omega^\dagger}(A).
\end{align*}
As $\mathcal{L}_{sA}^j(\mathbf{1})/\rho_s^{-j}$ converges uniformly to $h_s$, it follows that  
\[\frac{d}{ds} \log \frac{1}{n}\mathcal{L}_{sA}^n(\mathbf{1})(z_0) \xrightarrow{n \to \infty}  \int A h_s dm_s \]
uniformly for $s  = s_0 \pm \epsilon$, for $s_0> 0$ and $\epsilon > 0$ sufficiently small. Hence, the right hand side coincides with the derivative of the pressure function.
\end{proof}

Note that  $P(\delta)=0$  implies that $\mathfrak{h}(\mu_{\delta}) = - \delta \int  A d\mu_{\delta} > 0$. As $\mu_s$ 
varies continuously in $s$, it hence follows that $P'(s)< 0$ for $\delta <  s < \delta + \epsilon$ for $\epsilon$ sufficiently small. Hence, $P(s)$ is strictly decreasing in $[\delta, \delta + \epsilon]$ and, therefore, $\lambda_s < 1$ for $s \in  (\delta, \delta+ \epsilon]$.  
In particular,    
\begin{align*}
(\ast) & :=  \left|(s-\delta)\zeta_+(s) - \frac{\rho_s (s-\delta)}{1-\rho_s} h_s(x) m_s(a)\right|
 =    (s-\delta) \left| \sum_{n=1}^\infty \mathcal{L}^n_{s} (a)(x)  - \rho_s^n h_s(x) m_s(a)\right| \\
& \leq \textrm{Höl}_{\omega^\dagger}(a) \kappa \| h_s\|_\infty  (s-\delta)  \sum_{n=1}^\infty  \rho_s^n (c_s + n)^{-\gamma}.
\end{align*} 
Now assume that $\gamma > 1$. As $\lim_{s \to \delta+} \rho_s  = 1$, it follows that $\lim_{s \to \delta+} \sum_n \rho_s^n (c_s + n)^{-\gamma}$ exists and is finite by an application of Abel's theorem. Hence, the right hand side tends to $0$ as $s \to \delta+$. 
On the other hand, if $0 < \gamma  <  1$, it follows from a lengthy calculation with the Taylor expansion of $(1 - x)^{\gamma -1}$ that for $0 \leq t < 1$,
 \[
 \sum_{n=1}^\infty t^n n^{-\gamma} \leq \frac{e^{ (1-\gamma)^2\pi^2/12}}{1-\gamma} \left( (1-t)^{\gamma -1}  -1 \right){.}
 \]
In particular, it follows for $0 < \gamma <  1$ from this estimate, Lemma \ref{lem:differentiability} and the continuity of $h_s$ that
\begin{align*}
\limsup_{s \to \delta+ }\; (\ast)  & \leq  \textrm{Höl}_{\omega^\dagger}(a) \kappa   \lim_{s \to \delta+ }  \left( \frac{s-\delta}{\lambda_s -1} \| h_s\|_\infty  \right) \lim_{t\to 1-} (1-t)\sum_{n=1}^\infty t^n n^{-\gamma}  \\
& \ll   \textrm{Höl}_{\omega^\dagger}(a)  \|h_{\delta}\|_\infty \left(- \textstyle \int \log J d\mu_{\delta}\right)^{-1}  \lim_{t\to 1-} (1-t)^\gamma =0.
\end{align*}
This implies that $\lim_{s \to \delta+ }\; (\ast) = 0$. Hence,  
\begin{equation} \label{eq:asymptotics-zeta}
 \lim_{s \to \delta+} (s-\delta)\zeta_+(s) = - h_{\delta}(x) m_{\delta}(a)  \lim_{s \to {\delta}+} \frac{s-1}{e^{P(s)} -1} =   - \frac{h_{\delta}(x)   \int a dm_{\delta}}{\mathfrak{h}(\mu_{\delta})}. 
 \end{equation}

\medskip
\noindent\textsc{Step 2. The asymptotics of $(s-{\delta_0})R(s)$}. Let $c_n$ refer to the contraction rate of $T$. As already shown in Section \ref{sec-preliminares}, it holds in both cases  that $\omega^\dagger(c_n) \ll n^{-\gamma}$. Hence, for $s> {\delta_0}$
\begin{align*}
R(s) & \leq \textrm{Höl}_{\omega^\dagger}(a) \sum_{n=1}^\infty \omega^\dagger(c_n) \mathcal{L}_s^n(\mathbf{1}) \ll  
 \textrm{Höl}_{\omega^\dagger}(a) \sum_{n=1}^\infty  \mathcal{L}_s^n(\mathbf{1}) n^{-\gamma}\\
& \ll \textrm{Höl}_{\omega^\dagger}(a) \sum_{n=1}^\infty  \lambda_s^n n^{-\gamma}.
 \end{align*} 
It then follows from the same argument as above that $\lim_{s \to {\delta}+}(s-1)R(s) =0$. 
This proves the theorem for $a> 0$. The general case then follows by considering $a+1$.

\section{Examples and counterexamples} \label{exampand}
 
In this section, we discuss the existence of spectral triples and the representation of the Dixmier trace as in Theorem \ref{theo:dixmier-for-Ruelle-expansive-and-expanding}. As a first class of examples, we consider  topologically mixing Ruelle expanding maps, equipped with a normalized potential $\log J$, which is $\omega_{0,\beta\textrm{log}}$-Hölder continuous for some $\beta > 1$. In particular, it follows from Proposition \ref{prop:spectral-triple}, that $(H_U,A,D_s)$ is a spectral triple for each $s \in (0,s_0)$, with $s_0$ given by Proposition \ref{prop:spectral-triple}. On the other hand, as 
the representation as a Dixmier trace only holds for the single parameter $\delta =1$ (cf. Corollary \ref{cor:Dixmier-normalized-potential}, it is natural to ask in which cases a spectral triple comes with a non-commutative integral with the same exponent. 

However, the examples and results given above suggest that the relation between $s_0$ and $\delta$ is related to the topological dimension of $\Omega$. For subshifts of finite type, it is shown in Remark \ref{remark:s_0} that $s_0 = \infty$. Hence, one may choose $s = 1$ and obtains a spectral triple with a Dixmier trace representation of the equilibrium state. Moreover, by applying Theorem \ref{theo:spectral-metric-sft}, it follows that after a change of metric, the spectral metric coincides with the Wasserstein distance with respect to the new metric. Or in other words, if $\Omega$ is a Cantor set, then there exist spectral triples with strong properties.

If, on the other hand, the space is connected, the situation might be different. As shown in Section \ref{subsec-sierpinski}, there exists a spectral triple for the parameter $s= \log 2/\log 3$ whose spectral metric coincides with Wasserstein distance. Moreover, by going through the proof of Theorem \ref{theo:spectral-metric}, it follows that $(H_U,A,D_s)$ no longer is a spectral triple for $s > \log 2/\log 3$. However, as the potential is normalized,  the Dixmier representation only holds for $s=1$. 

Hence, from a general point of view, it seems that spectral triples and their spectral metric are associated to curve length whereas the Dixmier trace reflects a kind of volume.

\subsection{The Dyson potential} \label{subsec:dyson}
We now discuss an example from statistical mechanics and do not require that the potential is normalized.  That is, we consider $\Omega := \left\{-1,1\right\}^\mathbb{N}$ equipped with the shift map $T$(i.e. $T((x_i: i \geq 1)) = (x_i: i \geq 2)$) and the metric $d((x_i: i \geq 1),(y_i :i \geq 1)) := 2^{- \min\{i: x_i \neq y_i\}}$. Note that this implies that $T$ is uniformly expanding and $c_n(t) = 2^{-n}t$. For $\alpha > 1$, we then refer to  
\[
\Phi((x_i: i \geq 1))  = x_1 \sum_{n=1}^\infty n^{-\alpha}x_{n+1}
\]
as the \emph{Dyson potential} due to the similarity to Dyson's model of ferromagnetism. We now determine the regularity of $\Phi$. So assume that  $d(x,y)\leq 2^{-k}$ for $k>1$. Then $x_j = y_j$ for $j=1,\ldots k-1$ and 
\[|\Phi(x) - \Phi(y)| \leq \sum_{n=1}^\infty |x_{n+1} - y_{n+1}|  n^{-\alpha} \leq  2  \sum_{n \geq k} n^{-\alpha} \ll k^{1-\alpha}       = \left(-\frac{\log d(x,y)}{\log 2}\right)^{1-\alpha}.\] 
Hence, we have shown that Dyson potential is $\omega_{0,(\alpha-1)\textrm{log}}$-Hölder continuous. In particular, it follows from Corollary \ref{cor:decay-to-invariante-measure} that for $\alpha > 2$,  
\begin{equation} \nonumber
 \| \mathcal{L}_{\overline{\Phi}}(f) -\mu(f) \|_\infty \ll \textrm{Höl}(f) \, n^{2-\alpha},  
\end{equation}
for each $\omega_{0,(\alpha-2)\textrm{log}}$-Hölder continuous function $f$. However, as we see from the next result, Theorem \ref{theo:dixmier-for-Ruelle-expansive-and-expanding} is not immediately applicable. 
 
\begin{lemma} \label{lem:dyson-pressure-is-increasing} Assume that $\alpha > 2$. Then the pressure function $P(s)$ is increasing for $s\geq 0$. Furthermore, $P'(0) = 0$, $P(0) = \log 2$ and  $P(s) > s \max \Phi$ for all $s \geq 0$.
\end{lemma}

\begin{proof}
For $s=0$, $\Phi=1$. Hence, for this parameter, the measure of maximal entropy is the equilibrium state. This implies that $P(0) = \log 2$ and $P'(0) =0$ by symmetry. Moreover, note that 
the Dirac measure $\delta_{\overline{1}}$ on $(1111 \ldots)$ is an invariant measure with $\int \Phi d \delta_{\overline{1}} = \max \Phi$. Hence, by the variational principle, $P(s) \geq s \max \Phi$. 
However, in our setting, equilibrium states are unique (see, e.g.,  \cite{MR2342978}), $\mu_s$ is the unique equilibrium measure and $\mu_s$ is not atomic. Hence, $P(s) > s \max \Phi$. 
Now assume that $\int \Phi d\mu_s < 0$ for some $s> 0$. Moreover, define $\tilde{\mu_s}$ by  $\tilde{\mu_s}([(x_1 \ldots x_n)] ):= {\mu_s}([(-x_1 \ldots -x_n)])$, for $[(x_1 \ldots x_n)]:= \{ (y_i):y_i =x_i \forall i = 1,\ldots n\}$. As the Borel $\sigma$-algebra is generated by these cylinder sets, it follows that $\tilde{\mu_s}$ also is a probability measure.  
As  $\tilde{\mu_s}$ is invariant as well, the variational principle implies that 
\[P(s) \geq \mathfrak{h}(\tilde{\mu_s}) + s\int \Phi  d\tilde{\mu_s} = \mathfrak{h}( {\mu_s})  - s \int \Phi d\mu_s >  \mathfrak{h}( {\mu_s})  + s \int \Phi d\mu_s.  \]
Hence, $\mu_s$ is not an equilibrium measure, 
{which is a contradiction}. Therefore,  $\int \Phi d\nu_s \geq 0$ for all $s \geq 0$. The assertion then follows from Lemma \ref{lem:differentiability}.
\end{proof}

Hence, in case of the Dyson potential, $P(s)\neq 0$ for all $s \in \R$. However, by slightly modifying $\Phi$, we obtain the following for the potential $\Phi_t : = \Phi - t$.

\begin{lemma} \label{lem:dyson-zeta} Assume that $\Phi$ is the Dyson potential and  $\alpha>2$. Then, for each $t > \max \Phi$, the pressure function $s \mapsto P(s)$ is strictly decreasing and there is a unique $\delta$ with $P(\delta)=0$. Furthermore, 
\[\lim_{s \to \delta + }(s  - \delta) \sum_{k=0}^\infty e^{-nt} \mathcal{L}_{s\Phi}^n(f)(z)  
 = \frac{ m_{\delta}(f) h_{\delta}(z)}{\mathfrak{h}(\mu_{\delta})} .   \]  
\end{lemma}

\begin{proof} By the above lemma, there is no zero of $P_{\Phi_t}(s)$ for $t \leq \max \Phi$. On the other hand, as $\Phi_t < 0$,  
it follows from Corollary \ref{cor:Dixmier-negative-potential}, that Theorem \ref{theo:dixmier-for-Ruelle-expansive-and-expanding} is applicable. In particular, the assertion follows from \eqref{eq:asymptotics-zeta} in the proof of Theorem \ref{theo:dixmier-for-Ruelle-expansive-and-expanding}.
\end{proof}

A possible application of this Lemma is the following.
\begin{proposition} \label{prop:dixmier-triple}
 Assume that $\Phi$ is the Dyson potential, that $\alpha> 2$, that $t > \max \Phi$, that $J := e^{\Phi - t}$ and that 
 $x = (-1,x_2, x_3, \ldots)$ and $y = (1,y_2,y_3, \ldots)$. Then there exists a unique $\delta>0$ such that $P(\delta) =0$. Furthermore, $(H, A, \mathcal{D}_\delta)$ as in Theorem \ref{theo:spectral-metric-sft} is a spectral triple whose spectral metric is a metric and
\[  \textrm{Tr}_{\omega}( L_a \, |D_{\delta}|^{-1}) =  \frac{2 (h_\delta(x) + h_\delta(y))}{\mathfrak{h}(\mu)}   \int a dm_\delta  \]  
for any positive and $\omega_{0,(\alpha-2)\textrm{log}}$-Hölder continuous function $a$.
\end{proposition} 

\begin{proof} The statement is a consequence of Theorem \ref{theo:spectral-metric-sft}, Lemma \ref{lem:dyson-zeta} and Theorem \ref{theo:dixmier-for-Ruelle-expansive-and-expanding}.
\end{proof}

\subsection{Potentials from Walters' family} We now assume that $\Omega :=\{0,1\}^\mathbb{N}$, that $T$ is the shift map and that  potential $g:\Omega  \to  \mathbb{R}$  is continuous. We now would like to give examples where the statement of 
Theorem \ref{theo:dixmier-for-Ruelle-expansive-and-expanding} is true, that is  
\begin{align} 
\label{eq:zeta-walters-class}
\lim_{s\to 1^{+}}(s-1) \sum_{k=1}^\infty \mathcal{L}_{sg}^k (a)(x)
 = \frac{1}{c}\int_{\Omega}a\, dm,
\end{align}
for some $c>0$ and a continuous function $a$. In here, we assume that $f$ is in the so-called Walters family, introduced in \cite{MR2342978}.
\begin{definition}
A potential $g$ is in the Walters family if $g$ is continuous and there exist convergent sequences 
$(a_n)_{n \in \mathbb{N}}$, $(b_n)_{n \in \mathbb{N}}$, $(c_n)_{n \in \mathbb{N}}$ and $(d_n)_{n \in \mathbb{N}}$  such that, for any $x \in \Omega$ and all  $n \in \mathbb{N}$,
\[g(0^{n+1}1x) = a_{n+1},\; g(01^n0x) = b_n,\; g(1^{n+1}0x) = c_{n+1},\; g(10^n1x) = d_n. \]
\label{defi1}
\end{definition}
We refer to $a, b, c$ and $d$, respectively, as the limits of the sequences  $a_n, b_n, c_n$ and $d_n$. The relevance of this family is based on the fact that the family is sufficiently rich to include examples of important classes of functions like the Bowen class or functions satisfying Walters condition (see Theorem 1.1  in \cite{MR2342978}). Moreover, it is possible to obtain necessary and sufficient conditions in this class such that Ruelle's operator theorem holds (see Theorem 3.1 in \cite{MR2342978}).

For our purposes, it is in fact sufficient to study the following subclass of Walters' family. Namely, we assume that $g$ is of the form  
\begin{align}
f(x) = \begin{cases}
c_0  &:\, x \in [0], \\
c_n  &:\, x \in [1^n\,0], \hbox{ for } n \in \N, \\
c & \, x = (1,1,\ldots),
\end{cases}
\label{eq1}
\end{align}
for some convergent sequence $(c_n)$ and $c:= \lim_{n\to \infty} c_n$. 

In Appendix \ref{sec112}, we will analyse the action of the Ruelle operator $\mathcal{L}_g$ associated to potentials in this class in more detail in the appendix. In particular, we obtain in Proposition \ref{leel} a simple condition in order to guarantee that the leading eigenvalue of $\mathcal{L}_g$ is equal to 1. We will now give two examples of potentials not covered by Theorem \ref{theo:dixmier-for-Ruelle-expansive-and-expanding} in order to show that in lower regularity the limit  \eqref{eq:zeta-walters-class} only sometimes exists.

\begin{example} \label{kpo} 
For a given sequence $a_k>0$, $k \geq 1$, we consider the associated sequence $b_k=
\frac{a_{k+1}}{a_k}$, $k \geq 1$. Obviously, $(b_k)$ is determined by $(a_k)$ and vice-versa.
Furthermore, we assume that $\sum_{k=1} ^\infty  a_k=\alpha$ is finite, $a_1=1$ and that $\lim_{k \to \infty} b_k$ exists and that $f:\Omega\to\mathbb{R}$ satisfies
\begin{align} \label{rer}
f(x)
=
\begin{cases}
- \log \sum_{j=1}^\infty  a_k &:\,  x\in [0],
\\
\phantom{- } \log b_n  &:\,    x\in [1^n0], n\geq 1{,}
\\
\phantom{- }\log  b&:\, x=(1,1,1,\ldots).
\end{cases}
\end{align}
The classical Hofbauer case (see \cite{Hof}) as described in \cite{FL01,MR1242602,CL} corresponds to $a_k=k^{-\gamma}$, $\gamma>1$, $k\geq 1$. Moreover, for this case one can get an explicit expression for the eigenfunction $h$ and the eigenmeasure $m$ (see \cite{FL01}). If, in addition, $\gamma>2$, then there exists an equilibrium $\mu$ for this potential.

In the general setting, the following holds. Let $h$ be the continuous function with   
\begin{align*} \label{eq:eigenfunction-general-hofbauer}
h(x)  = 
\begin{cases}
  \phantom{\frac{1}{a_{n+1}}} \sum_{j=1}^\infty  a_k &:\,  x\in [0], \\
   \frac{1}{a_{n+1}} \sum_{j = n+1}^\infty  a_j  &:\,    x\in [1^n0], n\geq 1.
\end{cases}
\end{align*}
It follows from a straightforward calculation that $\mathcal{L}_f(h) = h$. Hence, $g = f + \log h - \log h\circ T$ is a normalized potential and the Jacobian $J$ of the associated equilibrium state $\mu$ is of the form $J = e^f h/h\circ T$. 
That is, for $x \in [1^n,0]$, $n\geq 1$, 
\[
J(x) = \frac{\sum_{j=n+1}^{\infty} a_j}{ \sum_{j=n}^{\infty} a_j} = 1- \frac{a_{n}}{\sum_{j=n}^{\infty} a_j}.
\]
We will now analyse the case were $a=1_{[1]}$, that is, the indicator function of the cylinder $[1]$. Our goal in this case is to obtain an explicit expression of the zeta function 
\[\zeta^{+}(s):=   \sum_{k=1}^\infty \mathcal{L}_{sg}^k (a)(z) \]
and to evaluate $\lim_{s \to 1^+}(s-1)\, \zeta^{+} (s)$. We will now assume that $s>1$ 
{and that $z$ belongs to the}\break
{cylinder $[01]$.}
As we will see, the convergence of the series 
$\sum_{n=1}^\infty a_n \log a_n $ {will play a} decisive role.

Note that $\mathcal{L}_{s\,f}^n (\mathbf{1}_{[1]})(z)=\mathcal{L}_{s\,f}^n (\mathbf{1}_{[1]})(0,z)= \mathcal{L}_{s\,f}^n (\mathbf{1}_{[1]})(0,1^d,0 ,\ldots) $, 
for $d>0$.
To simplify the notation we set $L_s^n :=  \mathcal{L}_{s\,f}^n (\mathbf{1}_{[1]})(z)$
for $n\geq 1$.
Then 
\begin{align*}
 L_s^1  = & a_2^s 
 \\
 L_s^2 = &  L^1_s  \frac{1}{\alpha^s} + a_3^s 
 \\
 L_s^3 = &  L^2_s  \frac{1}{\alpha^s} +L^1_s  \frac{1}{\alpha^s} a_2^s+ a_4^s \\
 L_s^4 = &  L^3_s  \frac{1}{\alpha^s} +   L^2_s  \frac{1}{\alpha^s}  a_2^s +   L^1_s  \frac{1}{\alpha^s} a_3^s + a_5^s\\ 
 & \vdots
\end{align*}
By induction one then obtains 
the following renewal equation for $L_s^n$:
\[ L_s^n
=
\sum_{k=1}^{n-1} \frac{a_{k}^s}{\alpha^s} L_s^{n-k}
+
a_{n+1}^{s}. \]
From this we obtain the convergence of the following series and
the expression
\[
\sum_{k=1}^{\infty} L_{s}^{k}
=(\sum_{j=1} ^\infty  a_j^s\, -1)
 +  (\sum_{j=1} ^\infty  a_j)^{-s}  \,(\sum_{j=1} ^\infty  a_j^s\,)\,      \sum_{k=1}^{\infty} L_{s}^k.
\]
Therefore  $\zeta_{+}(s)=L_s^1 +  L_s^2 +  L_s^3 +\ldots+  L_s^n+\ldots$ is given in an explicit form by
\[ 
\zeta_{+}(s)
=
\frac{    \sum_{j=1} ^\infty  a_j^s -1   }{ 1- \sum_{j=1} ^\infty  a_j^s  \,(\,\sum_{j=1} ^\infty  a_j)^{-s}  }.
\]
Consider the function  $\gamma(s)= \sum_{j=1} ^\infty  a_j^s  \,(\,\sum_{j=1} ^\infty  a_j)^{-s}  $.
 In the case $\gamma$ is
differentiable at the left of $s=1$ there exists
\[
\gamma '(1^{-})=  \lim_{s \to 1} \frac{ \sum_{j=1} ^\infty  a_j^s  \,(\,\sum_{j=1} ^\infty  a_j)^{-s}  -1}{s-1}<0.
\]
Therefore,
\[
\lim_{s \to 1^+}(s-1)\, \zeta^{+} (s)
= \frac{ \sum_{j=1} ^\infty  a_j -1  }{-\,\gamma'(1) },
\]
Now, we estimate
\begin{equation} \label{trf}
 \gamma '(1)  =
 \,\frac{\sum_{n=1}^\infty a_n \log a_n}{\sum a_n} - \,\log(\sum_{n=1}^\infty a_n) .
\end{equation}
Note that $ \sum_{n=1}^\infty a_n \log a_n$ is finite for $a_n=n^{-\gamma}$ and $\gamma>2$. Then, $\gamma '(1)  $ is finite and non-zero.
This case is an example of a potential $f$ which is not in the Walters class but 
\[
\lim_{s \to 1^+}(s-1)\, \zeta^{+} (s)= \frac{1}{h_{\mu_f}}\int_{\Omega}a\, d\mu_f,
\]
in the same way as in  (\ref{kw}).

\qed

\end{example}

\smallskip

\begin{example} \label{bobo}  We will now construct a potential $f$ of the form (\ref{rer}) where the $a_n$, $n \in \mathbb{N}$, of last example  are such that
the expression (\ref{trf})  for $\gamma' (1)$ above is not finite.
We are interested in a sequence $a_n$, $n \in \mathbb{N}$, such that,
\begin{enumerate}
\item[(a)] $\sum_{n=1}^\infty a_n$ is finite,
\item[(b)] 
$\,\sum_{n=1}^\infty a_n \log a_n$ does not converges and
\item[(c)]  there exists the limit: $\lim_{k \to \infty} C_k=C$.
\end{enumerate}
  Take 
\begin{equation} \label{sda} a_n= \, \frac{1}{ n \, (\log_2 (n))^p},\end{equation} 
$n \geq 2$, for some positive $p$. Remember that we take $a_1=1$. In this case $C=1$.
Now, considering  the case $p=2$, we get $\sum_{n=1}^\infty a_n \log a_n = \infty$ (that is item (b) is true) and the potential associated to such a sequence $a_n$ (and, $C_n$) is not contained in Bowen's class (see Theorem 1.1 in Walters).  

In this case, one obtains an explicit conformal measure $\mu$, where $\mu (C_n)\sim a_n$, but there is no equilibrium probability with support on the all space $\Omega$. Moreover, one can show that $\mu([1])>0.$ Note that in this case  $\gamma'(1)=\infty$, and therefore the limit
\[
\lim_{s \to 1^+}(s-1)\, \zeta^{+} (s)
= \frac{ \sum_{j=1} ^\infty  a_j -1  }{-\,\gamma'(1) }=0.
\]
For this choice of $z_0 \in [010]$ and $[w]=[1]$, we get 
$$ \lim_{s \to 1^+}(s-1)\, \zeta^{+} (s) = 0.$$
This  provides an example where the results of the previous sections do not apply.
\end{example}
 
In Appendix \ref{sec112}, we will consider general potentials satisfying expression (\ref{eq1}) and more general functions $a$. Moreover, in  Theorem  \ref{luc} of Appendix \ref{sec3},  we consider the case of a potential $f$ (as in (\ref{eq1})) which is in the Walters class but is not H\"older continuous, such that, $ \lim_{s \to 1^+}(s-1)\, \zeta^{+} (s) $ is finite. In this result the observable $a$ can be the indicator function of any cylinder set $[w]$ in $\Omega$. 

\section{Dixmier Trace Representations for Equilibrium Measures of Topological Markov Shifts}
\label{sec-top-markov-chains}
The key ingredient {of the proof of} Theorem \ref{theo:dixmier-for-Ruelle-expansive-and-expanding}
is the uniform decay of the Ruelle operator. In this section we illustrate how to obtain a similar identity in {a} non-compact setting, using a spectral gap
 (see Definition \ref{def-SGP}).
The basic setup and presentation follow closely the reference \cite{MR2551790}.
Now {assume that} $\mathcal{A}=\mathbb{N}$ and
$T=(t_{ij})_{\mathcal{A}\times \mathcal{A}}$
is a matrix of zeroes and ones.
Let $\sigma:\Omega\to \Omega$ denote the left shift map, where
\[
\Omega\equiv \{(x_1,x_2,\ldots)\in \mathcal{A}^{\mathbb{N}} : t_{x_ix_{i+1}}=1 \}.
\]
We think of $\Omega$ as the collection of one sided infinite admissible words. We
equip it with its usual distance $d(x,y)=2^{-N(x,y)}$, where $N(x,y)\equiv \inf\{k: x_k\neq y_k\}$
(with the convention that $\inf\emptyset = +\infty$). The resulting topology is generated by the
cylinder sets
$
[y_1,y_2,\ldots,y_{n}]\equiv \{x\in X: x_i=y_i,\ i=1,\ldots,n  \},
$
where $n\geq 1$.
A word $\underline{y}\in\mathcal{A}^{n}$ is called admissible if {the associated} cylinder is non-empty. The length of an admissible word
$\underline{y}=(y_1,\ldots,y_{n})$
will be denoted in this section by $|\underline{y}|\equiv n$. We also assume that
$\sigma:\Omega\to \Omega$ is topologically mixing and locally compact.
{This means that for} any two symbols $p,q\in \mathcal{A}$,
there is an $N(p,q)\in\mathbb{N}$
such that for all $n\geq N(p,q)$ there is an admissible word of length $n$ which starts
at $p$ and ends at $q$, and for all $p\in\mathcal{A}$ we have
$\#\{q\in\mathcal{A}: t_{pq}=1\}<\infty$.

 We define the $n$-th variation
of a function $f:\Omega\to\mathbb{R}$ as
$\mathrm{var}_{n}(f)\equiv \sup\{|f(x)-f(y)|\ :\ x_1^{n}=y_1^{n}\}$,
where $z_{m}^{n}\equiv (z_m,\ldots,z_n)$.
A function $f:\Omega\to\mathbb{R}$ is
called a weakly $\theta$-H\"older continuous function
for $0<\theta<1$ if there exists a positive number $\mathrm{Hol}_{\theta}(f)$
such that $\mathrm{var}_{n}(f) \leq \mathrm{Hol}_{\theta}(f)\ \theta^n$.
The Birkhoff sum of a function $f$ is denoted by
$S_n(f)(x)\equiv \sum_{k=0}^{n-1}f\circ\sigma^{k}$.

Suppose $f$ is a weakly $\theta$-H\"older continuous function (or has summable variations)
and $X$ is topologically mixing. The
\textit{Gurevich} pressure of $f$ is the limit
\[
P_{G}(f) = \lim_{n\to\infty} \frac{1}{n}\log Z_{n}(f,q),
\quad  
Z_n(f,q) = \sum_{\sigma^n(x)=x}\exp(S_n(f)(x))1_{[q]}(x),
\]
for some $q\in \mathcal{A}$. This limit is independent of $q$, and if
$\sup f<\infty$ then it is equal to
$\sup \{h_{\mu}+\int f\, d\mu \}$, where the supremum ranges
over all invariant probability measures such that the supremum is not of
the form $\infty-\infty$ (see \cite{MR1738951}).

In this setting the Ruelle operator associated with $f$ is defined
in {similar way by}
\[
\mathcal{L}_{f}(\varphi)(x)= \sum_{q\in\mathcal{A}} e^{f(qx)}\varphi(qx).
\]
This is well-defined for functions $f$ such that the sum converges for
all $x\in X$. Let $\mathrm{dom}(\mathcal{L}_{f})$ denote the collection
of such functions.

\begin{definition}[Spectral Gap Property - SGP]
\label{def-SGP}
Suppose that $f$ is $\theta$-weakly H\"older continuous,
and that $P_{G}(f)<\infty$. We say that $f$ has the spectral gap property (SGP)
if there is a Banach space of continuous functions $\mathcal{B}$ such that
\begin{enumerate}
	\item $\mathcal{B}\subset \mathrm{dom}(\mathcal{L}_{f})$ and
	$\mathcal{B}\supset \{1_{[\underline{a}]}: \underline{a}\in \mathcal{A}^n, n\in\mathbb{N} \}$,

	\item  $f\in \mathcal{B}$ implies $|f|\in\mathcal{B}$,
	$\||f|\|_{\mathcal{B}}\leq \|f\|_{\mathcal{B}} $,

	\item  $\mathcal{B}$-convergence implies uniform convergence on cylinders,

	\item  $\mathcal{L}_{f}(\mathcal{B})\subset \mathcal{B}$ and
	$\mathcal{L}_{f}:\mathcal{B}\to \mathcal{B}$ is bounded,

	\item  $\mathcal{L}_{f}=\lambda P+N$, where $\lambda=\exp(P_{G}(f))$,
	and $PN=NP=0$, $P^2=P$, $\dim(\mathrm{Im}\,P)=1$, and the spectral radius of $N$
	is less than $\lambda$,

	\item  if $g$ is $\theta$-H\"older, then
	$\mathcal{L}_{f+zg}:\mathcal{B}\to\mathcal{B}$ is bounded and
	$z\longmapsto \mathcal{L}_{f+zg}$ is analytic on some complex neighborhood of zero.

\end{enumerate}

\end{definition}

The motivation to introduce this concept is the next theorem.
{This result has been proven} in several contexts by many authors
(see \cite{MR2551790} and references therein).

\begin{theorem}
Suppose $X$ is a topologically mixing countable Markov shift,
and $f:X\to\mathbb{R}$ is a $\theta$-weakly H\"older continuous potential
with finite Gurevich pressure, finite supremum, and the SGP.
Write $\mathcal{L}_{f}=\lambda P+N$. Then
\begin{enumerate}
		\item $P$ takes the form $Pf = h \int f\, d\nu$, where $h\in\mathcal{B}$
		is a positive function,
		and $\nu$ is a measure which is finite and positive on all cylinder sets,

		\item the measure $d\mu= hd\nu$ is a $\sigma$-invariant probability measure {such that,} if $\mu$ has finite entropy, then $\mu$ is the unique equilibrium measure
		of $f$,

		\item there is a constant $0<\kappa<1$ such that for all $g\in L^{\infty}(\mu)$ and
		$f$ bounded H\"older continuous, there exists a positive constant $C(\varphi,\psi)$ such that
		$|\mathrm{cov}_{\mu}(\varphi,\psi\circ \sigma^n)|
		\leq C(\varphi,\psi)\kappa^n$
		($\mathrm{cov}$ = covariance).

		\item Suppose $g$ is a bounded H\"older continuous function, such that
		$\mathbb{E}_{\mu }[g]=0$. If $g\neq f-f\circ\sigma$ with
		$f$ continuous, then there is $\varrho>0$
		such that $S_n(g)/\sqrt{\varrho n}$ converges
		in distribution (w.r.t. $\mu$) to a standard normal distribution.

		\item Suppose $g$ is a bounded H\"older continuous function, then the function
		$t\longmapsto P_{G}(f+tg)$ is real analytic on a neighborhood of zero.
\end{enumerate}

\end{theorem}

Now the $C^*$-algebra $A$ is taken to be as $C_{0}(X,\mathbb{C})$,
the set $X$ of complex valued continuous functions $f:X\to\mathbb{C}$
that vanish at infinity.
The Hilbert space $H$ is  $\ell^2(W^*)\oplus \ell^2(W^*)$, where
$W^*$ is the set of all finite length admissible words
 $w=(w_1,w_2,\ldots,w_n)$, where $n\in \mathbb{N} $ and
$w_j\in \mathcal{A}$. The space $\ell^2(W^*)$ is defined as before, the complex vector space
of all functions $\epsilon:W^{*}\to \mathbb{C}$, satisfying $\sum_{w \in W^*} |\epsilon(w)|^2 <\infty$.
Fix two arbitrary elements $x,y\in \Omega$.
For each $a\in A$ the operator $L_a: H \to H$ is defined by
\[
L_a (
\bigoplus_{ w \in W^*(x,y) }   \binom{\epsilon_1(w) }{\epsilon_2(w) }
)
=
\bigoplus_{w \in W^*(x,y)}
\binom{a( wx)\, \epsilon_1(w) }{a( wy)\,\epsilon_2(w) },
\]
where $W^*(x,y)$ is the set of all admissible words $w\in W^{*}$
such that $t_{w_nx_1}=\ t_{w_ny_1}=1$.
Finally, the Dirac operator is given by
\[
D
=
( \bigoplus_{w \in W^*(x,y)}  \binom{\epsilon_1(w) }{\epsilon_2(w) } )
=
\bigoplus_{w \in W^*(x,y)} \frac{1}{\mathbb{J}(wx)} \,\binom{\epsilon_1(w)}{\epsilon_2(w)},
\]
where $\mathbb{J}(wx)=[J(w_{l(w)}x)\ldots J(w_{l(w)}\cdots w_{1}x)]$,
$l(w)$ is the \textit{length} of a string $w\in W^*(x,y)$ and
$
\log J = f + \log h -\log\circ h\circ\sigma -\log \lambda.
$
We also assume there are constants $0<\kappa_1<\kappa_2$ such that
$e^{-\kappa_1 q}\leq J(qx)\leq  e^{-\kappa_2 q}$ for all $x\in \Omega$.

The verification that $(A,H,D)$ is a spectral triple is
more involved. The difficult {part} is to verify item (3) of Definition \ref{def-spec-triples}.
When working with spectral triples in this context (infinite alphabets)
we need to restrict ourselves to a class of functions $J$
for which the operator $D^{-1}$ is a compact operator.
A simple example of  such a  function is $J(x)=e^{-x_1+1}(1-e)^{-1}$.
To prove the existence of a dense subset of $C_{0}(\Omega,\mathbb{C})$
satisfying $\{a\in A: \|[D,L_a]\|<+\infty \}$ it is enough to observe that
\[
\left\{
a\in C_{0}(\Omega,\mathbb{C}):
\sum_{w\in W^*}\frac{|a(wx)-a(wy)|}{|\mathbb{J}(wx)|}<+\infty
\right\}
\]
is a self-adjoint subalgebra of $C_{0}(\Omega,\mathbb{C})$,
separating points in $\Omega$ and for any $x\in\Omega$ there is an element $a$
in this subalgebra such that $a(x)\neq 0$, and therefore we can apply the Stone-Weiertrass
theorem for locally compact spaces. Indeed, the family of functions
$(S_n)_{n\in\mathbb{N}}$, given by
\[
S_n(x) \equiv \arctan(\frac{1}{x_n})\exp(-2\kappa_1 \sum_{j=1}^{n-1}x_j)
\]
is in this subalgebra, {is non-vanishing and is separating points.}
Since we are assuming that $P_{G}(f)<\infty$, it
follows immediately that $h_{\mu }(\sigma)$ is finite.
By using the main result of \cite{MR2551790} we can find a potential
$J$ having SGP and satisfying the above conditions. Therefore
{it} follows from the above theorems
and a similar computation as presented before that
\[
\textrm{Tr}_{\omega}(L_{a}D^{-1})
=
\frac{2}{h_{\mu}}\int_{\Omega}a\, d\mu,
\qquad \forall a\in C_{0}(X,\mathbb{C}).
\]

\appendix
\section{The Ruelle Operator applied to indicators of cylinders for a family of potentials}
\label{sec112}
First we consider the case of a general $f$ and a general cylinder set $[w]$.
Let $z_0 \in [010]$ and assume that $\psi$ is the indicator function $I_{[w]}$ of the cylinder $[w]$ for some finite word $w$ in the alphabet $\left\lbrace 0,1 \right\rbrace$. Then
\begin{align*}
\mathcal{L}_s^n \psi (z_0) &= \sum_{y \in \sigma^{-n}(z_0)} \left[\prod_{j=0}^{n-1} f^s \left(\sigma^j(y) \right) \right] \. I_{[w]}(y) 
= \sum_{a \in \left\lbrace 0,1 \right\rbrace^{n - |w|}} \left[ \prod_{j=0}^{n-1} f^s \left( \sigma^j(w a z_0) \right) \right]
\end{align*}
for all $n \geq |w|$.
Holding $w$ and $z_0$ fixed, we write $\mathcal{L}_s^n I_{[w]} (z_0)$ simply as $L_s^n$.

Now, for a potential $f$ as in expression (\ref{eq1}) we formally derive some expressions, whose validity will depend on the convergence of some sequences. Later we investigate in some examples the convergence issue.
\begin{theorem}
For $n > |w|$,
\begin{align*}
L_s^n = \sum_{j=1}^{n - |w|} \left( \prod_{i=0}^{j-1} C_i^s \right) \. L_s^{n-j} + \prod_{k=0}^{n-1} f^s \left( \sigma^k(w1^{n-|w|}z_0) \right) \ .
\end{align*}
\end{theorem}
\begin{proof}
\begin{align*}
& \sum_{j=1}^{n-|w|} \left( \prod_{i=0}^{j-1} C_i^s \right) L_s^{n-j} \\
&= \sum_{j=1}^{n-|w|} \left(\prod_{i=0}^{j-1}C_i^s \right) \. \sum_{a \in \left\lbrace 0,1 \right\rbrace^{n-j-|w|}} \prod_{k=0}^{n-j-1} f^s \left(\sigma^k (waz_0) \right) \\
&= \sum_{j=1}^{n-|w|} \left[ \prod_{i=0}^{j-1} f^s \left( \sigma^i(01^{j-1} z_0) \right) \right] \sum_{a \in \left\lbrace 0,1 \right\rbrace^{n-j-|w|}} \prod_{k=0}^{n-j-1} f^s \left( \sigma^k(wa01^{j-1}z_0) \right) \\
&= \sum_{j=1}^{n-|w|} \left[ \prod_{i=0}^{j-1} f^s \left( \sigma^{n-j+i}(wa01^{j-1} z_0) \right) \right] \sum_{a \in \left\lbrace 0,1 \right\rbrace^{n-j-|w|}} \prod_{k=0}^{n-j-1} f^s \left( \sigma^k(wa01^{j-1}z_0) \right) \\
&= \sum_{j=1}^{n-|w|} \sum_{a \in \left\lbrace 0,1 \right\rbrace^{n-j-|w|}} \prod_{k=0}^{n-1} f^s \left(\sigma^k(wa01^{j-1}z_0) \right)
\end{align*}
The set
\begin{align*}
\left\lbrace wa01^{j-1}z_0;\  a \in \left\lbrace 0,1 \right\rbrace^{n-j-|w|} \right\rbrace
\end{align*}
is the set of all $n$-th pre-images of $z_0$ which has $z_0$ preceded by exactly $(j-1)$ digits 1 and belongs to $[w]$. Thus,
\begin{align*}
\bigcup_{j=1^{n-|w|}} \left\lbrace wa01^{j-1}z_0;\  a \in \left\lbrace 0,1 \right\rbrace^{n-j-|w|} \right\rbrace
\end{align*}
has all the $n$-th pre images of $z_0$ in the set $[w]$, except for $w1^{n-|w|}z_0$. We have, therefore,
\begin{align*}
\sum_{j=1}^{n - |w|} \left( \prod_{i=0}^{j-1} C_i^s \right) \. L_s^{n-j} + \prod_{k=0}^{n-1} f^s \left( \sigma^k(w1^{n-|w|}z_0) \right)   & =  \\
\sum_{a \in \left\lbrace 0,1 \right\rbrace^{n-|w|}} \prod_{k=0}^{n-1} f^s \left( \sigma^k(waz_0) \right)  & = L_s^n \ .
\end{align*}
\end{proof}
This theorem has the following consequences, whose proofs are omitted.
\begin{corollary}
For $N \in \N$,
\begin{align*}
\sum_{n=0}^N L_s^{|w|+n} = L_s^{|w|} + \sum_{n=0}^{N-1} \left( \sum_{j=0}^{N-1-n} \prod_{i=0}^j C_i^s \right) L_s^{|w|+n} + \sum_{n=1}^N \prod_{k=0}^{n+|w|-1} f^s\left( \sigma^k(w1^{n}z_0) \right) \ .
\end{align*}
\label{cor1}
\end{corollary}

\begin{corollary}
\begin{align*}
\sum_{n=0}^{\infty} L_s^{|w|+n} = \frac{L_s^{|w|} + \sum_{n=1}^{\infty} \prod_{k=0}^{n+|w|-1} f^s \left( \sigma^k(w1^n z_0) \right)}{1- \sum_{j=0}^{\infty} \prod_{i=0}^j C_i^s}
\end{align*}
\label{cor2}
\end{corollary}
From Corollary \ref{cor2}, we conclude that
\begin{align}
\zeta_+(s) = \frac{L_s^{|w|} + \sum_{n=1}^{\infty} \prod_{k=0}^{n+|w|-1} f^s \left( \sigma^k(w1^n z_0) \right)}{1- \sum_{j=0}^{\infty} \prod_{i=0}^j C_i^s} + \sum_{n=1}^{|w|-1} L_s^n \ .
\label{eq2}
\end{align}
The second term in the right hand side of Equation \eqref{eq2} is a finite sum (does not have convergence problems), but we shall see in the next section that $\lim_{s \to 1^+}\zeta_+(s) = +\infty$ under the hypothesis of a normalized eigenvalue. We shall investigate the existence of the limit
\begin{align*}
\lim_{s \to 1^+} (s-1) \zeta_+(s) = \lim_{s \to 1^+} (s-1) \. \frac{L_s^{|w|} + \sum_{n=1}^{\infty} \prod_{k=0}^{n+|w|-1} f^s \left( \sigma^k(w1^n z_0) \right)}{1- \sum_{j=0}^{\infty} \prod_{i=0}^j C_i^s} \ .
\end{align*}
Here, it is worth noting that the  denominator does not depends of the cylinder $[w]$.

\subsection{Normalization of the Eigenvalue}
\label{sec2}
We are now interested in examples where the main eigenvalue of the Ruelle operator is equal to $1$ but $f$ not necessarily is normalized.
We recall that, according to Corollary 3.5 in \cite{MR2342978}, $\lambda$ is the maximal eigenvalue if and only if,
\begin{align}
\frac{1}{\lambda^2} \left[ e^{d_1} + \sum_{j=1}^\infty e^{d_{1+j}} \frac{e^{a_2 + \ldots + a_{1+j}}}{\lambda^j} \right] \left[ e^{b_1} + \sum_{j=1}^\infty e^{b_{1+j}} \frac{e^{c_2 + \ldots + c_{1+j}}}{\lambda^j} \right] = 1 \ .
\label{eq3}
\end{align}
For the potential defined in \eqref{eq1}, we have $a_n \equiv \log C_0 \equiv b_n$, $d_n \equiv  \log C_1$ and $c_{n+1} = \log C_{n+1}$. Hence, \eqref{eq3} translates to the following. 

\begin{proposition} \label{leel} If 
\begin{align}
C_0 = \left[ 1 + \sum_{j=1}^{\infty} \left( \prod_{k=1}^j C_k \right) \right]^{-1} ,
\label{eq4}
\end{align}
then, the main eigenvalue of the Ruelle operator for the potential $g$ is equal to $1$.

\end{proposition}

\begin{proof}
Assuming  that $\lambda = 1$ in  \eqref{eq3}, we get
\begin{align*}
& \left[ C_1 + \sum_{j=1}^\infty C_1 C_0^j \right] \left[ C_0 + \sum_{j=1}^\infty C_0 C_2 \ldots C_{j+1} \right]  \\
=& C_0 C_1 \left[ 1 + \sum_{j=1}^\infty C_0^j \right] \left[ 1 + \sum_{j=1}^\infty C_2 \ldots C_{j+1} \right] 
= \left( \sum_{j=1}^\infty C_0^j \right) \left[ \sum_{j=1}^{\infty} \left( \prod_{k=1}^j C_k \right)\right] \\
=& \frac{C_0}{1-C_0} \. \left[ \sum_{j=1}^{\infty} \left( \prod_{k=1}^j C_k \right)\right] = 1 .
\end{align*}
In this case $C_0$ satisfies (\ref{eq4})
\end{proof}

Therefore, we now investigate potentials given by $(C_n)_{n \in \N}$, for which $\sum_{j=1}^{\infty} \prod_{k=1}^j C_k$ converges and for which the restriction to the cylinder $[0]$ is equal to $C_0$ as in \eqref{eq4}.

Define $\gamma: [1,+\infty) \to \R$ to be the second term in the denominator of $\zeta_+(s)$ in \eqref{eq2}, that is
\begin{equation} \label{treta}
\gamma(s) := \sum_{j=0}^{\infty} \prod_{i=0}^j C_i^s = C_0^s \left[ 1 + \sum_{j=1}^{\infty} \prod_{i=1}^j C_i^s \right].
\end{equation}
If we suppose that the eigenvalue is normalized, then $\gamma(1) = 1$, by \eqref{eq4}, and, by \eqref{eq2}, we see that $\lim_{s \to 1^+} \zeta_+(s) = + \infty$, which makes the question of the existence of the limit $\lim_{s \to 1^+}(s-1) \zeta_+(s)$ a non trivial one. Let $z(s)$ be the numerator of the first term in $\zeta_+(s)$, for $s \geq 1$. Therefore,
\begin{equation}\label{lor}
\lim_{s \to 1^+} (s-1) \zeta_+(s) = \lim_{s \to 1^+} \frac{z(s)}{(s-1)^{-1} \left( \gamma(1) - \gamma(s) \right)} \ .
\end{equation}
As $\gamma(1)=1$, 
we conclude that, under the above hypothesis, if $\gamma$ has right derivative at $s=1$ and $z(1)$ is finite, then the above limit is finite.

\begin{remark}
Note that the term on the numerator depends on the cylinder $[w]$ and the term on the denominator does not.
\end{remark}

\medskip

\subsection{An example with Walters regularity}
\label{sec3}
We consider the potential $g = \log f: \Omega =\{0,1\}^\mathbb{N} \to \mathbb{R}$ of the following kind. Let
\begin{align*}
   t = (e^{\frac{1}{3}} + e^{-\frac{1}{2}})^{-1} \ ,
\end{align*}
   and, for $n \in \mathbb{N}$, define $C_n = t \, \exp \left[ \frac{(-1)^{n+1}}{n} \right]$, which satisfies $\lim_n C_n = t$. $f$ is then defined by \eqref{eq1}. We would like to point out that P. Walters mentioned in a private message a version of the above example (a potential which is in the Walters class but it is not of H\"older continuous).
   
   We will now show (see Theorem \ref{luc}) that in this case, the limit
   \begin{equation} \label{kle} \lim_{s \to 1^+} (s-1) \zeta_+(s) = - \frac{z(1)}{\gamma'(1)}.
   \end{equation}
   exists. As stated in \cite{MR2342978},  a potential in the Walters family is in Walters class  of regularity if and only if, for the associated sequences $(a_n)_{n \in \N}$ and $(c_n)_{n \in \N}$ given by Definition \ref{defi1} with respective limits $a$ and $c$, the following holds. 
   \begin{align*}
   \sum_{n \in \N} (a_n - a) \ \text{and} \ \sum_{n \in \N} (c_n - c) \ \text{converge} \ .
   \end{align*}
   For the potential defined in the first paragraph of this section, $a_n \equiv \log C_0$ while $c_n = \log C_n = \log t + \frac{(-1)^{n+1}}{n}$, so
   \begin{align*}
   \sum_{n \in \N} (a_n - a) = 0 \ ,
   \end{align*}
   and
   \begin{align*}
   \sum_{n \in \N} (c_n - c) = \sum_{n \in \N} \left[ \log t + \frac{(-1)^{n+1}}{n} - \log t \right] = \sum_{n \in \N} \frac{(-1)^{n+1}}{n} \ ,
   \end{align*}
   which is convergent.  Therefore, $g = \log f$ is of Walters class of regularity. 
   Note that a potential $g = \log f$, for $f$ as in \eqref{eq1}, cannot be normalized, unless the sequence $C_n$ is constant (which is equivalent to a constant potential), as
\begin{align}
\mathcal{L}_{g} \mathbf{1} (x) = e^{g(0x)} + e^{g(1x)} = C_0 + f(1x) = 1
\label{eq6}
\end{align}
for every $x \in \Omega$. 

We will show that this potential is not  H\"older continuous. In order to do so, define, for $\phi \in C(\Omega, \mathbb{R})$,
\begin{align*}
v_n(\phi) = \sup \left\lbrace |\phi(x) - \phi(y)|;\ x,y \in \Omega,\ x_i = y_i, i  = 0, \ldots, n-1  \right\rbrace \ .
\end{align*}
A function $\phi$ is H\"older with H\"older exponent $\alpha \in (0,1)$ if
\begin{align*}
\sup_{n \in \mathbb{N}} \frac{v_n(\phi)}{2^{-n \alpha}} < \infty \ .
\end{align*}
Given $n \geq 2$, let $x = 1^n0x$, for any fixed $x \in \Omega$, and $y = 1^\infty$. Then $d(x,y) = 2^{-(n+1)}$, and
$|\log f(x) - \log f(y)| = {1}/{n}$. This implies that $ v_n(\log f) \geq \frac{1}{n}$ and, in particular, 
This implies that $\log f$ is not H\"older continuous.

We now determine the derivative of $\gamma$ defined in \eqref{treta}. In order to so, let 
\begin{align*}
\gamma_n(s) = C_0^s \left[ \sum_{j=1}^n \left( \prod_{k=1}^j C_k^s \right) + 1 \right]
\end{align*}
Then
\begin{align*}
  \gamma_n'(s)   = & \log C_0 \. C_0^s \left[ \sum_{j=1}^n \left( \prod_{k=1}^j C_k \right)^s + 1 \right] + C_0^s 
\left[ \sum_{j=1}^n \left( \prod_{k=1}^j C_k \right)^s \. \log \left( \prod_{k=1}^j C_k^s \right) \right] \\
 = &
 C_0^s \left\lbrace \log C_0 \left[ 1 + \sum_{j=1}^n t^{js} e^{s \. \sum_{k=1}^j \frac{(-1)^{k+1}}{k}} \right] \right. 
 \\ & + \left. \sum_{j=1}^n \left[ t^{js} e^{s \. \sum_{k=1}^j \frac{(-1)^{k+1}}{k}} \. j \. \log t \.  \sum_{k=1}^j \frac{(-1)^{k+1}}{k} \right] \right\rbrace
\end{align*}
We will now address the question of differentiability of  the function $\gamma$.
\begin{theorem} \label{gaga} The derivative $\gamma ' (1)$ exists and is finite. \end{theorem}
\begin{proof}
Remember that $t<1$. Let
\begin{align*}
\rho = \sup_{n \in \N} \sum_{i=1}^n \frac{(-1)^{i+1}}{i} \ .
\end{align*}
Then, for $\epsilon > 0$ and $s \in [1,1+\epsilon]$,
\begin{align*}
 & C_0^s \log C_0 \left[ \sum_{j=1}^n t^{js} e^{s \. \sum_{k=1}^j \frac{(-1)^{k+1}}{k}} \right] 
\\  \leq &
  \log C_0 \. \max \{ C_0,C_0^{1+\epsilon}\} \. \max \{e^\rho,e^{\rho (1+\epsilon)} \} \. \sum_{j=1}^n \left(t^s \right)^j
\\  \leq &
\log C_0 \. \max \{ C_0,C_0^{1+\epsilon}\} \. \max \{e^\rho,e^{\rho (1+\epsilon)} \} \. \sum_{j=1}^n  t^j 
\end{align*}
and
\begin{align*}
& C_0^s \sum_{j=1}^n \left[ t^{js} e^{s \. \sum_{k=1}^j \frac{(-1)^{k+1}}{k}} \. j \. \log t \.  \sum_{k=1}^j \frac{(-1)^{k+1}}{k} \right] 
\\  \leq & 
\max \{C_0,C_0^{1+\epsilon}\} \. \lambda \. \max \{e^{\rho},e^{\rho (1+\epsilon)}\} \. \sum_{j=1}^n j \. t^j \ .
\end{align*}
By the Weierstrass $M$ criterion, the series of functions $\gamma_n': [1,1+\epsilon] \to \R$ converges uniformly to the function
\begin{align*}
 s \mapsto &   C_0^s \left\lbrace \log C_0 \left[ 1 + \sum_{j=1}^\infty t^{js} e^{s \. \sum_{k=1}^j \frac{(-1)^{k+1}}{k}} \right] \right. \\&+ \left. \sum_{j=1}^\infty \left[ t^{js} e^{s \. \sum_{k=1}^j \frac{(-1)^{k+1}}{k}} \. j \. \log t \.  \sum_{k=1}^j \frac{(-1)^{k+1}}{k} \right] \right\rbrace \ .
\end{align*}
Hence, the function  $\gamma':[1,1+\epsilon] \to \R$ is well defined. 
\end{proof}
It follows from the above  that expression (\ref{kle}) is true if $z(1)$ is finite.
Now we will investigate if $z(1)$ is finite for the case of a general cylinder $[w]$.
Considering the indicator function $I_{[w]}$ of the cylinder $[w]$, we obtain from the results in Section \ref{sec2}, that
\begin{align*}
z(s) = L_s^{|w|} + \sum_{n=1}^{\infty} \prod_{k=0}^{n+|w|-1} f^s \left( \sigma^k(w1^n z) \right),
\end{align*}
where it remains to analyze the convergence as $s \to 1$. 

\begin{theorem} \label{luc}
$z(1)$ is finite  in the case of the existence of the limit 
\begin{equation} \label{sili} \lim_{s \to 1}\sum_{k=0}^{\infty} \prod_{i=0}^k C_i^s.
\end{equation}
Therefore, in this case, 
$\lim_{s \to 1^+} (s-1) \zeta_+(s) = - \frac{z(1)}{\gamma'(1)}$ exists. Note that condition (\ref{sili}) does not depend on the cylinder set $[w]$.
\end{theorem}

\begin{proof}
When $w = u0$ (it is possible for $u$ to be the empty word),
\begin{align*}
\prod_{k=0}^{n+|w|-1} f^s \left( \sigma^k(w1^n z) \right) &= \prod_{k=0}^{|w|-1}f^s \left(\sigma^k(u0\ldots) \right) \prod_{k=0}^{n-1} f^s\left( \sigma^k(1^n z) \right) \\
&= \prod_{k=0}^{|w|-1}f^s \left(\sigma^k(u0\ldots) \right) \prod_{k=0}^{n-1} C_{n-k}^s \ ,
\end{align*}
where $u0\ldots$  can be any point in $\Omega$ in the cylinder $[u0]$. Thus,
\begin{align*}
z(s) = L_s^{|w|} +  \prod_{k=0}^{|w|-1}f^s \left( \sigma^k (u0\ldots) \right) \sum_{n=1}^{\infty} \prod_{k=0}^{n-1} C_{n-k}^s
\end{align*}
When $w = u1^l$, with $u$ empty or having 0 as the last digit,
\begin{align*}
\prod_{k=0}^{n+|w|-1} f^s \left( \sigma^k(w1^n z) \right) &= \prod_{k=0}^{|w|-l}f^s \left(\sigma^k(u\ldots) \right) \prod_{k=0}^{n+l-1} f^s\left( \sigma^k(1^{n+l} z) \right) \\
&= \prod_{k=0}^{|w|-1}f^s \left(\sigma^k(u\ldots) \right) \prod_{k=0}^{n+l-1} C_{n+l-k}^s \ ,
\end{align*}
where $u\ldots$ can be any point in $\Omega$ in the cylinder $[u]$. Thus,
\begin{align*}
z(s) = L_s^{|w|} +  \prod_{k=0}^{|w|-l-1}f^s \left( \sigma^k (u\ldots) \right) \sum_{n=1}^{\infty} \prod_{k=0}^{n-1} C_{n+l-k}^s
\end{align*}
In both cases, convergence of  $\lim_{s \to 1}z(s)$ is guaranteed by the existence of $\sum_{k=0}^{\infty} \prod_{i=0}^k C_i^s$.
\end{proof}

\subsection{A more general family of examples}
\label{subsec:Lambda}
The most general situation where the calculations of the preceding section are valid is in the case of uniform convergence of the sequence of functions
\begin{align*}
[1,1+\epsilon] \ni s &\mapsto \sum_{j=1}^n \left( \prod_{k=1}^j C_k \right)^s \\
[1,1+\epsilon] \ni s &\mapsto \sum_{j=1}^n \left( \prod_{k=1}^j C_k \right)^s \. \log \left( \prod_{k=1}^j C_k \right)
\end{align*}
Nevertheless, the second condition is implied by the first, for
\begin{align*}
& \sum_{j=1}^\infty \left[ \left( \prod_{k=1}^j C_k \right)^s \. \log \left( \prod_{k=1}^j C_k \right) \right] - \sum_{j=1}^n \left[ \left( \prod_{k=1}^j C_k \right)^s \. \log \left( \prod_{k=1}^j C_k \right) \right] \\
= &  \sum_{j=n+1}^\infty \left( \prod_{k=1}^j C_k \right)^s \. \log \left( \prod_{k=1}^j C_k \right) \ ,
\end{align*}
and the first convergence conditions bounds $\log \prod_{k=1}^j C_k$.

Trying to generalize the calculations of the preceding section, we present a class $\Lambda$ of potentials $f$ for which they can be extended.  Choose a sequence $\left( \alpha_k \right)_{k \in \N}$, with the property that 
$\sum_{k \in \N} \alpha_k$ converges and fix $t \in (0,1)$. Then define $C_k := t e^{\alpha_k}$, for $k \in \N$, and define $C_0$ by \eqref{eq4}, so that the eigenvalue is normalized. The potential $f$ associated to those choices is then given by \eqref{eq1}. We denote by $\Lambda$ the class of such potentials. We observe that $\Lambda$ is a family indexed by the set
\begin{align*}
\left\lbrace \left( (\alpha_k)_{k \in \N},t \right);\ \sum_{k \in \N} \alpha_k \ \text{converges} \ \& \ t \in (0,1) \right\rbrace \ .
\end{align*}
Using the notation of the preceding sections,
\begin{align*}
\gamma_n'(s) = C_0^s \left\lbrace \log C_0 \left[ 1 + \sum_{j=1}^n t^{js} e^{s \. \sum_{k=1}^j \alpha_k} \right] + \sum_{j=1}^n \left[ t^{js} e^{s \. \sum_{k=1}^j \alpha_k} \. j \. \log t \.  \sum_{k=1}^j \alpha_k \right] \right\rbrace  \ ,
\end{align*}
the same arguments used there to prove the uniform convergence of $\left(\gamma_n' \right)$ can be used again to show that $\gamma ' (1)$ is finite (see Theorem \ref{gaga}). 
Indeed, under the hypothesis of convergence of the series $\sum_{k \in \N} \alpha_k$, we can define
\begin{align*}
\rho = \sup_{n \in \N} \sum_{k=1}^n \alpha_k \ ,
\end{align*}
and get the same estimates as in the preceding section. In the same way as before, if condition (\ref{sili}) is true, then the convergence of $\lim_{s \to 1^+} (s-1) \zeta_+(s)$ is guaranteed. Hence, we have shown the following. 

\begin{theorem}
For the potentials $g = \log f$, with $f \in \Lambda$, the limit
\begin{equation}
\lim_{s \to 1^+} (s-1) \zeta_+(s)
\label{ccc}
\end{equation}
exists.
\end{theorem}

\section{Appendix B: A Haar basis for \texorpdfstring{$L^2$}{L2} spaces and explicit calculations of eigenfunctions} \label{Haar}

We now briefly discuss the construction of spectral triples in \cite{MR3084488} which we recall now. Let $(\Omega,T)$ refer to a topologically mixing subshift of finite type (cf. Section \ref{subsec:sft}), $\varphi: \Omega \to \mathbb{R}$ a Hölder continuous      
potential, $\mu$ to the associated equilibrium state as in Corollary \ref{cor:decay-to-invariante-measure} and 
$L^2(\mu)$ to the Hilbert space of complex valued and $\mu$-square integrable functions. Furthermore, for $a \in C(\Omega)$, set $L_a :  L^2(\mu) \to L^2(\mu)$, $f \mapsto a f$ and note that $a \mapsto L_a$ defines a faithful representation of the $C^\ast$-algebra $C(\Omega)$. Hence, in order to obtain a spectral triple, it remains to construct  a suitable Dirac operator $D$. In order to do so, the authors construct a Haar basis of $L^2(\mu)$, then use this basis in order to define $D$ on the dense subset of finite linear combinations of this subset (cf. Equation (13) in \cite{MR3084488}) and finally show that $(C(\Omega),L^2(\mu),D)$ is a spectral triple (Theorem 4.1 in in \cite{MR3084488}) and that the associated spectral metric is a metric.  

We now recall their construction for the particular case of $\Omega := \left\lbrace 0,1 \right\rbrace^\N$ in order to have the most simple example at hand. In here, we refer to $W^\ast = \bigcup_{n\geq 1}\left\lbrace 0,1 \right\rbrace^n$ as the set of finite words. Moreover, for each $w \in W^\ast$, define the inner product 
\[\langle (x_0,x_1),(y_0,y_1) \rangle_w  :=  \mu([w0]) x_0y_0 +  \mu([w1]) x_1y_1  \]
on $\R^2$ and set $f_{w,0} := \mu([w0])^{-1/2}(1,0)$ and $f_{w,1} := \mu([w1])^{-1/2}(0,1)$. Then, $\{f_{w,i}:i=0,1\}$ is an orthonormal basis  with respect to  $\langle \cdot, \cdot \rangle_w$. As we are in dimension 2 and $\|\mu([w])^{-1}(1,1)\|_w =1$, there exists a unique matrix $A_\omega \in \textrm{GL}_2(\R)$ such that the  determinant of $A_w$ is positive, $A_w$ is an isometry with respect to $\langle \cdot, \cdot \rangle_w$ and $A_w f_{w,1} =  \mu([w])^{-1/2}(1,1)$. Namely, as it easily can be verified, 
\[ A_w = \sqrt{\tfrac{\mu([w1])}{\mu([w])}} \begin{pmatrix} 1 & 1 \\ -\frac{\mu([w0])}{\mu([w1])} & 1 \end{pmatrix}.\]
Moreover, the $A_w$-image of $f_{w,0}$ is  
\[ 
A_w(f_{w,0}) = \sqrt{\tfrac{\mu([w1])}{\mu([w])\mu([w0])}}  \begin{pmatrix} 1 & 1 \\ -\frac{\mu([w0])}{\mu([w1])} & 1 \end{pmatrix} \begin{pmatrix} 1 \\0   \end{pmatrix} = 
 \mu([w])^{-1/2} \begin{pmatrix} \sqrt{  {\mu([w1])}/{\mu([w0])} }  \\ - \sqrt{  {\mu([w0])}/{\mu([w1])}  }    \end{pmatrix}.
\] 
From the above, following the reasoning in \cite{MR3084488}, we set     
\begin{align} \nonumber e_w & := \frac{\langle f_{w,0}, A_w(f_{w,0}) \rangle_w}{\sqrt{\mu([w0])}} \mathbf{1}_{[w0]} + \frac{\langle f_{w,1}, A_w(f_{w,0}) \rangle_w}{\sqrt{\mu([w1])}} \mathbf{1}_{[w1]}\\
\label{eq5}
& = \frac{1}{\sqrt{\mu([w])}} \left( \sqrt{  \tfrac{\mu([w1])}{\mu([w0])}} \mathbf{1}_{[w0]} -  \sqrt{  \tfrac{\mu([w0])}{\mu([w1])}} \mathbf{1}_{[w1]} \right).
\end{align}
We are now in position to apply Theorem 3.5 in \cite{MR3084488}. That is, one obtains that 
\begin{align*}
\mathbb{B} := \left\lbrace e_w;\ w \in \textstyle W^\ast \right\rbrace \cup \left\lbrace  \mu([0])^{-\frac{1}{2}} \mathbf{1}_{[0]} \ ,  \   \mu([1])^{-\frac{1}{2}} \mathbf{1}_{[1]} \right\rbrace  
\end{align*} 
is a Haar basis of $L^2(\mu)$. Furthermore, as the proof of Theorem 3.5 only makes use of the property that $\mu(\Omega) =1$, this in fact holds for any probability measure $\mu$. As an immediate consequence of orthonormality of $\mathbb{B}$, it  follows that the following operator $D$ is well defined for any finite linear combination $f$ of  elements in $\mathbb{B}$ by   
\begin{align} \label{nhec2}
 D(f) :=  
   \frac{\langle f,\mathbf{1}_{[0]} \rangle }{\mu([0])} \mathbf{1}_{[0]}  + 
   \frac{\langle f,\mathbf{1}_{[1]} \rangle}{\mu([1])}  \mathbf{1}_{[1]} 
 - \langle f,\mathbf{1} \rangle\mathbf{1}  +
  \sum_{w \in W^\ast} \frac{\langle f, e_w \rangle} {\mu([w]) }   e_{w}.
\end{align} 
As $\mathbb{B}$ is a basis of $L^2(\mu)$, it follows that $D$ is densely defined. 
Now assume that $\varphi: \Omega \to \mathbb{R}$ is $\omega_{0,\beta \textrm{log}}$-Hölder continuous for some $\beta > 1$ and that $\mu$ is given by Corollary \ref{cor:decay-to-invariante-measure}. Moreover, as the proof of Theorem 4.1 in \cite{MR3084488} only makes use of $\mu([w]) \to 0$ as the length of $w$ tends to infinity and the distortion estimate in \eqref{eq:distortion-expanding}, we obtain the following partial generalisation to not necessarily Hölder continuous potentials.  
\begin{corollary} If $\varphi: \Omega \to \mathbb{R}$ is $\omega_{0,\beta \textrm{log}}$-Hölder continuous for some $\beta > 1$
and $\mu$ is given by Corollary \ref{cor:decay-to-invariante-measure}, then $(C(\Omega), {L}^2 (\mu),D)$ is a spectral triple.
\end{corollary}
  
However, in order to have an explicit expression for $D$ {it is necessary} to compute the values $\mu([w])$ for any cylinder set  $[w]$ in $\Omega$. The main motivation of this appendix is in fact to provide these expressions for a class of normalized potentials in the Walters family. Before we do so, we give two simple examples.
 
\begin{example}  \label{MaMa} Assume that $P=(P_{i,j})_{i,j=0,1}$ is a stochastic matrix and that $\pi \in \R^2$ is a left invariant probability vector of $P$. It is then well known that 
\[\mu([w_0\ldots w_n]) := \pi(w_0) P_{w_0w_1} \cdots  P_{w_{n-1}w_n}\]
defines a $T$-invariant probability measure. Recall that a measure of this type also is known as the Markov measure associated with $P, \pi$.  In this case, $\mathbb{B}$ is the union of $\pi(0)^{-1/2} \mathbf{1}_{[0]}$, $\pi(1)^{-1/2} \mathbf{1}_{[1]}$ and, for $w = (w_0\ldots w_n) \in W^\ast$, 
\begin{align}
    e_x  =\frac{1}{ \sqrt{\mu([w])}}   \left( \sqrt{ \tfrac{P_{x_n,1}}{P_{x_n,0}}} \mathbf{1}_{[x0]}
    - \sqrt{\tfrac{P_{x_n,0}}{P_{x_n,1}}} \mathbf{1}_{[x1]}
    \right).
    \label{eq52}
    \end{align}
\end{example}

\begin{example} \label{ind} For the case of the measure $\mu$ of maximal entropy, the above simplifies to, with $|w|$ referring to the length of $w$,  
\begin{align}
   \mathbb{B} = \left\{ \sqrt{2}^{-1} \mathbf{1}_{[0]}, \sqrt{2}^{-1} \mathbf{1}_{[1]},   \right\} 
   \cup \left\{ \sqrt{2}^{-|w|} (\mathbf{1}_{[w0]} -  \mathbf{1}_{[w1]} ): w \in W^\ast   \right\} .    \label{eq235}
\end{align}
\end{example}

\subsection{Explicit computations for the equilibrium probability for potentials on the Walters family} \label{expWal}

Our purpose here is to describe how one can get explicit expressions for \eqref{eq5} which are necessary for defining the momentum operator $D$ explicitly in the case of potentials on the Walters family.
    
Firstly, we will state a theorem of \cite{MR2342978} which we will apply. For a potential $g$ in the Walters family, which satisfies the hypotheses of Ruelle's Theorem, $\lambda$ will refer to the maximal positive eigenvalue and $h: \Omega \to \R$ to the corresponding positive eigenfunction. Furthermore, $\hat{\phi}: \Omega \to \R$ will denote the exponential of the normalized potential, that is,  
\begin{align*}
\hat{\phi} = \frac{h \. e^g}{ \lambda h \circ \sigma} \ .
\end{align*} 
The probability measures $\mu$ and $\nu$ are, respectively, the eigenfunction of the dual of the Ruelle Operator and the equilibrium state of $g$. Note that they are related through the Radon-Nikodym derivative by $d\mu = h^{-1} d\nu$. In here, we are mainly interested in potentials of the form $g = \log f$  with $f \in \Lambda$ and $\Lambda$ as constructed in \ref{subsec:Lambda}. Observe that in this case, $\lambda =1$. 

\begin{theorem}
For a potential $g$ of Walters type, which is determined by the convergent sequences $(a_n)_{n \in \N}$, $(b_n)_{n \in \N}$, $(c_n)_{n \in \N}$ and $(d_n)_{n \in \N}$, the eigenfunction $h: \Omega \to \R$ of Ruelle's operator is determined by the following equations:
\begin{align*}
    h(0^n1z) = \alpha_n = \frac{\alpha(\lambda - e^a)}{\lambda e^d} \left[ e^{d_n} + \sum_{j=1}^\infty \frac{e^{d_{j+n}}}{\lambda^j} \exp \left( \sum_{i=1}^j a_{n+i} \right) \right] \ ,
    \end{align*}
    \begin{align*}
    h(1^\infty) = \beta = \frac{\alpha e^b(\lambda - e^a)}{e^d(\lambda - e^c)\lambda}\left[ e^{d_1} + \sum_{j=1}^\infty \frac{e^{d_{j+1}}}{\lambda^j} \exp \left( \sum_{i=1}^j a_{i+1} \right) \right] \ ,
    \end{align*}
    \begin{align*}
    h(1^n0z) = \beta_n = \frac{\beta(\lambda - e^c)}{\lambda e^b} \left[ e^{b_n} + \sum_{j=1}^\infty \frac{e^{b_{j+n}}}{\lambda^j} \exp \left( \sum_{i=1}^j c_{n+i} \right) \right] \ ,
    \end{align*}
	for all $n \in \N$ and all $x \in \Omega$. The free positive parameter $\alpha$ is chosen so that $\mu$ is a probability measure.
	\label{theo1}
	\end{theorem}
	
For $g = \log f$, with $f \in \Lambda$, we derive the expressions for $h$:
	\begin{align*}
    h(0^n1z) & = \alpha_n 
    = \frac{\alpha(\lambda - e^a)}{\lambda e^d} \left[ e^{d_n} + \sum_{j=1}^\infty \frac{e^{d_{j+n}}}{\lambda^j} \exp \left( \sum_{i=1}^j a_{n+i} \right) \right] \\
    & = 
    \frac{\alpha (1-C_0)}{C_1} \left( C_1 + \sum_{j=1}^\infty C_1 C_0^j \right) 
    = \alpha (1-C_0) \left( 1+ \frac{C_0}{1-C_0} \right) 
    = \alpha \ .
    \end{align*}
Moreover,
   \begin{align*}  h(1^\infty) & = \beta 
    = \frac{\alpha e^b(\lambda - e^a)}{e^d(\lambda - e^c)\lambda}\left[ e^{d_1} + \sum_{j=1}^\infty \frac{e^{d_{j+1}}}{\lambda^j} \exp \left( \sum_{i=1}^j a_{i+1} \right) \right] \\ &=
    \frac{\alpha C_0 (1-C_0)}{C_1(1-t)} \left( C_1 + \sum_{j=1}^\infty C_1 C_0^j \right) 
    = \frac{\alpha C_0}{1-t} ,\end{align*} 
    and
 \begin{align*} h(1^n0z) & = \beta_n 
    = \frac{\beta(\lambda - e^c)}{\lambda e^b} \left[ e^{b_n} + \sum_{j=1}^\infty \frac{e^{b_{j+n}}}{\lambda^j} \exp \left( \sum_{i=1}^j c_{n+i} \right) \right] 
    \\ &=
    \frac{\beta(1-t)}{C_0} \left[ C_0 + \sum_{j=1}^\infty C_0 \left( \prod_{i=1}^j C_{n+i} \right) \right] 
    = \beta (1-t) \left[ 1 + \sum_{j=1}^\infty \left( \prod_{i=1}^j C_{n+i} \right) \right] \\
    & =
  \alpha C_0 \left[ 1 + \sum_{j=1}^\infty \left( \prod_{i=1}^j C_{n+i} \right) \right] 
    = \alpha \left( \frac{1+ \sum_{j=1}^\infty \prod_{i=1}^j C_{n+j}}{1+ \sum_{j=1}^\infty \prod_{i=1}^j C_i} \right).
  \end{align*}
Note  that $\lim_{n \to \infty} \beta_n = \beta$. Hence, 
\begin{align*}
    \lim_n \beta_n = \beta(1-t) \left[ 1 + \lim_n \sum_{j=1}^\infty \prod_{i=1}^j C_{n+j} \right]
    = \beta(1-t) \left[ 1 + \sum_{j=1}^\infty t^j \right] 
    = \beta.
   \end{align*} 
From the above equations, it is also possible to derive explicitly the function $\hat{\phi}$:
\begin{align*}    
    \hat{\phi} (0^{n+1} 1 z) & = C_0 := \gamma_{n+1}, \quad 
    \hat{\phi} (01^n0z) = \frac{\alpha C_0}{\beta_n} \\
    \hat{\phi} ( 1^{n+1} 0z) & = \frac{\beta_{n+1} \. C_{n+1}}{\beta_n} := \delta_{n+1}, \quad 
    \hat{\phi} (10^n1z) = \frac{\beta_1 \. C_1}{\alpha},\end{align*}
    and the normalization condition can be explicitly verified:
\begin{align*}    
\hat{\phi}(10^n1z) & = \frac{\beta_1 \. C_1}{\alpha}   = \frac{C_1}{\alpha} \alpha C_0 \left[ 1 + \sum_{j=1} \infty \prod_{i=1}^j C_{1+i} \right] 
\\ & =
  C_0 \left[ \sum_{j=1}^\infty \prod_{i=1}^j C_i \right] = 1 - C_0 = 1 - \gamma_{n+1}.\end{align*}
Moreover,
\begin{align*}    
    \hat{\phi}(01^n0z) = & \frac{\alpha \. C_0}{\beta_n} = \frac{\alpha C_0}{\alpha C_0 \left[ 1 + \sum_{j=1}^\infty \prod_{i=1}^j C_{n+i} \right]} = \frac{1}{1+\sum_{j=1}^\infty \prod_{i=1}^j C_{n+ i}} = 1- x \\
    & \Updownarrow \\
    x = & 1 - \frac{1}{1 + \sum_{j=1}^\infty \prod_{i=1}^j C_{n+1}} = \frac{\sum_{j=1}^\infty \prod_{i=1}^j C_{n+i}}{1 + \sum_{j=1}^\infty \prod_{i=1}^j C_{n+i} } \\
    = & 
    \frac{\alpha C_0 \, C_{n+1} \left[ 1+ \sum_{j=1}^\infty \prod_{i=1}^j C_{n+1+i} \right]}{\alpha C_0 \left[ 1 + \sum_{j=1}^\infty \prod_{i=1}^j C_{n+i} \right] } = C_{n+1} \frac{\beta_{n+1}}{\beta_n} 
   = \delta_{n+1} . \end{align*}
The equilibrium measure $\nu$, which is the Ruelle Operator's eigenmeasure for the potential $\log \hat{\phi}$, instead of $g$, satisfies the following equations:
\begin{theorem}
For each $k \in \N$, we define
\[
\begin{matrix*}[l]
i_k = \prod_{j=2}^k \gamma_j,  
& c_k = (1-\gamma_{k+1}), \\
 d_k = (1-\delta_{k+1}),  \prod_{j=2}^k \delta_j, 
& f_k = \prod_{j=2}^k \delta_j, \\
\Gamma_k = \sum_{i=0}^{\infty} \left( \prod_{j=0}^i \gamma_{k+i} \right), 
& \Delta_k = \sum_{i=0}^\infty \left( \prod_{j=0}^i \delta_{k+i} \right) 
\end{matrix*}
\]
and write $\Theta = \Gamma_2 + \Delta_2 + 2.$ The unique equilibrium state $\nu$ for $g$ is determined by the equations
\[
\begin{matrix*}[l]
 \nu([01]) = \nu([10]) = \frac{1}{\Theta}, & \nu([00]) = \frac{\Gamma_2}{\Theta}, \quad \quad \nu([11]) = \frac{\Delta_2}{\Theta}, \\
 \nu([0^{n+2}]) = \frac{\left( \prod_{j=2}^{n+1} \gamma_j \right) \. \Gamma_{n+2}}{\Theta}, 
 & \nu([1^{n+2}]) = \frac{\left( \prod_{j=2}^{n+1} \delta_j \right) \. \Delta_{n+2}}{\Theta} . 
\end{matrix*}
\]
		Moreover, for $r \in \N$, $i \in \N_r$, $k_i,l_i \in \N$,
		\begin{align*}
		\nu([0^{k_1}1^{l_1}0^{k_2}1^{l_2}\ldots0^{k_r}1^{l_r}]) &= \frac{i_{k_1} d_{l_1} \left( \prod_{j=2}^{r-1}c_{k_j}d_{l_j} \right) c_{k_r} f_{l_r}}{\Theta} \ , \\
		\nu([0^{k_1}1^{l_1}0^{k_2}1^{l_2}\ldots1^{l_{r-1}}0^{k_r}]) &= \frac{i_{k_1} \left( \prod_{j=1}^{r-2} d_{l_j} c_{k_j} \right) d_{l_{r-1}}(r) i_{k+r}}{\Theta} \ ;
		\end{align*}
		if, in the two last equations, we change zeros by ones and ones by zeros in the left sides, the right sides changes by switching $f's$ by $i's$, $i's$ by $f's$, $c's$ by $d's$ and $d's$ by $c's$.
		\label{sec3theo1}
	\end{theorem}

From the above we get
$$  i_k = C_0^{k-1} \, ,
    c_k = (1-C_0) C_0^{k-1} \, ,
    d_k =  \left( 1 - \frac{\beta_{k+1}C_{k+1}}{\beta_k} \right) \frac{\beta_k}{\beta_1} \prod_{j=2}^k C_j \, ,
    f_k =  \frac{\beta_k}{\beta_1} \prod_{j=2}^k C_j .$$
With all this information we can  determine the $\nu$-measure or the $\mu$-measure of each cylinder set $[x]$. For $\nu$, \ref{sec3theo1} and the above identities are enough, whereas for
 $\mu$-measures, it is necessary to use the relation $d\mu = h^{-1} \ d\nu$. For example,
    \begin{align*}
\mu([00]) = \sum_{k=2}^\infty \mu[0^k1] = \sum_{k=2}^\infty \int 1_{[0^k1]} h^{-1} \ d\nu = \sum_{k=2}^\infty \frac{\nu[0^k1]}{h(0^k1z)} = \sum_{k=2}^\infty \frac{\prod_{j=2}^k \gamma_j}{h(0^k1z) \left( \Gamma_2 + \Delta_2 + 2 \right)} \ ,
\end{align*}
and then, it follows
\begin{align*}
\mu([00]) &= \frac{1}{\Theta} \sum_{k=2}^\infty \frac{\prod_{j=2}^k \gamma_ j}{h(0^k1z)} = \frac{1}{\alpha \. \Theta} \sum_{k=2}^\infty C_0^{k-1} = \frac{C_0}{\alpha \Theta (1-C_0)} \ .
\end{align*}
Or, another example,
\begin{align*}
& \mu( [0^{k_1} 1^{l_1} \ldots 0^{k_r} 1^{l_r}0] ) \\
= &  \int_{[0^{k_1} 1^{l_1} \ldots 0^{k_r} 1^{l_r}0]} h^{-1} \ d\nu 
= h^{-1}(0^{k_1}1z) \nu [0^{k_1} 1^{l_1} \ldots 0^{k_r} 1^{l_r}0] 
\\  = & 
h^{-1}(0^{k_1} 1z) \frac{i_{k_1} \left( \prod_{j=1}^{r-1}d_{l_j}c_{k_{j+1}} \right) d_{l_r}}{\Gamma_2  + \Delta_2 + 2}
\\ = &  
\frac{C_0^{k_1-1}}{\alpha \Theta} \prod_{j=1}^{r-1} \left[ \left( 1 - \frac{\beta_{l_j+1} C_{l_j+1}}{\beta_{l_j}} \right) \frac{\beta_{l_j}}{\beta_1} \left( \prod_{i=2}^{l_j}C_i \right)(1-C_0) C_0^{k_j} \right] \left( 1-  \frac{\beta_{l_r+1}C_{l_r+1}}{\beta_{l_r}} \right) \frac{\beta_{l_r}}{\beta_1} \prod_{j=2}^{l_r} C_j.
\end{align*}
The computations above describe a general method for getting the probabilities of several kinds of cylinder sets.
%
%
%
\subsection*{Acknowledgements}
L. Cioletti and A. O. Lopes partially supported by Conselho Nacional de Desenvolvimento Científico e Tecnológico, and M. Stadlbauer acknowledges financial support from Fundação Carlos Chagas Filho de Amparo à Pesquisa do Estado do Rio de Janeiro (FAPERJ) through grant E-26/210.388/2019. Furthermore, this study was financed in part by the Coordenação de Aperfeiçoamento de Pessoal de Nível Superior - Brasil (CAPES) - Finance Code 001 through a visiting grant of M. Stadlbauer of the PrInt program.

%

%
%


\end{document}